\documentclass[12pt]{amsart} 
\usepackage{amssymb}
\usepackage{bbm}
\usepackage{amsmath,bm}
\usepackage{array}
\usepackage{graphics}
\usepackage{lettrine}
\usepackage{color}
\usepackage[color,matrix,arrow]{xy}
\usepackage{xcolor}
\usepackage{xypic}
\usepackage{tikz}
\usepackage[utf8]{inputenc}   
\usepackage{microtype} 
\usepackage{comment}
\usepackage{mathtools, amsthm, amssymb, eucal} 
\usepackage[normalem]{ulem}                    
\usepackage{mdframed}
\usepackage{verbatim}
\usepackage{booktabs} 
\usepackage{mathrsfs}
\usepackage{xypic}
\usepackage[paper=letterpaper, margin=1in, headsep=20pt]{geometry}
\usepackage{hyperref}
\usepackage{enumerate}
\usepackage{multicol}
    \theoremstyle{plain}
    \newtheorem{thm}{Theorem}[section]
    \newtheorem{lem}[thm]{Lemma}
    \newtheorem{prop}[thm]{Proposition}
    \newtheorem{cor}[thm]{Corollary}

\newtheorem{innercustomthm}{Theorem}
\newenvironment{customthm}[1]
    {\renewcommand
    \theinnercustomthm{#1}\innercustomthm}{\endinnercustomthm}

    \theoremstyle{definition}
    \newtheorem{defn}{Definition}[section]
    
    \newtheorem{exmp}{Example}
    [section]
    \newtheorem{question}{Question}
    \newtheorem*{thm*}{Theorem}
    \newtheorem*{prop*}{Proposition}

    \theoremstyle{remark}
    \newtheorem{rem}{Remark}[section]

\newcommand{\introthmname}{}
    \newtheorem{introthminn}{\introthmname}

\definecolor{ki}{RGB}{150,50,0}

\definecolor{ray}{RGB}{0,0,200}

\title{The Hom-Ext quiver and applications to exceptional collections}
\author{Kiyoshi Igusa}
\address{Department of Mathematics, Brandeis University, Waltham, MA 02454}
\email{igusa@brandeis.edu}
\thanks{Supported by the Simons Foundation}
\author{Ray Maresca}
\address{Department of Mathematics, Bowdoin College, 8600 College Station, Brunswick, Maine 04011}
\email{r.maresca@bowdoin.edu}

\begin{document}

\begin{abstract}
We study what we call the Hom-Ext quiver and characterize it as a type of `superquiver'. In type $\tilde{\mathbb{A}}$, the Hom-Ext quiver of an exceptional set is the tiling algebra of the corresponding geometric model. And, in that case, Hom-Ext quivers classify exceptional sets up to Dehn twist of the corresponding geometric model. We show that these Dehn twists are realized by twist functors and give autoequivalences of the derived category. We provide a generating set for the group of autoequivalences of the derived category in type $\tilde{\mathbb{A}}$, and show that the Hom-Ext quiver classifies exceptional sets up to {\color{black} the action of the subgroup of the automorphism group of the derived category generated by twist functors associated to exceptional cycles.} We introduce superquivers, which are a generalization of Hom-Ext quivers. Exceptional sets over finite acyclic quivers are realized as representations of superquivers. Throughout, we list several questions and conjectures that make for, what we believe, exciting new research.
\end{abstract}

\maketitle

\tableofcontents

\section{Introduction}

{\color{black}
Exceptional sequences are a topic of continuing interest in representation theory and combinatorics. The widespread interest in exceptional sequences comes from the many combinatorial models which were known earlier. Exceptional sequences were introduced algebraically in 1987 by \cite{gorrud1987exceptional} \cite{rudakov1990exceptional}. But equivalent combinatorial models were known much earlier since they are in bijection with factorizations of the Coxeter element of the Weyl group \cite{hurwitz1891riemann} \cite{looijenga1974complement}. This led to many combinatorial models for exceptional sequences {\color{black}\cite{araya2013exceptional} \cite{garver2015combinatorial} \cite{garver2019combinatorics} \cite{maresca2022combinatorics} \cite{chang2023exceptional}}.   However, some crucial properties of exceptional sequences, such as the fact that the braid group acts transitively on the set of exceptional sequences are only known to have representation theoretics proofs \cite{crawley1992exceptional}\cite{ringel1994braid}.} 

In this paper, we focus on exceptional collections (or exceptional sets), which are sets of modules that can be ordered into an exceptional sequence in at least one way. An interesting and natural question asked to both authors by Theo Douvropoulos is ``how many exceptional sequences correspond to an exceptional collection?" To answer this question, we define what we call the Hom-Ext quiver in Section \ref{sec: H-E quiver} (Definitions \ref{defn: H-E quiver in D} and \ref{defn: Hom-Ext quiver in repQ}). This is the quiver associated to an Ext algebra. Using the Hom-Ext quiver, we are able to answer the question posed by Theo:

\begin{customthm}{A}[Theorem \ref{thm: linear extensions and exceptional sequences}]\label{thm A}
For a finite acyclic quiver, the number of exceptional sequences associated to an exceptional collection of representations is equal to the number of linear extensions of the poset defined by the Hom-Ext quiver.
\end{customthm}

This result opens the door to several interesting relationships between the theory of exceptional collections and combinatorics as described in Section \ref{sec: H-E quiver counts excepitonal orderings}. For the remainder of the paper, we switch to a more algebraic analysis of Hom-Ext quivers. Namely, we wish to answer the questions ``given an arbitrary exceptional collection, how can we compute the Hom-Ext quiver? What does the Hom-Ext quiver tell us algebraically about the exceptional collection?"

In Section \ref{sec: H-E quiver is tiling algebra}, we begin to answer these questions for quivers of type $\tilde{\mathbb{A}}$, which we always take to be acyclic. In fact, we show that Hom-Ext quivers can be computed directly from geometric models. Namely, in Proposition \ref{prop: H-E quiver is tiling algebra}, we show that for exceptional sets in type $\tilde{\mathbb{A}}$, the path algebra of the Hom-Ext quiver is isomorphic to the tiling algebra of the corresponding geometric model in the sense of \cite{baur2021geometric}. This allows us to conclude that Hom-Ext quivers of exceptional sets in type $\tilde{\mathbb{A}}$ are gentle. In Section \ref{sec: prelim}, we describe all the notations used in our geometric models. These are the same notations that were used in \cite{maresca2022combinatorics}, and \cite{igusa2024clusters}.

Knowing that for exceptional collections in type $\tilde{\mathbb{A}}$, the Hom-Ext quiver is encoded in the corresponding geometric model, we study how the Hom-Ext quiver changes when a Dehn twist is applied to the corresponding geometric model and prove the following:

\begin{customthm}{B}[Theorem \ref{thm: iso H-E quiver iff Dehn twist}]\label{thm B}
Two exceptional sets in type $\tilde{\mathbb{A}}$ have {\color{black} isomorphic} Hom-Ext quivers if and only if their corresponding geometric models differ by a sequence of Dehn twists.
\end{customthm}

This is a satisfying geometric result because it allows us to classify the infinitely many exceptional sets in type $\tilde{\mathbb{A}}$ into finitely many families defined by isoclasses of Hom-Ext quivers. Moreover, these Dehn twists can be realized as twist functors \cite{opper2019auto}. We introduce our notation for twist functors and other autoequivalences in Section \ref{sec: H-E quiver and autoequivalences} following the notations of \cite{broomhead2017discrete} and \cite{opper2019auto}. In this section, we provide an explicit generating set for the group of autoequivalences of the bounded derived category of representations of a quiver of type $\tilde{\mathbb{A}}$ (Theorem \ref{thm: gen set for derived}). Using this, we prove the following:

\begin{customthm}{C}[Corollary \ref{cor: iso HE quivers iff derived equiv}] \label{thm C}
In type $\tilde{\mathbb{A}}$, two exceptional collections have {\color{black} isomorphic} Hom-Ext quivers if and only if they differ by {\color{black} a sequence of twist functors associated to exceptional cycles.}
\end{customthm}

This result is an algebraic reformulation of Theorem \ref{thm B}. Theorem \ref{thm C} leads us to believe that the Hom-Ext quiver is a promising tool in studying exceptional collections over other representation-infinite algebras.

Finally, in Section \ref{sec: superquivers}, we define superquivers (Definition \ref{defn: superquiver}), which are a generalization of Hom-Ext quivers. Using superquivers, we can realize exceptional sets as $\mathcal{C}$-representations of superquivers (functors from superquivers into a suitable triangulated category $\mathcal{C}$). We define twists of superquivers in terms of frozen arrows (Definition \ref{defn: superquiver twist}), and reformulate Theorem \ref{thm C} in this new language:

\begin{customthm}{D}[Theorem \ref{thm: superquiver twists}] \label{thm D}
In type $\tilde{\mathbb{A}}$, if two exceptional collections differ by {\color{black} a sequence of twist functors associated to exceptional cycles}, then their corresponding Hom-Ext quivers are twist equivalent as superquivers with frozen arrows.
\end{customthm}

Throughout the paper, we state some questions and conjectures which we believe make for interesting research.

\section{Preliminaries}\label{sec: prelim}

Let $\Bbbk = \overline{\Bbbk}$ be an algebraically closed field. Let $Q=(Q_0,Q_1,s,t)$ denote a \textbf{quiver} with vertex set $Q_0$, arrow set $Q_1$, and functions $s,t:Q_1 \rightarrow Q_0$ that assign to each arrow a starting and terminal point respectively. We denote by $\Bbbk Q$ the path algebra of $Q$. A \textbf{path of length} $m \geq1$ in $Q$ is a finite sequence of arrows $\alpha_1\alpha_2\dots\alpha_m$ where $t(\alpha_j) = s(\alpha_{j+1})$ for all $1 \leq j \leq m-1$. A \textbf{monomial relation} in $Q$ is given by a path of the form $\alpha_1\alpha_2\dots\alpha_m$ where $m\geq 2$. A two sided ideal $I$ of $\Bbbk Q$ is called \textbf{admissible} if $R^m_Q \subset I \subset R^2_Q$ for some $m \geq 2$ where $R_Q$ is the arrow ideal in $\Bbbk Q$. For an admissible ideal $I$, we denote by mod$\Bbbk Q/I \cong \text{rep} (Q,I)$ the category of finitely generated right $\Bbbk Q/I$-modules, or equivalently, the category of finite dimensional representations of the quiver $Q$ that satisfy the relations $I$. 

To an arbitrary finite quiver $Q$, we can equip an equivalence relation $\sim \, := (\sim_0, \sim_1)$, where $\sim_i$ is an equivalence relation on $Q_i$ for $i \in \{0,1\}$ and if $\alpha \sim_1 \beta$, then $s(\alpha) \sim_0 s(\beta)$ and $t(\alpha) \sim_0 t(\beta)$. We define the \textbf{quotient quiver} of $Q$ by $\sim$ as the quiver $Q/\sim \, := (Q_0/\sim_0, Q_1/\sim_1, s',t')$ where $s'([\alpha]) = [s(\alpha)]$ and $t'([\alpha]) = [t(\alpha)]$ for all $\alpha \in Q_1$. Note that there is a morphism of quivers $\pi: Q \rightarrow Q/\sim$ given by $\pi = (\pi_0,\pi_1)$ where $\pi_0$ ($\pi_1$) is the quotient map on the vertices (arrows) with respect to $\sim_0$ ($\sim_1$). The set of relations $R$ on $\pi(Q) = Q/\sim$ is generated by linear combinations of paths in $Q/\sim$ whose preimages under $\pi$ are zero. {An example of a quotient quiver can be found in Example \ref{exmp: H-E quiver as a quotient}.}

\begin{rem}

Throughout the paper, unless otherwise stated, all arbitrary quivers $Q$ are taken to be finite and acyclic so that $\Bbbk Q$ is a hereditary algebra (submodules of projective modules are projective).

\end{rem}

In the case of hereditary algebras, we have the following lemma, which we will refer to as the Happel--Ringel lemma:

\begin{lem}[Lemma 4.1 in \cite{happel1982tilted}]\label{lem: Happel-Ringel Lemma}
    Let $\Bbbk Q$ be hereditary. If $T_1$ and $T_2$ are indecomposable with Ext$(T_1,T_2) = 0$, then any nonzero morphism $T_2 \rightarrow T_1$ is either a monomorphism or an epimorphism. In particular, if $T_1$ is indecomposable with Ext$(T_1,T_1) = 0$, then Hom$(T_1,T_1)$ is a division ring. 
\end{lem}

The Happel--Ringel lemma allows us to make the following definition. If $E$ is an indecomposable $\Bbbk Q$ module that does not self-extend (Ext$(E,E)= 0)$), we call $E$ an \textbf{exceptional} $\Bbbk Q$ module. By an \textbf{exceptional sequence}, we mean an ordered sequence $(E_1, E_2, \dots, E_k)$ of exceptional $\Bbbk Q$ modules where Ext$(E_i,E_j) = 0 =  \text{Hom}(E_i,E_j)$ for all $j < i$. By an \textbf{exceptional collection (set)} of $\Bbbk Q$ modules, we mean a set $\{E_1, E_2, \dots, E_k\}$ that can be ordered into an exceptional sequence in at least one way. We call an exceptional sequence/collection \textbf{complete} if $k = |Q_0|$. One of the many important properties of exceptional sequences is that given a hereditary algebra $\Bbbk Q$, the braid group acts transitively on the set of exceptional sequences of $\Bbbk Q$ modules \cite{crawley1992exceptional}, \cite{ringel1994braid}. To prove this, the following {uniqueness} lemma is useful{, and will be {\color{black} referenced} throughout the paper.}

\begin{lem}[\cite{crawley1992exceptional}, \cite{ringel1994braid}] \label{lem: uniqueness in exceptional sequence}
    Let $Q$ be an acyclic quiver with $n$ vertices and let $(E_1, E_2, \dots, $ $E_{i-1}, E_{i+1}, \dots, E_n)$ be an exceptional sequence of $\Bbbk Q$ modules. Then there is a unique module $E_i$ such that $(E_1, E_2, \dots, E_{i-1}, E_i, E_{i+1}, \dots, E_n)$ is a complete exceptional sequence for all $i \in \{1,2,\dots,n\}$.
\end{lem}

\subsection{Quivers of type $\tilde{\mathbb{A}}$}

Throughout this paper, we will be referring frequently to quivers of type $\tilde{\mathbb{A}}$. As a result, we briefly recall some important facts and establish notation. To make an extended Dynkin graph of type $\tilde{\mathbb{A}}_{n-1}$ a quiver, we define an \textbf{orientation vector} $\bm{\varepsilon} = (\varepsilon_1, \dots , \varepsilon_{n}) \in \{-,+\}^{n}$. Then \textbf{the quiver of type $\tilde{\mathbb{A}}_{n-1}$} with this orientation, denoted by $Q^{\bm{\varepsilon}} = (Q_0^{\bm{\varepsilon}},Q_1^{\bm{\varepsilon}},s,t)$, is the one such that $Q_0^{\bm{\varepsilon}} = \{1,2,\dots,n-1,n\}$ and for each $\alpha_i \in Q_1^{\bm{\varepsilon}}$ with $1\leq i \leq n$, we define $\alpha_i \in Q^{\varepsilon}_1$ as

\vspace{-.7cm}

\begin{center}
\begin{multicols}{2}

 \begin{displaymath}
   \alpha_i = \left\{
     \begin{array}{lr}
       i \rightarrow i+1 & : \varepsilon_i = +\\
       i \leftarrow i+1 & :  \varepsilon_i = -
     \end{array}
   \right.
\end{displaymath}

\columnbreak

 \begin{displaymath}
   \alpha_n = \left\{
     \begin{array}{lr}
       n \rightarrow 1 & : \varepsilon_n = +\\
       n \leftarrow 1 & :  \varepsilon_n = -
     \end{array}
   \right.
\end{displaymath}

\end{multicols}
\end{center}

So long as $\varepsilon_i \neq \varepsilon_j$ for some $i$ and $j$, which is a convention we take throughout the paper, the path algebras of these quivers are both hereditary and tame. 
By \textbf{tame}, we mean there are infinitely many indecomposable $\Bbbk Q$-modules and for all $n\in\mathbb{N}$, all but finitely many isomorphism classes of $n$-dimensional indecomposables occur in a finite number of one-parameter families. It is known that the module category, hence the Auslander--Reiten quiver, denoted by $\Gamma_{\Bbbk Q}$, of a tame hereditary algebra can be partitioned into three sections: {the preprojective component $\mathcal{P}$, the regular component $\mathcal{R}$, and the preinjective component $\mathcal{I}$.} For $\mathbb{\tilde{A}}$ quivers, the regular component consists of the left, right and homogeneous tubes. We will denote by $P_i, I_i,$ and $S_i$ the indecomposable projective, injective and simple representation at vertex $i$ respectively. For more on representation theory of tame algebras and definitions of these components see \cite{bluebook2}. 

The path algebras of quivers of type $\tilde{\mathbb{A}}$ have an additional structure that simplifies the aforementioned tripartite classification of $\Gamma_{\Bbbk Q}$; namely, they are gentile algebras. 

\begin{defn}
For an admissible ideal $I$, the algebra $B = \Bbbk Q/ I$ is a \textbf{string algebra} if
\begin{enumerate}
\item At each vertex of $Q$, there are at most two incoming arrows and at most two outgoing arrows.
\item For each arrow $\beta$ there is at most one arrow $\alpha$ and at most one arrow $\gamma$ such that $\alpha\beta \notin I$ and $\beta\gamma \notin I$.  \\

If moreover, we have the following two conditions, the string algebra $B$ is called \textbf{gentle}.\\

\item The ideal $I$ is generated by a set of monomials of length two.
\item For every arrow $\alpha$, there is at most one arrow $\beta$ and one arrow $\gamma$ such that $0 \neq \alpha\beta \in I$ and {\color{black} $0 \neq \gamma\alpha \in I$}.
\end{enumerate}
\end{defn}

It is well known that the indecomposable modules over string algebras are either string or band modules \cite{BR}. For $\Bbbk Q/I$ a sting algebra, to define string modules, we first define for $\alpha\in Q_1$ a \textbf{formal inverse} $\alpha^{-1}$, such that $s(\alpha^{-1}) = t(\alpha)$ and $t(\alpha^{-1}) = s(\alpha)$. Let $Q_1^{-1}$ denote the set of formal inverses of arrows in $Q_1$. We call arrows in $Q_1$ \textbf{direct} arrows and those in $Q_1^{-1}$ \textbf{inverse} arrows. We define a \textbf{walk} as a sequence $\omega = \omega_0\dots \omega_r$ such that for all $i\in\{0,1,\dots,r-1\}$, we have $t(\omega_i) = s(\omega_{i+1})$ where $\omega_i \in Q_1 \cup Q_1^{-1}$. A \textbf{string} is a walk $\omega$ with no sub-walk $\alpha\alpha^{-1}$ or $\alpha^{-1}\alpha$. A \textbf{band} $\beta = \beta_1\dots \beta_n$ is a cyclic string, that is, $t(\beta_n) = s(\beta_1)$, such that the $n$-fold concatenation $\beta^n$ is a string, but $\beta$ itself is not a proper power of any string. We define the \textbf{start or beginning} of a string $S = \omega_0\dots \omega_r$, denoted by $s(S)$, as $s(\omega_0)$. Similarly, we define the \textbf{end} of a string $S$, denoted by $t(S)$, as $t(\omega_r)$. For quivers of type $\mathbb{\tilde{A}}$ the band modules lie in the homogeneous tubes and we can classify in which component of the Auslander--Reiten quiver the string modules reside by their shape, as we will soon see. For quivers of type $\tilde{\mathbb{A}}$, we take the convention that all named strings move in the counter-clockwise direction around the quiver $Q$. 

For $l \in \mathbb{N}$ and $i,j\in[n]$, define the string module $(i,j;l)$ to be the string module associated to the reduced walk $e_{i+1}(\alpha_{i+1} \dots \alpha_i)^l\alpha_{i+1}\dots\alpha_{j-1}e_j$ where $\alpha_k \in Q_1 \cup Q_1^{-1}$ for all $k$ and we take the convention that $i+1 = 1$ if $i = n$. The simple module at vertex $j$ is denoted by $(j-1,j;0)$ and is associated to the walk $e_j$ where $e_j$ is the lazy path at vertex $j$. Notice that all strings in $Q$ can be uniquely written in this form. We take the convention that when we write strings in this notation, we take $l$ to be maximal.

{\color{black} 
\begin{exmp}\label{exam: example of string module}
    Consider the quiver $Q^{\bm{\varepsilon}}$ with $\bm{\varepsilon} = (+,-,-)$: 

\[
\xymatrix{
1 \ar_{\alpha_1}[r] \ar@/^1pc/^{\alpha_3}[rr] & 2 & 3 \ar^{\alpha_2}[l]
	}
\]

Below is a table listing some strings, the corresponding graph, our notation for the corresponding string module, and the representation corresponding to said string module.\\

\begin{center}
\begin{tabular}{|c|c|c|c|}
\hline
    String & Graph of String & String Module Notation & Corresponding Representation \\
\hline

$e_2$ & 2 & $(1,2;0)$ & \xymatrix{
0 \ar_{0}[r] \ar@/^1pc/^{0}[rr] & \Bbbk & 0 \ar^{0}[l]
	} \\
    \hline

$e_1(\alpha_1\alpha_2^{-1}\alpha_3^{-1})^1e_1$ & \resizebox{2.5cm}{.15cm}{\xymatrix{1 \ar[ddr] & & & 1 \ar[dl] \\ & & 3 \ar[dl] & \\ & 2 & &
	}} & $(3,1;1)$ & \xymatrix{
\Bbbk^2 \ar_{\begin{bmatrix} 1 & 0 \end{bmatrix}}[r] \ar@/^1pc/^{\begin{bmatrix} 0 & 1 \end{bmatrix}}[rr] & \Bbbk & \Bbbk \ar^{1}[l]
	}  \\
    \hline

$e_2(\alpha_2^{-1}\alpha_3^{-1}\alpha_1)^0\alpha_2^{-1}e_3$ & \xymatrix{ & 3 \ar[dl]\\ 2 & } & $(1,3;0)$ & \xymatrix{
0 \ar_{0}[r] \ar@/^1pc/^{0}[rr] & \Bbbk & \Bbbk \ar^{1}[l]
	} \\
    \hline
\end{tabular}
\end{center} 

One more example can be attained by `lengthening' the string module $(1,3;0)$ to $(1,3;2)$. In this case, the string would be $e_2(\alpha_2^{-1}\alpha_3^{-1}\alpha_1)^2\alpha_2^{-1}e_3$ and the corresponding graph is 

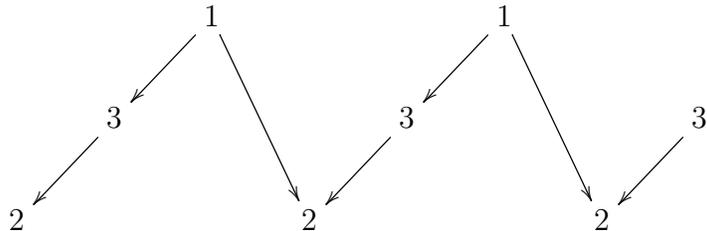
\[
\xymatrix{ & & 1 \ar[dl] \ar[ddr] & & & 1 \ar[dl] \ar[ddr] & & \\ & 3 \ar[dl] & & & 3 \ar[dl] & & & 3 \ar[dl] \\ 2 & & & 2 & & & 2 &  
	}
\]

A representative of the isoclass of representations corresponding to this string module is 

\[
\xymatrix{
\Bbbk^2 \ar_{\begin{bmatrix}1 & 0 & 0\\0 & 1 & 0\end{bmatrix}}[rr] \ar@/^1pc/^{\begin{bmatrix}0 & 1 & 0\\0 & 0 & 1\end{bmatrix}}[rrrr] & & \Bbbk^3 & & \Bbbk^3 \ar^{\begin{bmatrix}1 & 0 & 0\\0 & 1 & 0 \\ 0 & 0 & 1\end{bmatrix}}[ll]
	}
\]

\end{exmp}

}

String modules are useful, since their combinatorial structure can be used to describe both the AR translate $\tau$ \cite{BR}, and morphisms/extensions between string modules. To provide {\color{black} the description of morphisms in type $\tilde{\mathbb{A}}$}, we need some definitions. The following lemma follows from the definition of a string algebra, {and allows us to make the next definition.}

\begin{lem}
Let $S$ be a string of positive length and let $\varepsilon \in \{ Q, Q^{-1} \}$. There is at most one way to add an arrow preceding $s(S)$ whose orientation agrees with $\varepsilon$, such that the resulting walk is still a string. Similarly there is at most one way to add such an arrow following $t(S)$.
\end{lem}

\begin{defn} {\color{white} .}

    \begin{itemize}
        \item If a string $S$ can be extended at its end by a direct arrow, then we can \textbf{add a cohook} at $t(S)$, by which we mean, add a direct arrow at $t(S)$, followed by adding as many inverse arrows as possible. The inverse operation is called \textbf{deleting a cohook} at $t(S)$.

        \item If a string $S$ can be extended by an inverse arrow at $s(S)$, then we can \textbf{add a cohook} at $s(S)$, by which we mean, add an inverse arrow at $s(S)$, followed by as many direct arrows as possible. The inverse operation is called \textbf{deleting a cohook} at $s(S)$.

        \item If a string $S$ can be extended at its end by adding an inverse arrow, by \textbf{adding a hook} at $t(S)$, we mean adding this inverse arrow to $S$ along with as many direct arrows as possible. The inverse operation is called \textbf{deleting a hook} at $t(S)$.

        \item If a string $S$ can be extended at its start by a direct arrow, we can \textbf{add a hook} at $s(S)$, by which we mean, add the direct arrow at $s(S)$ followed by as many inverse arrows as possible. The inverse operation is called \textbf{deleting a hook} at $s(S)$.

    \end{itemize}
\end{defn}

{\color{black} 
\begin{exmp}\label{exmp: adding hooks}
    Consider the string module {\color{black}$(3,3;0)$} from {\color{black} the quiver in} Example \ref{exam: example of string module}. Then we can add a hook at the end of the string to obtain the string $(3,2;1)$ as indicated below:
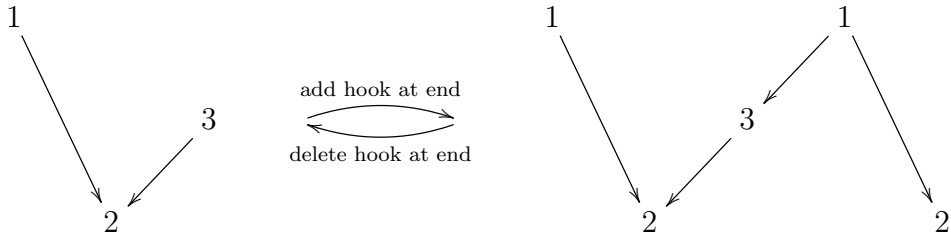
\begin{center}
\begin{tabular}{c}
    \xymatrix{1 \ar[ddr] & & & & & & 1 \ar[ddr] & & & 1 \ar[dl] \ar[ddr] & &  \\ & & 3 \ar[dl] & \ar@/^.5pc/^{\text{add hook at end}}[rr] & & \ar@/^.5pc/^{\text{delete hook at end}}[ll] & & & 3 \ar[dl] & & & \\ & 2 & & & & & & 2 & & & 2} 
\end{tabular}	
\end{center}

We can also delete a hook at the start of the string to obtain the string $(2,3;0)$ as indicated below:

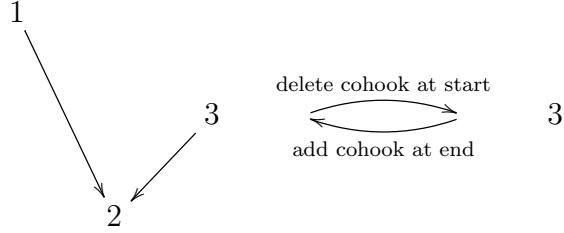
\begin{center}
\begin{tabular}{c}
    \xymatrix{1 \ar[ddr] & & & & & & &  \\ & & 3 \ar[dl] & \ar@/^.5pc/^{\text{delete cohook at start}}[rr] & & \ar@/^.5pc/^{\text{add cohook at end}}[ll] & 3  \\ & 2 & & & } 
\end{tabular}	
\end{center}
    
\end{exmp}
}

 Note that deleting a cohook at the end ({\color{black} resp.} start) of $S$ may not be defined if $S$ does not have any direct ({\color{black} resp.} inverse) arrows. Similarly, note that deleting a hook at the end ({\color{black} resp.} start) of $S$ may not be defined if $S$ does not have any inverse ({\color{black} resp.} direct) arrows. It is known that in type $\tilde{\mathbb{A}}$, all irreducible morphisms in the preprojective (preinjective) component are given by adding hooks (deleting cohooks).  

{\color{black}
A basis for the space of all morphisms between two strings was given by Crawley-Boevey in \cite{CBstrings}. Schr{\"o}er then reformulated the basis given by Crawley-Boevey in \cite{schroer1999modules}. This reformulation was then used in \cite{brustle2020combinatorics} to provide a basis for the vector space of extensions between two strings. Since we will be using the results of \cite{brustle2020combinatorics} throughout the paper, we will recite some of them here. The following definition follows Schr{\"o}er's reformulation in \cite{schroer1999modules}.

\begin{defn}
For $C$ a string in the quiver $Q$, the set of all factorizations of $C$ is $$\mathcal{P}(C) = \{(F,E,D) : F, E, \text{ and } D \text{ are strings of $Q$ and $C = FED$} \}.$$ A triple $(F,E,D) \in \mathcal{P}(C)$ is called a \textbf{quotient factorization} of $C$ if the following hold:
\begin{enumerate}
\item $D = e_{t(E)}$ or $D = \gamma D'$ with $\gamma \in Q_1$.
\item $F = e_{s(E)}$ or $F = F'\theta$ with $\theta \in Q_1^{-1}$.
\end{enumerate} 
We denote the set of all quotient factorizations of $C$ by $\mathcal{F}(C)$. A triple $(F,E,D) \in \mathcal{P}(C)$ is called a \textbf{submodule factorization} of $C$ if the following hold:
\begin{enumerate}
\item $D = e_{t(E)}$ or $D = \gamma D'$ with $\gamma \in Q_1^{-1}$.
\item $F = e_{s(E)}$ or $F = F'\theta$ with $\theta \in Q_1$.
\end{enumerate} 
We denote the set of all submodule factorizations of $C$ by $\mathcal{S}(C)$.
\end{defn}

A quotient factorization induces a module epimorphism from $C$ onto $E$ and dually, a submodule factorization induces a monomorphism from $E$ into $C$. 

\begin{defn}
For $C_1$ and $C_2$ strings in $Q$, a pair $((F_1,E_1,D_1),(F_2,E_2,D_2)) \in \mathcal{F}(C_1) \times \mathcal{S}(C_2)$ is called \textbf{admissible} if $E_1 = E_2$ or $E_1 = E_2^{-1}$.
\end{defn}

Each admissible pair $T = ((F_1,E_1,D_1),(F_2,E_2,D_2))$ provides a morphism $f_T:C_1 \rightarrow C_2$ achieved by projecting onto $E_1$ followed by identifying $E_1$ with $E_2$, then including $E_2$ into $C_2$. These morphisms will be called \textbf{graph maps} and they form a basis for the Hom space as shown by Crawley-Boevey in \cite{CBstrings}.

\begin{thm}\label{thm: basis for Hom}
If $A = \Bbbk Q /I$ is a string algebra and $C_1$ and $C_2$ are string modules, the set of graph maps is a basis for Hom$_A(C_1,C_2)$.
\end{thm} 

Using this result, the extension group between two strings in a gentle algebra was classified in \cite{brustle2020combinatorics}. To state this result, we need some more definitions.

\begin{defn}
A graph map $f_T$ given by an admissible pair $T = ((F_1,E_1,D_1),(F_2,E_2,D_2))$ is called \textbf{two-sided} if at least one of $D_1$ and $D_2$ has positive length, and the same holds for $F_1$ and $F_2$. We say that the string $C_1$ is \textbf{connectable} to the string $C_2$ if there exists an arrow $\alpha$ such that $C_1\alpha C_2$ is a string in $Q$, or dually that $C_2\alpha^{-1}C_1$ is a string in $Q$.
\end{defn}

Notice, if $C_1$ is connectable to $C_2$, we get a non-zero extension $$0 \rightarrow C_2 \rightarrow C_1\alpha C_2 \rightarrow C_1 \rightarrow 0.$$ Moreover, as shown in \cite{schroer1999modules}, two-sided graph maps $f_T\in\text{Hom}_A(C_2,C_1)$ coming from an admissible pair $T = ((F_2,E_2,D_2),(F_1,E_1,D_1))$ give rise to a non-zero extension $$0\rightarrow C_2 \rightarrow F_2ED_1 \oplus F_1ED_2 \rightarrow C_1 \rightarrow 0.$$ Not only this, in \cite{brustle2020combinatorics} it was shown that these extensions form a basis.

\begin{thm}\label{thm: basis for ext}
Let $\Bbbk Q/I$ be a gentle algebra. For strings $C_1$ and $C_2$ in $Q$, the set of extensions coming both from connections $\alpha$ of $C_1$ to $C_2$ and two-sided graph maps $F_T \in \text{Hom}(C_2,C_1)$ form a basis for Ext$^1_{A}(C_1,C_2)$.
\end{thm}
}
\noindent
 For more on string/gentle algebras see \cite{StringAlgebraInfo} and \cite{CBstrings}. 

Let $w = w_0\cdots w_r$ be an indecomposable $\Bbbk \mathbb{\tilde{A}}$-module. It is well known that $w$ is \textbf{preprojective} if and only if there are arrows $\alpha,\beta\in Q_1$ such that $t(\alpha) = s(w_0)$ and $t(\beta) = t(w_r)$. Similarly, $w$ is \textbf{preinjective} if and only if there are arrows $\alpha,\beta\in Q_1$ such that $s(\alpha) = s(w_0)$ and $s(\beta) = t(w_r)$, $w$ is \textbf{left regular} if and only if there are arrows $\alpha,\beta\in Q_1$ such that $t(\alpha) = s(w_0)$ and $s(\beta) = t(w_r)$, $w$ is \textbf{right regular} if and only if there are arrows $\alpha,\beta\in Q_1$ such that $s(\alpha) = s(w_0)$ and $t(\beta) = t(w_r)$ and finally $w$ is \textbf{homogeneous} if and only if $w$ is a band. {\color{black} For instance, for the quiver in Example \ref{exam: example of string module}, the module $(3,1;1)$ is preinjective (it is $I_2$), the module $(1,3;2)$ is preprojective and not projective, the module $(1,1;0)$ is left regular, and the module $(3,3;0)$ in Example \ref{exmp: adding hooks} is right regular.}

\subsubsection{A geometric model for the category of representations of $\tilde{\mathbb{A}}$ quivers}

Given our quiver $Q^{\bm{\varepsilon}}$, we will now associate an annulus $A_{Q^{\bm{\varepsilon}}}$ as follows. If $\varepsilon_i = +(-)$, then $i$ is a marked point on the outer (inner) boundary of the annulus. We moreover write the marked points in clockwise order respecting the natural numerical order of the vertices in $Q^{\bm{\varepsilon}}_0$ {when $Q$ is traversed counterclockwise.}  

\begin{defn}
Let $i,j\in\{1, 2, \dots , n\}$ and $\lambda \in \mathbb{Z}$. An \textbf{arc} $a(i,j)[\lambda]$ on $A_{Q^{\bm{\varepsilon}}}$ is an isotopy class of simple curves in $A_{Q^{\bm{\varepsilon}}}$ where any $\gamma \in a(i,j)[\lambda]$ satisfies: 
\begin{enumerate}
\item $\gamma$ begins at $i$ and ends at $j$, 
\item $\gamma$ travels clockwise through the interior of the annulus from $i$ to $j$. 
\item If $\gamma$ begins and ends on the same boundary component, the integer $\lambda$ is the winding number of $\gamma$ about the inner boundary component. If $\gamma$ connects the two boundary components, $\lambda$ is the clockwise winding number of $\gamma$ about the inner boundary circle of $A_{Q^{\bm{\varepsilon}}}$ when traversing $\gamma$ beginning from the outer boundary component. 
\end{enumerate}
  
A collection of such arcs will be called an \textbf{arc diagram}. If an arc from $i$ to $j$ has counter clockwise winding number $k$, we write $a(i,j)[-k]$. By the \textbf{support} of the arc $a(i,j)[\lambda]$, we mean the marked points $\{i+1, i+2, \dots, j\}$ taking the convention that $n+1 := 1$. Arcs that begin at one boundary component and end at another will be called \textbf{bridging arcs}, while those that begin and end on the same boundary component will be called \textbf{exterior or boundary arcs}.  

\end{defn}

{\color{black} 

\begin{rem}\label{rem: Duality on Annulus}
    It is well known that there is a duality $D: \text{rep}_{\Bbbk}Q \rightarrow \text{rep}_{\Bbbk}Q^{op}$ given by $D = \text{Hom}_{\Bbbk}(-,\Bbbk)$. Throughout the paper, this is the duality to which we are referring. In type $\tilde{\mathbb{A}}$, we can visualize this duality geometrically by continuously interchanging the inner and outer boundary components of the corresponding annulus while preserving the orientation.
\end{rem}

}

We say that two arcs $a(i_1,j_1)[\lambda_1]$ and $a(i_2,j_2)[\lambda_2]$ \textbf{intersect nontrivially} if any two curves $\gamma_1 \in a(i_{1},j_{1})[\lambda_1]$ and $\gamma_2 \in a(i_2,j_2)[\lambda_2]$ intersect in their interiors. Otherwise we say that $a(i_1,j_1)[\lambda_1]$ and $a(i_2,j_2)[\lambda_2]$ \textbf{do not intersect nontrivially}. We say an arc $a(i,j)[\lambda]$ \textbf{self-intersects} if any two curves $\gamma_1 \in a(i,j)[\lambda]$ and $\gamma_2 \in a(i,j)[\lambda]$ intersect in their interiors. Otherwise, we say $a(i,j)[\lambda]$ \textbf{does not self-intersect}. If $a(i_1,j_1)[\lambda_1]$ and $a(i_2,j_2)[\lambda_2]$ do not intersect nontrivially, we say $a_1 = a(i_1,j_1)[\lambda_1]$ is \textbf{clockwise} from $a_2 = a(i_2,j_2)[\lambda_2]$ (or equivalently $a(i_2,j_2)[\lambda_2]$ is \textbf{counterclockwise} from $a(i_1,j_1)[\lambda_1]$) if and only if $a_1$ and $a_2$ appear locally in one of the six possible configurations in Figure \ref{fig: 6 possible local configurations on Annulus}. We say that a collection of arcs $\{a(i_1,j_1)[\lambda_1], a(i_2,j_2)[\lambda_2], \dots , a(i_k,j_k)[\lambda_k]\}$ form a \textbf{cycle} if and only if $a(i_l,j_l)[\lambda_l]$ is clockwise from $a(i_{l+1},j_{l+1})[\lambda_{l+1}]$ for all $l < k$ and $a(i_k,j_k)[\lambda_k]$ is clockwise from $a(i_1,j_1)[\lambda_1]$. {\color{black} We say $a(i,j)[\lambda]$ forms a \textbf{closed arc} if $i = j$ and a \textbf{loop} if moreover, it does not self-intersect. We call an isotopy class of curves $\gamma[\lambda]$ a \textbf{closed curve} with winding number $\lambda$ if any $\delta\in\gamma[\lambda]$ encloses the inner boundary component of the annulus, does not intersect either boundary component of the annulus, and has winding number $\lambda$. We call a closed curve with winding number 1 a \textbf{simple closed curve}. Examples of closed curves can be found in Example \ref{exmp: closed curves}.} 

\begin{figure}
\begin{center}

\tikzset{every picture/.style={line width=0.75pt}} 

\begin{tikzpicture}[x=0.75pt,y=0.75pt,yscale=-1,xscale=1]

\draw   (90.54,86.29) .. controls (90.54,75.05) and (100.43,65.93) .. (112.63,65.93) .. controls (124.83,65.93) and (134.72,75.05) .. (134.72,86.29) .. controls (134.72,97.54) and (124.83,106.65) .. (112.63,106.65) .. controls (100.43,106.65) and (90.54,97.54) .. (90.54,86.29)(60,86.29) .. controls (60,58.18) and (83.56,35.39) .. (112.63,35.39) .. controls (141.69,35.39) and (165.26,58.18) .. (165.26,86.29) .. controls (165.26,114.4) and (141.69,137.19) .. (112.63,137.19) .. controls (83.56,137.19) and (60,114.4) .. (60,86.29) ;
\draw    (165.26,86.29) ;
\draw [shift={(165.26,86.29)}, rotate = 0] [color={rgb, 255:red, 0; green, 0; blue, 0 }  ][fill={rgb, 255:red, 0; green, 0; blue, 0 }  ][line width=0.75]      (0, 0) circle [x radius= 3.35, y radius= 3.35]   ;
\draw   (93.01,228.49) .. controls (93.01,217.25) and (102.9,208.13) .. (115.1,208.13) .. controls (127.3,208.13) and (137.19,217.25) .. (137.19,228.49) .. controls (137.19,239.74) and (127.3,248.85) .. (115.1,248.85) .. controls (102.9,248.85) and (93.01,239.74) .. (93.01,228.49)(62.47,228.49) .. controls (62.47,200.38) and (86.03,177.6) .. (115.1,177.6) .. controls (144.16,177.6) and (167.72,200.38) .. (167.72,228.49) .. controls (167.72,256.6) and (144.16,279.39) .. (115.1,279.39) .. controls (86.03,279.39) and (62.47,256.6) .. (62.47,228.49) ;
\draw    (115.1,208.13) ;
\draw [shift={(115.1,208.13)}, rotate = 0] [color={rgb, 255:red, 0; green, 0; blue, 0 }  ][fill={rgb, 255:red, 0; green, 0; blue, 0 }  ][line width=0.75]      (0, 0) circle [x radius= 3.35, y radius= 3.35]   ;
\draw    (150.46,66.58) -- (165.26,86.29) ;
\draw    (165.26,86.29) -- (150.46,102.75) ;
\draw    (94.96,201.3) -- (115.1,208.13) ;
\draw    (115.1,208.13) -- (134.43,201.3) ;
\draw   (258.26,86.29) .. controls (258.26,75.05) and (268.15,65.93) .. (280.35,65.93) .. controls (292.55,65.93) and (302.44,75.05) .. (302.44,86.29) .. controls (302.44,97.54) and (292.55,106.65) .. (280.35,106.65) .. controls (268.15,106.65) and (258.26,97.54) .. (258.26,86.29)(227.72,86.29) .. controls (227.72,58.18) and (251.28,35.39) .. (280.35,35.39) .. controls (309.42,35.39) and (332.98,58.18) .. (332.98,86.29) .. controls (332.98,114.4) and (309.42,137.19) .. (280.35,137.19) .. controls (251.28,137.19) and (227.72,114.4) .. (227.72,86.29) ;
\draw    (332.98,86.29) ;
\draw [shift={(332.98,86.29)}, rotate = 0] [color={rgb, 255:red, 0; green, 0; blue, 0 }  ][fill={rgb, 255:red, 0; green, 0; blue, 0 }  ][line width=0.75]      (0, 0) circle [x radius= 3.35, y radius= 3.35]   ;
\draw   (260.73,228.49) .. controls (260.73,217.25) and (270.62,208.13) .. (282.82,208.13) .. controls (295.02,208.13) and (304.91,217.25) .. (304.91,228.49) .. controls (304.91,239.74) and (295.02,248.85) .. (282.82,248.85) .. controls (270.62,248.85) and (260.73,239.74) .. (260.73,228.49)(230.19,228.49) .. controls (230.19,200.38) and (253.75,177.6) .. (282.82,177.6) .. controls (311.88,177.6) and (335.44,200.38) .. (335.44,228.49) .. controls (335.44,256.6) and (311.88,279.39) .. (282.82,279.39) .. controls (253.75,279.39) and (230.19,256.6) .. (230.19,228.49) ;
\draw    (282.82,208.13) ;
\draw [shift={(282.82,208.13)}, rotate = 0] [color={rgb, 255:red, 0; green, 0; blue, 0 }  ][fill={rgb, 255:red, 0; green, 0; blue, 0 }  ][line width=0.75]      (0, 0) circle [x radius= 3.35, y radius= 3.35]   ;
\draw   (409.95,86.29) .. controls (409.95,75.05) and (419.84,65.93) .. (432.04,65.93) .. controls (444.24,65.93) and (454.13,75.05) .. (454.13,86.29) .. controls (454.13,97.54) and (444.24,106.65) .. (432.04,106.65) .. controls (419.84,106.65) and (409.95,97.54) .. (409.95,86.29)(379.41,86.29) .. controls (379.41,58.18) and (402.97,35.39) .. (432.04,35.39) .. controls (461.1,35.39) and (484.67,58.18) .. (484.67,86.29) .. controls (484.67,114.4) and (461.1,137.19) .. (432.04,137.19) .. controls (402.97,137.19) and (379.41,114.4) .. (379.41,86.29) ;
\draw    (484.67,86.29) ;
\draw [shift={(484.67,86.29)}, rotate = 0] [color={rgb, 255:red, 0; green, 0; blue, 0 }  ][fill={rgb, 255:red, 0; green, 0; blue, 0 }  ][line width=0.75]      (0, 0) circle [x radius= 3.35, y radius= 3.35]   ;
\draw   (412.41,228.49) .. controls (412.41,217.25) and (422.3,208.13) .. (434.5,208.13) .. controls (446.7,208.13) and (456.59,217.25) .. (456.59,228.49) .. controls (456.59,239.74) and (446.7,248.85) .. (434.5,248.85) .. controls (422.3,248.85) and (412.41,239.74) .. (412.41,228.49)(381.88,228.49) .. controls (381.88,200.38) and (405.44,177.6) .. (434.5,177.6) .. controls (463.57,177.6) and (487.13,200.38) .. (487.13,228.49) .. controls (487.13,256.6) and (463.57,279.39) .. (434.5,279.39) .. controls (405.44,279.39) and (381.88,256.6) .. (381.88,228.49) ;
\draw    (434.5,208.13) ;
\draw [shift={(434.5,208.13)}, rotate = 0] [color={rgb, 255:red, 0; green, 0; blue, 0 }  ][fill={rgb, 255:red, 0; green, 0; blue, 0 }  ][line width=0.75]      (0, 0) circle [x radius= 3.35, y radius= 3.35]   ;
\draw    (282.42,127.7) .. controls (316.95,121.46) and (320.65,108.99) .. (332.98,86.29) ;
\draw    (291.05,115.23) .. controls (314.48,107.74) and (315.71,105.25) .. (332.98,86.29) ;
\draw    (282.82,208.13) .. controls (295.98,198.8) and (307.08,206.29) .. (312.01,223.75) ;
\draw    (282.82,208.13) .. controls (287.35,191.32) and (316.95,198.8) .. (321.88,216.27) ;
\draw    (448.9,49.11) .. controls (469.87,52.86) and (479.73,68.83) .. (484.67,86.29) ;
\draw    (453.84,64.08) .. controls (472.33,65.33) and (472.33,74.06) .. (484.67,86.29) ;
\draw    (409.44,193.81) .. controls (427.94,195.06) and (430.4,198.8) .. (434.5,208.13) ;
\draw    (399.57,215.02) .. controls (418.07,202.54) and (422.17,195.91) .. (434.5,208.13) ;

\draw (136,50.4) node [anchor=north west][inner sep=0.75pt]    {$a_{1}$};
\draw (135,100.4) node [anchor=north west][inner sep=0.75pt]    {$a_{2}$};
\draw (134,198.4) node [anchor=north west][inner sep=0.75pt]    {$a_{1}$};
\draw (305,83.4) node [anchor=north west][inner sep=0.75pt]    {$a_{1}$};
\draw (307,225.4) node [anchor=north west][inner sep=0.75pt]    {$a_{1}$};
\draw (434.04,38.79) node [anchor=north west][inner sep=0.75pt]    {$a_{1}$};
\draw (415,182.4) node [anchor=north west][inner sep=0.75pt]    {$a_{1}$};
\draw (80,199.4) node [anchor=north west][inner sep=0.75pt]    {$a_{2}$};
\draw (265,118.4) node [anchor=north west][inner sep=0.75pt]    {$a_{2}$};
\draw (291,184.4) node [anchor=north west][inner sep=0.75pt]    {$a_{2}$};
\draw (455.84,67.48) node [anchor=north west][inner sep=0.75pt]    {$a_{2}$};
\draw (395,218.4) node [anchor=north west][inner sep=0.75pt]    {$a_{2}$};

\end{tikzpicture}

\end{center}
\caption{The six possible local configurations of an arc $a_1$ being clockwise from another arc $a_2$.}
\label{fig: 6 possible local configurations on Annulus}
\end{figure}
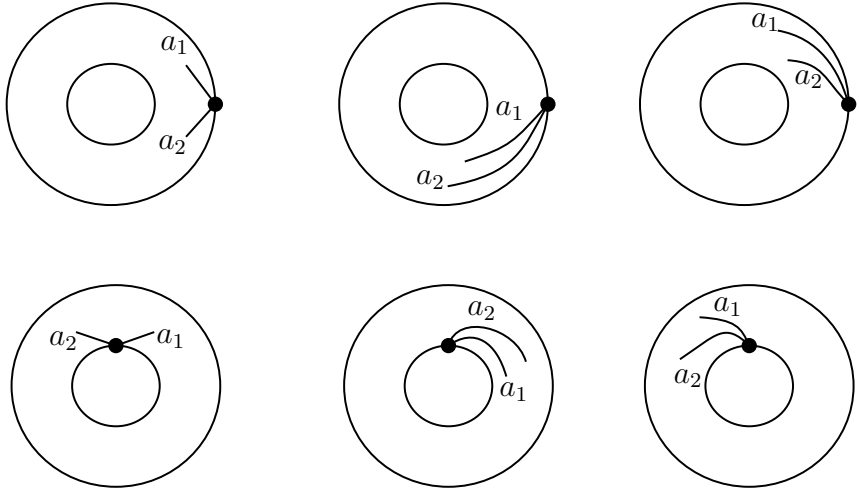

\begin{defn}
An \textbf{exceptional arc diagram} is a collection of $n = |Q_0^{\bm{\varepsilon}}|$ arcs on $A_{Q^{\bm{\varepsilon}}}$ that satisfies the following: 
\begin{enumerate}
\item no arc self-intersects or forms a loop,
\item distinct arcs do not intersect nontrivially, and 
\item the arcs do not form any cycles. 
\end{enumerate}

\end{defn}

\smallskip

\noindent
We have the following bijection $\varphi$ between strings {\color{black} that are not powers of bands} and arcs:

 \[ M = (i,j;l) \overset{\varphi}{\mapsto} \begin{cases} 
          a(i,j)[-l] & \text{if $M$ preinjective} \\
          a(i,j)[l] & \text{otherwise}
       \end{cases}.
    \]
    
Notice that under this bijection, bridging arcs that begin on the inner (outer) boundary component correspond to preinjective (preprojective) modules. On the other hand, exterior arcs beginning and ending on the inner (outer) boundary component correspond to right (left) regular modules. The following proposition was proven in \cite{maresca2022combinatorics} and will be important in what follows.

\begin{prop}[\cite{maresca2022combinatorics} Proposition 3.9] \label{prop: clockwise arcs and hom}

Let $Q^{\bm{\varepsilon}}$ be a quiver of type $\tilde{\mathbb{A}}_{n-1}$, let $U,V \in \text{ind}(\text{rep}(Q^{\bm{\varepsilon}}))$ be two exceptional string modules, and let $\gamma_1$ and $\gamma_2$ be their corresponding arcs on $A_{Q^{\bm\varepsilon}}$.
\begin{enumerate}

\item The arcs $\gamma_1$ and $\gamma_2$ do not intersect at any of their endpoints and they do not intersect nontrivially if and only if there are no extensions or morphisms between $U$ and $V$.

\item The arc $\gamma_1$ is clockwise from $\gamma_2$ if and only if Hom$(U,V)$ contains only non two-sided graph maps and Ext$(U,V) = \text{Hom}(V,U) = \text{Ext}(V,U) = 0$, or $U$ and $V$ are connectable and Hom$(U,V) =\text{Hom}(V,U) = \text{Ext}(V,U) = 0$.

\item The arcs $\gamma_1$ and $\gamma_2$ form a cycle if and only if $U$ is connectable to $V$, $V$ is connectable to $U$, and Hom$(U,V) = 0 = \text{Hom}(V,U)$. 

\item The arcs $\gamma_1$ and $\gamma_2$ intersect nontrivially if and only if there is a two-sided graph map between $U$ and $V$.

\end{enumerate}
\end{prop}

Using this proposition, the next theorem was proven in \cite{maresca2022combinatorics}.

\begin{thm}[\cite{maresca2022combinatorics} Theorem 3.8]\label{thm: ec bijection with arc diagrams}
There is a bijection between exceptional collections of $\Bbbk Q^{\bm{\varepsilon}}$ modules and exceptional arc diagrams on $A_{Q^{\bm{\varepsilon}}}$.
\end{thm}

{\color{black}
In what follows, it will also be important to understand the band modules. In type $\tilde{\mathbb{A}}$, there is one band (up to cyclic permutation and inversion), namely, $\beta = \beta_1 \beta_{2} \cdots \beta_{n}$ where $s(\beta_1) = t(\beta_{n})$ and $\beta_j$ appears exactly once for each $j$. Then $\beta$ defines, up to isomorphism, a family of band modules $(\beta,\varphi)$ where $\varphi$ is an indecomposable automorphism of $\Bbbk$. Here, each vertex of $\beta$ is replaced by a copy of the vector space $\Bbbk$, and the action of an arrow $\beta_l$ is induced by identity morphism if $l \neq n$, and by $\varphi$ if $l = n$. 

For each automorphism $\varphi$, the band module $(\beta,\varphi)$ is the unique quasi-simple at the mouth of a homogeneous tube indexed by $\varphi$. The modules that lie above the quasi-simple bands in homogeneous tubes are powers of the band $\beta$. More precisely, they are of the form $\beta^l = (\beta_1\beta_2\cdots\beta_n)^l$;
that is, they are the $l$-fold concatenation of $\beta$. The band $\beta^l$ corresponds to an infinite family of string modules $(\beta^l,\varphi)$ where $\varphi$ varies through the indecomposable automorphisms of $\Bbbk^l$. Here, each vertex of $\beta$ is replaced by a copy of the vector space $\Bbbk^l$, and the action of an arrow $\beta_j$ is induced by identity morphism if $j \neq n$ and by $\varphi$ if $j = n$.

Then we have the following bijection between families of powers of band modules and closed curves:

\[ (\beta^l,\varphi) \overset{\psi}{\mapsto} \gamma[l]. \]

where $\gamma[l]$ is the isotopy class of a closed curve around the inner boundary component of the annulus with winding number $l$.  }

\begin{exmp}\label{exmp: closed curves}
An example of the simple closed curve corresponding to the band $\beta$ in blue and the closed curve corresponding to $\beta^2$ in red.

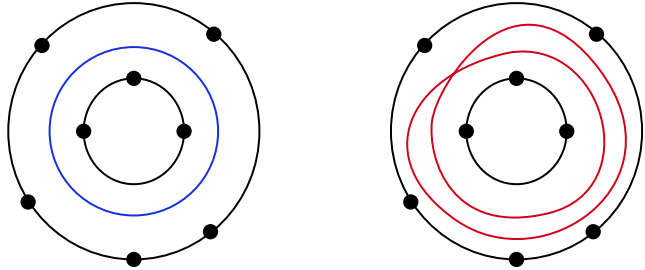
\begin{center}
\tikzset{every picture/.style={line width=0.75pt}} 

\begin{tikzpicture}[x=0.75pt,y=0.75pt,yscale=-1,xscale=1]

\draw   (216.35,141.7) .. controls (216.35,126.95) and (227.69,114.99) .. (241.68,114.99) .. controls (255.66,114.99) and (267.01,126.95) .. (267.01,141.7) .. controls (267.01,156.44) and (255.66,168.4) .. (241.68,168.4) .. controls (227.69,168.4) and (216.35,156.44) .. (216.35,141.7)(178.35,141.7) .. controls (178.35,105.97) and (206.7,77) .. (241.68,77) .. controls (276.65,77) and (305,105.97) .. (305,141.7) .. controls (305,177.43) and (276.65,206.39) .. (241.68,206.39) .. controls (206.7,206.39) and (178.35,177.43) .. (178.35,141.7) ;
\draw    (241.68,114.99) ;
\draw [shift={(241.68,114.99)}, rotate = 0] [color={rgb, 255:red, 0; green, 0; blue, 0 }  ][fill={rgb, 255:red, 0; green, 0; blue, 0 }  ][line width=0.75]      (0, 0) circle [x radius= 3.35, y radius= 3.35]   ;
\draw    (267.01,141.7) ;
\draw [shift={(267.01,141.7)}, rotate = 0] [color={rgb, 255:red, 0; green, 0; blue, 0 }  ][fill={rgb, 255:red, 0; green, 0; blue, 0 }  ][line width=0.75]      (0, 0) circle [x radius= 3.35, y radius= 3.35]   ;
\draw    (216.35,141.7) ;
\draw [shift={(216.35,141.7)}, rotate = 0] [color={rgb, 255:red, 0; green, 0; blue, 0 }  ][fill={rgb, 255:red, 0; green, 0; blue, 0 }  ][line width=0.75]      (0, 0) circle [x radius= 3.35, y radius= 3.35]   ;
\draw    (282.01,92.7) ;
\draw [shift={(282.01,92.7)}, rotate = 0] [color={rgb, 255:red, 0; green, 0; blue, 0 }  ][fill={rgb, 255:red, 0; green, 0; blue, 0 }  ][line width=0.75]      (0, 0) circle [x radius= 3.35, y radius= 3.35]   ;
\draw    (280.35,192.39) ;
\draw [shift={(280.35,192.39)}, rotate = 0] [color={rgb, 255:red, 0; green, 0; blue, 0 }  ][fill={rgb, 255:red, 0; green, 0; blue, 0 }  ][line width=0.75]      (0, 0) circle [x radius= 3.35, y radius= 3.35]   ;
\draw    (188.35,177.39) ;
\draw [shift={(188.35,177.39)}, rotate = 0] [color={rgb, 255:red, 0; green, 0; blue, 0 }  ][fill={rgb, 255:red, 0; green, 0; blue, 0 }  ][line width=0.75]      (0, 0) circle [x radius= 3.35, y radius= 3.35]   ;
\draw    (195.35,98.39) ;
\draw [shift={(195.35,98.39)}, rotate = 0] [color={rgb, 255:red, 0; green, 0; blue, 0 }  ][fill={rgb, 255:red, 0; green, 0; blue, 0 }  ][line width=0.75]      (0, 0) circle [x radius= 3.35, y radius= 3.35]   ;
\draw    (241.68,206.39) ;
\draw [shift={(241.68,206.39)}, rotate = 0] [color={rgb, 255:red, 0; green, 0; blue, 0 }  ][fill={rgb, 255:red, 0; green, 0; blue, 0 }  ][line width=0.75]      (0, 0) circle [x radius= 3.35, y radius= 3.35]   ;
\draw  [color={rgb, 255:red, 22; green, 48; blue, 226 }  ,draw opacity=1 ] (199.18,141.7) .. controls (199.18,118.22) and (218.2,99.2) .. (241.68,99.2) .. controls (265.15,99.2) and (284.18,118.22) .. (284.18,141.7) .. controls (284.18,165.17) and (265.15,184.2) .. (241.68,184.2) .. controls (218.2,184.2) and (199.18,165.17) .. (199.18,141.7) -- cycle ;
\draw   (409.35,141.7) .. controls (409.35,126.95) and (420.69,114.99) .. (434.68,114.99) .. controls (448.66,114.99) and (460.01,126.95) .. (460.01,141.7) .. controls (460.01,156.44) and (448.66,168.4) .. (434.68,168.4) .. controls (420.69,168.4) and (409.35,156.44) .. (409.35,141.7)(371.35,141.7) .. controls (371.35,105.97) and (399.7,77) .. (434.68,77) .. controls (469.65,77) and (498,105.97) .. (498,141.7) .. controls (498,177.43) and (469.65,206.39) .. (434.68,206.39) .. controls (399.7,206.39) and (371.35,177.43) .. (371.35,141.7) ;
\draw    (434.68,114.99) ;
\draw [shift={(434.68,114.99)}, rotate = 0] [color={rgb, 255:red, 0; green, 0; blue, 0 }  ][fill={rgb, 255:red, 0; green, 0; blue, 0 }  ][line width=0.75]      (0, 0) circle [x radius= 3.35, y radius= 3.35]   ;
\draw    (460.01,141.7) ;
\draw [shift={(460.01,141.7)}, rotate = 0] [color={rgb, 255:red, 0; green, 0; blue, 0 }  ][fill={rgb, 255:red, 0; green, 0; blue, 0 }  ][line width=0.75]      (0, 0) circle [x radius= 3.35, y radius= 3.35]   ;
\draw    (409.35,141.7) ;
\draw [shift={(409.35,141.7)}, rotate = 0] [color={rgb, 255:red, 0; green, 0; blue, 0 }  ][fill={rgb, 255:red, 0; green, 0; blue, 0 }  ][line width=0.75]      (0, 0) circle [x radius= 3.35, y radius= 3.35]   ;
\draw    (475.01,92.7) ;
\draw [shift={(475.01,92.7)}, rotate = 0] [color={rgb, 255:red, 0; green, 0; blue, 0 }  ][fill={rgb, 255:red, 0; green, 0; blue, 0 }  ][line width=0.75]      (0, 0) circle [x radius= 3.35, y radius= 3.35]   ;
\draw    (473.35,192.39) ;
\draw [shift={(473.35,192.39)}, rotate = 0] [color={rgb, 255:red, 0; green, 0; blue, 0 }  ][fill={rgb, 255:red, 0; green, 0; blue, 0 }  ][line width=0.75]      (0, 0) circle [x radius= 3.35, y radius= 3.35]   ;
\draw    (381.35,177.39) ;
\draw [shift={(381.35,177.39)}, rotate = 0] [color={rgb, 255:red, 0; green, 0; blue, 0 }  ][fill={rgb, 255:red, 0; green, 0; blue, 0 }  ][line width=0.75]      (0, 0) circle [x radius= 3.35, y radius= 3.35]   ;
\draw    (388.35,98.39) ;
\draw [shift={(388.35,98.39)}, rotate = 0] [color={rgb, 255:red, 0; green, 0; blue, 0 }  ][fill={rgb, 255:red, 0; green, 0; blue, 0 }  ][line width=0.75]      (0, 0) circle [x radius= 3.35, y radius= 3.35]   ;
\draw    (434.68,206.39) ;
\draw [shift={(434.68,206.39)}, rotate = 0] [color={rgb, 255:red, 0; green, 0; blue, 0 }  ][fill={rgb, 255:red, 0; green, 0; blue, 0 }  ][line width=0.75]      (0, 0) circle [x radius= 3.35, y radius= 3.35]   ;
\draw [color={rgb, 255:red, 208; green, 2; blue, 27 }  ,draw opacity=1 ]   (431,102) .. controls (479,94) and (500,171) .. (451,183) .. controls (402,195) and (390,156) .. (392,137) .. controls (394,118) and (434.79,54.02) .. (476.9,111.07) .. controls (519.01,168.13) and (448.61,217.34) .. (404.31,186.67) .. controls (360,156) and (379,113) .. (431,102) -- cycle ;

\end{tikzpicture}
\end{center}
\end{exmp}

\begin{rem}
{
Mathematicians have been using triangulations of marked surfaces to study {\color{black} gentle} algebras as seen in \cite{clustersareinbijectionwithtriangulations}, \cite{Master'sThesisClustersandTriangulations}, and \cite{baur2021geometric} to name a few. Here, our bijection between arcs and indecomposable representations is, {\color{black} we believe}, equivalent to the one in \cite{baur2021geometric}. In \cite{igusa2024clusters}, we showed precisely when the convention in \cite{Master'sThesisClustersandTriangulations} and the one in this paper coincide. }
\end{rem}

\subsection{A geometric model for the derived category of quivers of type $\tilde{\mathbb{A}}$}

Throughout the paper, we reserve $\mathscr{D}$ to denote the bounded derived category of mod$\Bbbk Q$ for $Q$ a finite acyclic quiver. We denote the shift functor by $\Sigma$. In the case $Q$ is finite and acyclic, we know that the indecomposable objects in $\mathscr{D}$ are stalk complexes centered at 0, or finite non-zero shifts of these stalk complexes centered at 0 {(\cite{schofield1990triangulated} Section 5.2)}. For $X$ a $\Bbbk Q$-module, we denote the stalk complex with $X$ in the 0th position by $X$, and the $i$th shift of it by $\Sigma^i X$. 

Now, let $Q^{\bm{\varepsilon}}$ be an acyclic quiver of type $\tilde{\mathbb{A}}_{n-1}$. Then all objects in $\mathscr{D}$ are either indecomposable \textbf{string objects} ($\Sigma^i X$ where $X$ is a string module), or indecomposable \textbf{band objects} ($\Sigma^i X$ where $X$ is a band module).  As was done in \cite{opper2018geometric} for gentle algebras, we can construct a geometric model for $\mathscr{D}$ by using the annulus $A_{Q^{\bm{\varepsilon}}}$. 

\begin{defn}
    Let $a(i,j)[\lambda]$ be an arc on $A_{Q^{\bm{\varepsilon}}}$ and let $P$ denote the set of marked points $\{1,2,\dots,n\}$ on the boundary components of $A_{Q^{\bm{\varepsilon}}}$. A function $f:P\rightarrow \mathbb{Z}$ is called a \textbf{grading} on $a(i,j)[\lambda]$ if for all {marked points} $l$ in the support of $a(i,j)[\lambda]$, we have

    \[ f(l+1) = \begin{cases} 
          f(l) + 1 & \text{if $l$ is on the outer boundary} \\
          f(l) - 1 & \text{otherwise}
       \end{cases}
    \]

    where we take the convention that $n+1 := 1$. We call $(a(i,j)[\lambda],f)$ a \textbf{graded arc}. We define \textbf{graded closed curves} {\color{black} similarly}.
\end{defn}

Before stating the bijection, we note that in the case of a finite acyclic quiver of type $\tilde{\mathbb{A}}$, a grading always exists and is uniquely determined by its value at one point.

\begin{thm}[\cite{opper2018geometric} Theorem 2.12]
\leavevmode
    \begin{itemize}
        \item There is a bijection between isoclasses of indecomposable string objects in $\mathscr{D}$ and graded arcs $(\gamma,f)$.
        \item There is a bijection between families of indecomposable band objects and graded closed curves $(\gamma,f)$.
    \end{itemize}  
\end{thm}

When we say the {\textbf{homotopy class}} of an object in $\mathscr{D}$, we mean the homotopy class of its corresponding arc $\gamma$ in the above bijection. 

In what follows, it will be important to fix a convention for the Auslander--Reiten quiver for rep$Q$, and thus $\mathscr{D}$. We take the convention that for preprojectives, moving down corays corresponds to adding hooks at the end of the string, while moving up rays corresponds to adding hooks at the start of the string. Dually, for preinjectives, moving down corays corresponds to deleting cohooks at the start of the {\color{black} string}, where moving up rays corresponds to deleting cohooks at the end of the string. In the tube of left (right) regular modules, moving one position up a ray corresponds to adding a hook at the start (end) of the string, and moving one position down a coray corresponds to deleting a cohook at the end (start) of the string. This is the same convention used in {\cite{maresca2022combinatorics}}, and by taking this convention, the remainder of the {\color{black} Auslander--Reiten} quiver of $\mathscr{D}$ is fixed.

\section{The Hom-Ext quiver} \label{sec: H-E quiver}

In this section, we will define the Hom-Ext quiver and explore some of its useful properties. Throughout this section, let $Q$ denote a finite, acyclic quiver with $n$ vertices, {\color{black} and} let $\mathscr{D}$ denote the bounded derived category of rep$Q$. We will define the Hom-Ext quiver in two ways; one by viewing it in terms of $\mathscr{D}$ and the other by viewing it in terms of rep$Q$.

\begin{defn} \label{defn: H-E quiver in D}
Let $\{X_1, X_2, \dots, X_m\}$ be a collection of {\color{black} indecomposable} $\Bbbk Q$-modules and $\chi$ their direct sum. Consider the object $\chi \oplus \Sigma \chi \subset \mathscr{D}$ and let $(Q^E,I)$ be the quiver with relations associated to the endomorphism algebra End$_{\mathscr{D}}(\chi \oplus \Sigma \chi)$. In $Q^E = (Q^E_0,Q^E_1,s^E, t^E)$, we label the arrows as follows:

\begin{itemize}
\item If the starting and ending points of an arrow both lie in the set $\chi$, say $X_j\to X_k$, we label this arrow with $\alpha_i$ such that there is no other arrow with the same label.

\item If the starting point of an arrow lies in the set $\chi$ and the ending point lies in $\Sigma \chi$, we label this arrow with $\beta_i$ such that there is no other arrow with the same label.

\item If the starting and ending points of an arrow both lie in the set $\Sigma \chi$, say $\Sigma X_j\to\Sigma X_k$, we label this arrow with $\Sigma \alpha_i$ {\color{black} where} $\alpha_i:X_j\to X_k$ is an arrow in $Q^E_1$.
\end{itemize}

We now equip $Q^E$ with an equivalence relation $\sim \, = (\sim_0, \sim_1)$ as follows. Two vertices are related, denoted by $\Sigma^iX_j \sim_0 \Sigma^kX_l$, if and only if $j = l$. Analogously, two arrows are related, denoted by $\Sigma^i\gamma_j \sim_1 \Sigma^k\gamma_l$, if and only if $\gamma_j = \gamma_l$. The \textbf{Hom-Ext quiver} of $\chi$ is the quotient quiver $(Q^\chi,R) := (Q^E / \sim,R)$ where $R$ is generated by linear combinations of paths in $Q^{\chi}$ whose preimages are zero under {the quotient map} {\color{black}$\pi:Q^{\color{black} E}\to Q^E/\sim$}.

\end{defn}

\begin{exmp}\label{exmp: H-E quiver as a quotient}
Let $Q$ be the quiver 

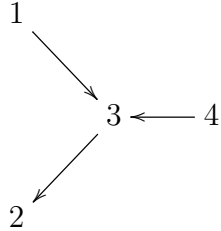
\[\begin{xymatrix}{1 \ar[dr] & & \\ & 3 \ar[dl] & 4 \ar[l] \\ 2 & &}\end{xymatrix} \]

Consider the set of modules $\{{4\atop {3 \atop 2}}, {1 \,\, 4 \atop {3 \, 3 \atop 2}}, {1\atop 3}\}$. Then in $\mathscr{D}$, we have the following quiver associated to the endomorphism algebra:

\[\begin{xymatrix}{ {4\atop {3 \atop 2}} \ar[r]^{\alpha_1} & {1 \,\, 4 \atop {3 \, 3 \atop 2}} \ar[r]^{\alpha_2} & {1\atop 3} \ar[dll]^{\beta_1} \\ {\huge \Sigma} {4\atop {3 \atop 2}} \ar[r]_{\Sigma \alpha_1} & \Sigma {1 \,\, 4 \atop {3 \, 3 \atop 2}} \ar[r]_{\Sigma \alpha_2} & \Sigma {1\atop 3} }\end{xymatrix} \]
 subject to the relations generated by paths of length two.

 To attain the Hom-Ext quiver, we take the quotient and get the following quiver:

\[
\begin{xymatrix}{
\bigg[{4\atop {3 \atop 2}}\bigg] \ar[r]^{[\alpha_1]} & \bigg[{1 \,\, 4 \atop {3 \, 3 \atop 2}}\bigg] \ar[r]^{[\alpha_2]} & \bigg[{1\atop 3} \bigg] \ar@/^2pc/^{[\beta_1]}[ll]^{\beta_1}}\end{xymatrix}
\]

subject to the relations generated by paths of length two.
\end{exmp}

\begin{rem}
Note that in Definition \ref{defn: H-E quiver in D}, we need only look at the objects $\chi$ and $\Sigma \chi$ since $\Bbbk Q$ is hereditary. Moreover, the three types of arrows listed in that definition is an exhaustive list again, since $Q$ is finite and acyclic. This definition is readily generalizable to higher ({potentially} even infinite) global dimensions by looking at all shifts of the modules in $\chi$.
\end{rem}

\begin{rem}
We would like to note that Buan, Reiten, and Thomas defined a Hom-Ext quiver in \cite{buan2011three}. The Hom-Ext quiver defined here is different though, since not all morphisms/extensions between modules correspond to arrows in the quiver. We take only the irreducible ones.
\end{rem}

In order to avoid equivalence classes, we can define the Hom-Ext quiver entirely in terms of rep$Q$. To do this, we will need to introduce some terminology.

\begin{defn}
Let $X, Y,$ and $Z$ be representations of $Q$ and suppose Ext$(X,Z) \neq 0 \neq \text{Ext}(Y,Z)$. We say that the extension $\xi \in \text{Ext}(X,Z)$ is \textbf{pushed forward to the extension} $\xi' \in \text{Ext}(Y,Z)$ \textbf{along $h$} if there exists a morphism $h \in \text{Hom}(X,Y)$ such that there is a morphism $\xi \rightarrow \xi'$ of short exact sequences of the following form:

\begin{center}
\begin{tabular}{c}
\begin{xymatrix}{
0 \ar[r] \ar[d] & Z \ar[r] \ar[d]_{\mathbbm{1}} & E \ar[r] \ar@{.>}[d] & X \ar[r] \ar[d]_h & 0 \ar[d] \\
0 \ar[r] & Z \ar[r] & E' \ar[r] & Y \ar[r] & 0 }
\end{xymatrix}
\end{tabular}
\end{center}

In this case, we also say that the extension $\xi'$ is \textbf{pulled back to the extension $\xi$ along $h$}. 

\begin{rem}\label{rem: push forward in D}
The extension $\xi$ being pushed forward to the extension $\xi'$ is equivalent to $\xi \in \text{Hom}_{\mathscr{D}}(X,\Sigma Z)$ being equal to the composition $\xi' \circ h$ where $h \in \text{Hom}_{\mathscr{D}}(X,Y)$ and $\xi' \in \text{Hom}_{\mathscr{D}}(Y, \Sigma Z)$.
\end{rem}

On the other hand, suppose Ext$(X,Z) \neq 0 \neq \text{Ext}(X,Y)$. We say that the extension $\xi \in \text{Ext}(X,Z)$ is \textbf{pulled back to the extension} $\xi' \in \text{Ext}(X,Y)$ \text{along $h$} if there exists a morphism $h \in \text{Hom}(Y,Z)$ such that there is a morphism $\xi' \rightarrow \xi$ of short exact sequences of the following form:

\begin{center}
\begin{tabular}{c}
\begin{xymatrix}{
0 \ar[r] \ar[d] & Y \ar[r] \ar[d]_{h} & E' \ar[r] \ar@{.>}[d] & X \ar[r] \ar[d]_{\mathbbm{1}} & 0 \ar[d] \\
0 \ar[r] & Z \ar[r] & E \ar[r] & X \ar[r] & 0 }
\end{xymatrix}
\end{tabular}
\end{center}

In this case, we also say that the extension $\xi'$ is \textbf{pushed forward to the extension $\xi$ along the morphism $h$}. 

\begin{rem}\label{rem: pull back in D}
Again, the extension $\xi$ being pulled back to the extension $\xi'$ is equivalent to $\xi \in \text{Hom}_{\mathscr{D}}(X,\Sigma Z)$ being the composition $\Sigma h \circ \xi'$ where $h \in \text{Hom}_{\mathscr{D}}(Y,Z)$ and $\xi' \in \text{Hom}_{\mathscr{D}}(X, \Sigma Y)$.
\end{rem}

\end{defn}

We are now ready to give an alternative definition of the Hom-Ext quiver:

\begin{defn}\label{defn: Hom-Ext quiver in repQ}
Let $\chi = \{X_1, X_2, \dots, X_m\}$ be a collection of {\color{black} indecomposable} $\Bbbk Q$-modules. Here, we use the notation that rHom$(X_i,X_j)$ is the subspace of Hom$(X_i,X_j)$ generated by all maps $X_i\to X_j$ which factor through some {\color{black} $X_q \in \chi$} for $q\neq i,j$ {\color{black} or through some radical endomorphism of $X_i$ or $X_j$}. We denote the \textbf{Hom-Ext quiver of $\chi$} by $(Q^{\chi},R) = ((Q_0^{\chi},Q_1^{\chi},s,t),R)$ and define it as follows. \\

The vertices of $Q^{\chi}$ are the modules in $\chi$; that is, $Q_0^{\chi} = \chi$. For $i,j \in \{1,2,\dots, m\}$, there is an arrow from $X_i$ to $X_j$ if \textit{any} of the following hold:


\begin{enumerate}

{\color{black}
\item There is an arrow from $X_i$ to $X_j$ if and only if there is an arrow from $X_i$ to $X_j$ in the quiver of End$_{\Bbbk Q}(\chi)$. In other words, fix a subset $\{f_1, f_2, \dots, f_k\}$ of Hom$(X_i,X_j)$ which {\color{black}for $i=j$ lie in rad$(\text{Hom}(X_i,X_i))$ and forms a basis in the quotient rad$(\text{Hom}(X_i,X_j))/\text{rHom}(X_i,X_j)$. {\color{black} Moreover,} for $i\neq j$, $\{f_1, f_2, \dots, f_k\}$ map to a basis for Hom$(X_i,X_j)/\text{rHom}(X_i,X_j)$}. Then there is an arrow in $Q_1^{\chi}$ from $X_i$ to $X_j$ for each $f_l \in \{f_1,f_2, \dots, f_k\}$ which is a morphism from $X_i$ to $X_j$. 

{\color{black}
\item Similarly, let $\{\xi_1, \xi_2, \dots, \xi_k\}$ be a basis for Ext$(X_i,X_j)$ modulo the subspace of Ext$(X_i,X_j)$ spanned by all extensions which pull back from an extension in $\text{Ext}(X_i, X_q)$ or push forward to some extension in $\text{Ext}(X_q,X_j)$ {\color{black}for some $X_q \in \chi$}, where $q\neq i,j$ {\color{black}or which are pull backs or push forwards of extensions in Ext$(X_i,X_j)$ along radical endomorphisms of $X_i$ or $X_j$}. For each $l\in\{1,2,\dots,k\}$ so that $\xi_l \in \text{Ext}(X_i,X_j)$, there is an arrow from $X_i$ to $X_j$ in $Q^{\chi}$. 
}
}

\end{enumerate}

In order to define the relations, we will describe what is meant by compositions of arrows in $Q^{\chi}$. We have the following three cases, where the arrows decorated with  a 0 (1) correspond to a morphism (extension): \\

\noindent
Case 1: \begin{xymatrix}{X \ar[r]^f_0 & Y \ar[r]^g_0 & Z}\end{xymatrix} \\

In this case, the composition of $f$ and $g$ give a morphism $g \circ f$. \\

\noindent
Case 2: \begin{xymatrix}{X \ar[r]^f_1 & Y \ar[r]^g_0 & Z}\end{xymatrix}\\

By definition of the Hom arrows, the composition of the arrows $f$ and $g$ can't be a morphism, and hence must be an extension or zero. In this case, we define the composition $g \circ f$ as the push-forward of the extension $f$ along the morphism $g$. \\

\noindent
Case 3: \begin{xymatrix}{X \ar[r]^f_0 & Y \ar[r]^g_1 & Z}\end{xymatrix} \\

Like in the previous case, the composition of $f$ and $g$ must be an extension. In this case, we define the composition $g \circ f$ as the pull-back of the extension $g$ along the morphism $f$. \\

\noindent
Case 4: \begin{xymatrix}{X \ar[r]^f_1 & Y \ar[r]^g_1 & Z}\end{xymatrix} \\

The composition of $f$ and $g$ in this case is always zero since $\Bbbk Q$ is hereditary. \\

We define compositions of finitely many arrows recursively. The set of relations $R$ is generated by the minimal paths that compose to zero and linear combinations of paths that equal zero.

\end{defn}

{\color{black} 
\begin{rem}
    We note that the Hom-Ext quiver isn't new to representation theory. It is the graded quiver associated to the Ext algebra, also called the Yoneda algebra, which is important in the study of $A_{\infty}$ structures. See \cite{keller2001introduction} and \cite{lu2009infinity} to name a couple appearances of the Ext algebra in the study of $A_{\infty}$ structures. What we believe is new though, is the application of the {\color{black}quiver of the} Ext algebra to the study of  exceptional sequences/sets.
\end{rem}
}

\begin{lem}
The definitions of the Hom-Ext quiver in Definitions \ref{defn: H-E quiver in D} and \ref{defn: Hom-Ext quiver in repQ} are equivalent.
\end{lem}

\begin{proof}
These are both equivalent definitions of the quiver associated to the Ext algebra.
\end{proof}

\subsection{Applications to exceptional sets}\label{sec: H-E quiver counts excepitonal orderings}

Now that we have two equivalent definitions of the Hom-Ext quiver, we will analyze its applications to the study of exceptional sets. Namely, our goal for this subsection is to show that the Hom-Ext quiver can be used to detect exceptionality of a set of modules, and moreover, that it can be used to compute how many exceptional orderings there are of a fixed exceptional set. 

\begin{lem}\label{lem: Hom-Ext quiver acyclic}
Let $\chi = \{X_1, X_2, \dots, X_n\}$ be a set of representations of $Q$ where $|Q_0| = n$. If $\chi$ is an exceptional set then $Q^{\chi}$ has no loops or cycles.
\end{lem}

\begin{proof}

Suppose that $\chi$ is an exceptional set. Then $X_i$ is exceptional for all $i$, hence Hom$(X_i,X_i) \cong \Bbbk$ and Ext$(X_i,X_i) =0$, allowing us to conclude that $Q^{\chi}$ has no loops. Now, suppose $Q^{\chi}$ has an oriented cycle: $X_i \rightarrow X_{i+1} \rightarrow \cdots \rightarrow X_j \rightarrow X_i$. Then there is a nontrivial morphism or extension from $X_k$ to $X_{k+1}$, for all $k \in \{i, i+1, \dots, j-1\}$. Thus, in any exceptional ordering, we have that $X_j$ must come after $X_i$. On the other hand, since there is an arrow from $X_j$ to $X_i$, we have that $X_j$ must come before $X_i$ in any exceptional ordering. Therefore, $\chi$ is not an exceptional collection and we have the result by contraposition.
\end{proof}

In order to compute the number of exceptional sequences associated to a fixed exceptional collection, we need to introduce a partial order on the Hom-Ext quiver.

\begin{defn}\label{defn: partial order}
{\color{black} Let $Q^{\chi}$ be the Hom-Ext quiver associated to a set $\chi = \{X_1, X_2, \dots, X_m\}$ of representations of $Q$.} We define a partial order on $Q^{\chi}_0$ by fixing $X_i \leq_e X_j$ if and only if there is a path (which can have relations and can be the lazy path) from $X_i$ to $X_j$.
\end{defn}

\begin{lem}\label{lem: partial order well-defined}
In the case that $Q^{\chi}$ is finite and acyclic, the partial order in Definition \ref{defn: partial order} is well-defined.
\end{lem}


\begin{proof} Let $\chi$ be a set of $Q$ representations and $Q^{\chi}$ be the corresponding Hom-Ext quiver. Moreover, suppose that $Q^{\chi}$ is finite and acyclic. 

\begin{enumerate}
\item $X_i \leq_e X_i$ since we have a lazy path $e_i$.
\item Suppose $X_i \leq_e X_j$ and $X_j \leq_e X_i$. Then there is a path from $X_i$ to $X_j$ and one back from $X_j$ to $X_i$. Since we assume $Q^{\chi}$ to be finite and acyclic, these two paths must be the lazy path, hence $X_i = X_j$.
\item Suppose $X_i \leq_e X_j$ and $X_j \leq_e X_k$. Then combining the path from $X_i$ to $X_j$ with the path from $X_j$ to $X_k$ gives a path from $X_i$ to $X_k$, hence $X_i \leq_e X_k$.
\end{enumerate}
\end{proof}

To study the number of exceptional sequences associated to an exceptional collection, we will need to study the total orderings of the vertices of the Hom-Ext quiver that respect the partial order in Definition \ref{defn: partial order}. Formally these total orders are called linear extensions.

\begin{defn}
Let $\leq$ and $\leq'$ be partial orders on a set $S$. We call $\leq'$ a \textbf{linear extension} of $\leq$ when {\color{black} \textit{both}} of the following hold:

\begin{enumerate}
\item $\leq'$ is a total order {\color{black} on $S$}.
\item For every $s_1,s_2 \in S$, if $s_1 \leq s_2$, then $s_1 \leq' s_2$.
\end{enumerate}
\end{defn}

\begin{thm}\label{thm: linear extensions and exceptional sequences}
Let $\chi$ be an exceptional collection. The exceptional orderings of $\chi$ are in bijection with the linear extensions of $\leq_e$. 
\end{thm}

\begin{proof}
Since $\chi$ is an exceptional set, it follows from Lemmas \ref{lem: Hom-Ext quiver acyclic} and \ref{lem: partial order well-defined} that $\leq_e$ is a well-defined partial order. Let $(X_1, X_2, \dots, X_n)$ be an exceptional ordering of $\chi$ and suppose that for some $j > i$, we have $X_j \leq_e X_i$. Then there is a path in $Q^{\chi}$ from $X_j$ to $X_i$, allowing us to conclude that $X_j$ must come before $X_i$ in any exceptional sequence, a contradiction. {\color{black} Thus, the exceptional ordering $(X_1, X_2, \dots, X_n)$ gives a linear extension $X_1 \leq' X_2 \leq' \cdots \leq' X_n$.}

On the other hand, suppose $\leq'$ is a linear extension of $\leq_e$ and that $(X_{i_1}, X_{i_2}, \dots, X_{i_n})$ is the ordering of the vertices of $Q^{\chi}$ given by $\leq'$. For some distinct $k,l \in \{1,2,\dots,n\}$, suppose Hom$(X_{i_l},X_{i_k}) \neq 0$ or Ext$(X_{i_l}, X_{i_k}) \neq 0$ where $X_{i_k} \leq' X_{i_l}$. Then there is a (relation avoiding) path in $Q^{\chi}$ from $X_{i_l}$ to $X_{i_k}$. Thus $X_{i_l} \leq_e X_{i_k}$, which contradicts the fact that $\leq'$ is a linear extension of $\leq_e$. {\color{black} Thus each linear extension $\leq'$ of $\leq_e$ gives an exceptional ordering of $\chi$.}
\end{proof}

\begin{cor}
Let $\chi = \{X_1, X_2, \dots, X_n\}$ be a set of representations of $Q$ where $|Q_0| = n$. Then $\chi$ is an exceptional set if and only if $Q^{\chi}$ has no loops or cycles.
\end{cor}

\begin{proof}
The forward direction is Lemma \ref{lem: Hom-Ext quiver acyclic}. Conversely, suppose $Q^{\chi}$ has no loops or cycles. Then by Lemma \ref{lem: partial order well-defined}, we have a well-defined partial order $\leq_e$. Pick a linear extension $\leq'$ of $\leq_e$. Then the ordering of $Q_0^{\chi}$ with respect to $\leq'$ is an exceptional sequence {\color{black} by Theorem \ref{thm: linear extensions and exceptional sequences}.}
\end{proof}

Theorem \ref{thm: linear extensions and exceptional sequences} provides an interesting relationship between combinatorics and representation theory. Namely, it connects Young-tableaux, linear extensions, and parking functions with the number of exceptional sequences associated to an exceptional collection. A lot of exciting {\color{black} research in combinatorics} has been done on these topics. For instance, in 1991, it was proven that the problem of counting linear extensions is $\#P$-complete \cite{brightwell1991counting}. There are connections between the number of linear extensions and the number of standard Young-tableaux of skew-shape \cite{morales2018asymptotics}. Moreover, in \cite{morales2018asymptotics}, the asymptotics of these Young-tableaux are studied. It would be interesting to see if there is a representation theoretic analogue {\color{black} of these asymptotics}. There is also a relationship between pattern-avoiding parking functions and certain Young-tableaux \cite{garcia2024defective}. We believe it is the case that not all Hom-Ext quivers of exceptional collections give rise to Young-diagrams of skew shape; however, that they give rise to `Tetris-like' diagrams. We plan to study these diagrams in our future work. 

{\color{black}
Moreover, there is a relationship between the structure of the Hom-Ext quiver and relative projectivity/injectivity in exceptional sets of modules over finite dimensional hereditary algebras. In a project in progress, the second author along with Sarah Levine study this relationship, and in particular, show that the Hom-Ext quiver can be used to completely determine relative projectivity and injectivity in exceptional collections in type $\mathbb{A}$.
}

\section{A geometric realization of the Hom-Ext quiver in type $\tilde{\mathbb{A}}$} \label{sec: H-E quiver is tiling algebra}

In this section, we will show that for exceptional collections in type $\tilde{\mathbb{A}}$, Hom-Ext quivers {\color{black} of exceptional sets} can be seen in the corresponding exceptional arc diagrams. Moreover, we will study how the Hom-Ext quiver behaves relative to twists of the annulus. Throughout the section, unless otherwise stated, let $Q^{\bm{\varepsilon}}$ be a quiver of type $\tilde{\mathbb{A}}_{n-1}$ and let $A_{Q^{\bm{\varepsilon}}}$ be the corresponding annulus with $p$ marked points on the inner boundary component, and $q$ on the outer where $p + q = n$. 

Let $x$ be a marked point {\color{black} on a boundary component of }$A_{Q^{\bm{\varepsilon}}}$ and $D_{\chi}$ an exceptional arc diagram. On the same boundary component as $x$, suppose $m$ and $m'$ are two non-marked points such that $x$ is the only marked point between them. Draw an arc $c$ with $m$ and $m'$ as endpoints and winding number zero that passes over the unique marked point between $m$ and $m'$, and is therefore  homotopic to the boundary component between them. Following the notation in \cite{baur2021geometric}, we call the sequence of arcs $\gamma_i, \gamma_{i+1}, \dots, \gamma_h \in D_{\chi}$ which $c$ crosses in the clockwise order the \textbf{complete fan at $x$}. We say that $\gamma_i$ is \textbf{immediately clockwise} of $\gamma_j$ if there exists some marked point $x$ such that $\gamma_i$ immediately follows $\gamma_j$ in the complete fan at $p$ (in this case, we also say that $\gamma_j$ is immediately counter clockwise of $\gamma_i$). We also call $\gamma_j$ and $\gamma_i$ \textbf{immediate successors}.

\begin{exmp}\label{exam: complete fan}

Consider the quiver $Q^{\bm{\varepsilon}}$ where $\bm{\varepsilon} = (+,-,+,-)$: 

\[
\xymatrix{
1 \ar_{\alpha_1}[r] \ar@/^2pc/^{\alpha_4}[rrr] & 2 & 3 \ar^{\alpha_2}[l] \ar_{\alpha_3}[r] & 4 
	}
\]

Then we have an exceptional collection given by $\{(4,2;0), (1,3;0), (4,3;0), (3,4;0)\}$. Below is the exceptional arc diagram.

\begin{center}  
\tikzset{every picture/.style={line width=0.75pt}} 
\begin{tikzpicture}[x=0.75pt,y=0.75pt,yscale=-1,xscale=1]
\draw   (272.01,143.25) .. controls (272.01,125.19) and (287.06,110.55) .. (305.64,110.55) .. controls (324.21,110.55) and (339.27,125.19) .. (339.27,143.25) .. controls (339.27,161.31) and (324.21,175.95) .. (305.64,175.95) .. controls (287.06,175.95) and (272.01,161.31) .. (272.01,143.25)(222.96,143.25) .. controls (222.96,98.1) and (259.97,61.5) .. (305.64,61.5) .. controls (351.3,61.5) and (388.32,98.1) .. (388.32,143.25) .. controls (388.32,188.4) and (351.3,225) .. (305.64,225) .. controls (259.97,225) and (222.96,188.4) .. (222.96,143.25) ;
\draw [color={rgb, 255:red, 22; green, 48; blue, 226 }  ,draw opacity=1 ]   (305.64,110.55) .. controls (398.21,33) and (389.22,316.5) .. (222.96,143.25) ;
\draw [color={rgb, 255:red, 208; green, 2; blue, 27 }  ,draw opacity=1 ]   (222.96,143.25) .. controls (249.92,109.5) and (264.9,93) .. (305.64,110.55) ;
\draw [color={rgb, 255:red, 144; green, 19; blue, 254 }  ,draw opacity=1 ]   (305.64,110.55) .. controls (368.25,84) and (371.24,201) .. (305.64,177.04) ;
\draw [color={rgb, 255:red, 65; green, 117; blue, 5 }  ,draw opacity=1 ]   (222.96,143.25) .. controls (266.39,238.5) and (356.27,234) .. (388.32,143.25) ;
\draw    (234,102) .. controls (252,117) and (250,181) .. (238,189) ;
\draw    (286,117) .. controls (287,86) and (331,98) .. (322,115) ;

\draw (404.52,133.83) node [anchor=north west][inner sep=0.75pt]    {$1$};
\draw (311.62,152.15) node [anchor=north west][inner sep=0.75pt]    {$2$};
\draw (287.66,119.45) node [anchor=north west][inner sep=0.75pt]    {$4$};
\draw (197.48,134.9) node [anchor=north west][inner sep=0.75pt]    {$3$};
\draw (175.59,34.4) node [anchor=north west][inner sep=0.75pt]    {$\{\textcolor[rgb]{0.56,0.07,1}{( 4,2;0)} ,\ \textcolor[rgb]{0.25,0.46,0.02}{( 1,3;0}) ,\ \textcolor[rgb]{0.09,0.19,0.89}{( 4,3;0)} ,\ \textcolor[rgb]{0.82,0.01,0.11}{( 3,4;0)}\}$};

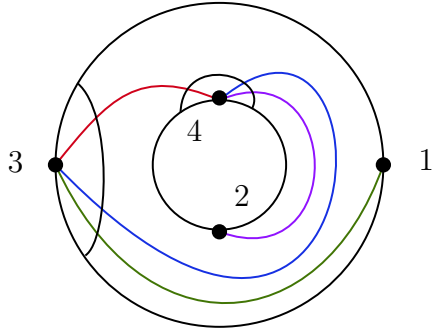
\draw    (305.64,109.46) ;
\draw [shift={(305.64,109.46)}, rotate = 0] [color={rgb, 255:red, 0; green, 0; blue, 0 }  ][fill={rgb, 255:red, 0; green, 0; blue, 0 }  ][line width=0.75]      (0, 0) circle [x radius= 3.35, y radius= 3.35]   ;
\draw    (388.32,143.25) ;
\draw [shift={(388.32,143.25)}, rotate = 0] [color={rgb, 255:red, 0; green, 0; blue, 0 }  ][fill={rgb, 255:red, 0; green, 0; blue, 0 }  ][line width=0.75]      (0, 0) circle [x radius= 3.35, y radius= 3.35]   ;
\draw    (305.64,177.04) ;
\draw [shift={(305.64,177.04)}, rotate = 0] [color={rgb, 255:red, 0; green, 0; blue, 0 }  ][fill={rgb, 255:red, 0; green, 0; blue, 0 }  ][line width=0.75]      (0, 0) circle [x radius= 3.35, y radius= 3.35]   ;
\draw    (222.96,143.25) ;
\draw [shift={(222.96,143.25)}, rotate = 0] [color={rgb, 255:red, 0; green, 0; blue, 0 }  ][fill={rgb, 255:red, 0; green, 0; blue, 0 }  ][line width=0.75]      (0, 0) circle [x radius= 3.35, y radius= 3.35]   ;
\end{tikzpicture}
\end{center}
\vspace{-2.5 cm}
At vertex 3, we see the arc called $c$ in the above paragraph. The complete fan at $3$ is then $a(3,4)[0], a(4,3)[0], a(1,3)[0]$. Similarly, the complete fan at vertex 4 is $a(3,4)[0], a(4,3)[0], a(4,2)[0]$. The complete fans at vertices 2 and 1 are the singletons $a(4,2)[0]$ and $a(1,3)[0]$ respectively.
\end{exmp}

Using these notions, we can begin exploring the relationship between the Hom-Ext quiver associated to an exceptional set and the corresponding exceptional arc diagram.

\begin{figure}[h!]
\begin{center}

\tikzset{every picture/.style={line width=0.75pt}} 

\begin{tikzpicture}[x=0.75pt,y=0.75pt,yscale=-1,xscale=1]

\draw   (87.01,100.25) .. controls (87.01,82.19) and (102.06,67.55) .. (120.64,67.55) .. controls (139.21,67.55) and (154.27,82.19) .. (154.27,100.25) .. controls (154.27,118.31) and (139.21,132.95) .. (120.64,132.95) .. controls (102.06,132.95) and (87.01,118.31) .. (87.01,100.25)(37.96,100.25) .. controls (37.96,55.1) and (74.97,18.5) .. (120.64,18.5) .. controls (166.3,18.5) and (203.32,55.1) .. (203.32,100.25) .. controls (203.32,145.4) and (166.3,182) .. (120.64,182) .. controls (74.97,182) and (37.96,145.4) .. (37.96,100.25) ;

\draw    (120.64,182) .. controls (93,179.5) and (74,158.5) .. (70,140.5) ;
\draw [color={rgb, 255:red, 208; green, 2; blue, 27 }  ,draw opacity=1 ]   (74,128.5) .. controls (78,149.5) and (82,160.5) .. (120.64,182) ;
\draw   (88.01,298.25) .. controls (88.01,280.19) and (103.06,265.55) .. (121.64,265.55) .. controls (140.21,265.55) and (155.27,280.19) .. (155.27,298.25) .. controls (155.27,316.31) and (140.21,330.95) .. (121.64,330.95) .. controls (103.06,330.95) and (88.01,316.31) .. (88.01,298.25)(38.96,298.25) .. controls (38.96,253.1) and (75.97,216.5) .. (121.64,216.5) .. controls (167.3,216.5) and (204.32,253.1) .. (204.32,298.25) .. controls (204.32,343.4) and (167.3,380) .. (121.64,380) .. controls (75.97,380) and (38.96,343.4) .. (38.96,298.25) ;

\draw   (320.01,100.25) .. controls (320.01,82.19) and (335.06,67.55) .. (353.64,67.55) .. controls (372.21,67.55) and (387.27,82.19) .. (387.27,100.25) .. controls (387.27,118.31) and (372.21,132.95) .. (353.64,132.95) .. controls (335.06,132.95) and (320.01,118.31) .. (320.01,100.25)(270.96,100.25) .. controls (270.96,55.1) and (307.97,18.5) .. (353.64,18.5) .. controls (399.3,18.5) and (436.32,55.1) .. (436.32,100.25) .. controls (436.32,145.4) and (399.3,182) .. (353.64,182) .. controls (307.97,182) and (270.96,145.4) .. (270.96,100.25) ;

\draw   (321.01,298.25) .. controls (321.01,280.19) and (336.06,265.55) .. (354.64,265.55) .. controls (373.21,265.55) and (388.27,280.19) .. (388.27,298.25) .. controls (388.27,316.31) and (373.21,330.95) .. (354.64,330.95) .. controls (336.06,330.95) and (321.01,316.31) .. (321.01,298.25)(271.96,298.25) .. controls (271.96,253.1) and (308.97,216.5) .. (354.64,216.5) .. controls (400.3,216.5) and (437.32,253.1) .. (437.32,298.25) .. controls (437.32,343.4) and (400.3,380) .. (354.64,380) .. controls (308.97,380) and (271.96,343.4) .. (271.96,298.25) ;

\draw [color={rgb, 255:red, 22; green, 48; blue, 226 }  ,draw opacity=1 ]   (81,128.5) .. controls (83,149.5) and (90,148.5) .. (120.64,182) ;
\draw [color={rgb, 255:red, 22; green, 48; blue, 226 }  ,draw opacity=1 ]   (121.64,380) .. controls (152,374.5) and (193,342.5) .. (189,324.5) ;
\draw [color={rgb, 255:red, 208; green, 2; blue, 27 }  ,draw opacity=1 ]   (174,320.5) .. controls (178,341.5) and (158,362.5) .. (121.64,380) ;
\draw [color={rgb, 255:red, 0; green, 0; blue, 0 }  ,draw opacity=1 ]   (161,318.5) .. controls (160,343.5) and (158,351.5) .. (121.64,380) ;
\draw    (353.64,182) .. controls (326,179.5) and (302,161.5) .. (299,139.5) ;
\draw    (354.64,380) .. controls (327,377.5) and (308,356.5) .. (304,338.5) ;
\draw [color={rgb, 255:red, 208; green, 2; blue, 27 }  ,draw opacity=1 ]   (307,128.5) .. controls (311,149.5) and (315,160.5) .. (353.64,182) ;
\draw [color={rgb, 255:red, 22; green, 48; blue, 226 }  ,draw opacity=1 ]   (405,127.5) .. controls (398,153.5) and (396,162.5) .. (353.64,182) ;
\draw [color={rgb, 255:red, 22; green, 48; blue, 226 }  ,draw opacity=1 ]   (415,329.5) .. controls (414,350.5) and (395,365.5) .. (354.64,380) ;
\draw [color={rgb, 255:red, 208; green, 2; blue, 27 }  ,draw opacity=1 ]   (405,321.5) .. controls (404,346.5) and (393,356.5) .. (354.64,380) ;

\draw (122.64,185.4) node [anchor=north west][inner sep=0.75pt]    {$p$};
\draw (355.64,185.4) node [anchor=north west][inner sep=0.75pt]    {$p$};
\draw (356.64,383.4) node [anchor=north west][inner sep=0.75pt]    {$p$};
\draw (123.64,383.4) node [anchor=north west][inner sep=0.75pt]    {$p$};

\draw    (120.64,182) ;
\draw [shift={(120.64,182)}, rotate = 0] [color={rgb, 255:red, 0; green, 0; blue, 0 }  ][fill={rgb, 255:red, 0; green, 0; blue, 0 }  ][line width=0.75]      (0, 0) circle [x radius= 3.35, y radius= 3.35]   ;

\draw    (121.64,380) ;
\draw [shift={(121.64,380)}, rotate = 0] [color={rgb, 255:red, 0; green, 0; blue, 0 }  ][fill={rgb, 255:red, 0; green, 0; blue, 0 }  ][line width=0.75]      (0, 0) circle [x radius= 3.35, y radius= 3.35]   ;

\draw    (353.64,182) ;
\draw [shift={(353.64,182)}, rotate = 0] [color={rgb, 255:red, 0; green, 0; blue, 0 }  ][fill={rgb, 255:red, 0; green, 0; blue, 0 }  ][line width=0.75]      (0, 0) circle [x radius= 3.35, y radius= 3.35]   ;

\draw    (354.64,380) ;
\draw [shift={(354.64,380)}, rotate = 0] [color={rgb, 255:red, 0; green, 0; blue, 0 }  ][fill={rgb, 255:red, 0; green, 0; blue, 0 }  ][line width=0.75]      (0, 0) circle [x radius= 3.35, y radius= 3.35]   ;

\end{tikzpicture}

\end{center}
\caption{The four local possibilities for three arcs to intersect at an endpoint $p$. The arc corresponding to $X_i$ is depicted in blue, $X_k$ in red, and $X_j$ in black.}
\label{fig: options for not immediate successor}
\end{figure}
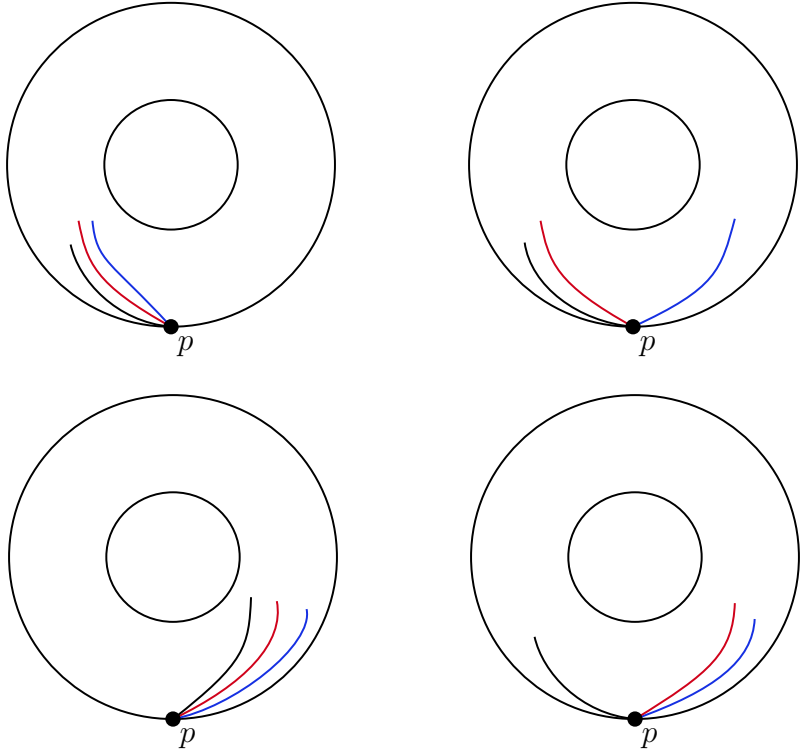

\begin{lem}\label{lem: arrow and immediately clockwise}
Let $\chi = \{X_1, X_2, \dots, X_n\}$ be an exceptional collection of $Q^{\varepsilon}$ representations. There is an arrow $X_i \rightarrow X_j$ in $Q^{\chi}$ if and only if $\gamma_{X_i}$ is immediately clockwise of $\gamma_{X_j}$ at some marked point $p$, where $\gamma_{X_i}$ is the arc corresponding to $X_i$ in the exceptional arc diagram $D_{\chi}$.

\end{lem}

\begin{proof}
Suppose first that $\gamma_{X_i}$ is not immediately clockwise of $\gamma_{X_j}$. Then if $\gamma_{X_i}$ and $\gamma_{X_j}$ do not intersect at an endpoint, we have that the corresponding representations are Hom-Ext orthogonal by Proposition \ref{prop: clockwise arcs and hom}. This implies that there is no arrow from $X_i$ to $X_j$ in $Q^{\chi}$. In the case that the arcs do intersect at a shared endpoint $p$, we have that locally, the exceptional arc diagram looks like one of the four pictures in Figure \ref{fig: options for not immediate successor}. We will prove the top two cases, as the bottom two cases are similar. Moreover, the case in which the endpoint is on the inner boundary component follows from the duality in Remark \ref{rem: Duality on Annulus}.

Suppose we are in the case of the diagram in the top left. Then by Proposition \ref{prop: clockwise arcs and hom}, we have that Hom$(X_i, X_k) \neq 0 \neq \text{Hom}(X_k,X_j)$ and Hom$(X_i,X_j) \neq 0$. Let $f \in \text{Hom}(X_i, X_k), g \in \text{Hom}(X_k, X_j)$, and $h \in \text{Hom}(X_i, X_j)$. Since $\chi$ is an exceptional collection, by the Happel-Ringel lemma (Lemma \ref{lem: Happel-Ringel Lemma}), we know that each of the above morphisms are either injective or surjective. If both $f$ and $g$ are injective/surjective, then their composition is non-zero and hence $h$ is a {\color{black} multiple} of $g\circ f$. Similarly, if $f$ is surjective and $g$ is injective, we have that $h$ is a {\color{black} multiple} of $g\circ f$. In any of the above cases, when the arcs appear in this fashion, we have a non-zero path {\color{black} $X_i \overset{f}{\rightarrow} X_k \overset{g}{\rightarrow} X_j$ in $Q^{\chi}$}, and not an arrow. Finally, in the case that $f$ is injective and $g$ is surjective, then this forces $\gamma_{X_j}$ to be a boundary arc. Thus we can write $X_j$ as the string $\alpha_1 \alpha_2 \cdots \alpha_j$ with $\alpha_1$ and $\alpha_{j+1}$ both direct arrows and $s(\alpha_1) = p+1$. We get an admissible pair $(e_{s(\alpha_1)}, \alpha_1 \alpha_2 \cdots \alpha_j, \alpha_{j+1} \cdots \alpha_k),(e_{s(\alpha_1)},\alpha_1\alpha_2\cdots\alpha_j,e_{t(\alpha_j)})$ that gives the graph map $g$ from $X_k$ to $X_j$. Since $f$ is an injection and the string $X_i$ begins with $\alpha_1$, we have that the composition $g \circ f \neq 0$, and therefore we have a non-zero path {\color{black} $X_i \overset{f}{\rightarrow} X_k \overset{g}{\rightarrow} X_j$ in $Q^{\chi}$}, not an arrow.

{\color{black}
Suppose now that we are in the case of the diagram in the top right. Suppose first that the morphism $f$ from $X_k$ to $X_j$ is surjective. The arcs $X_k$ and $X_j$ must both have winding number 0, for if not they would cross. Let $\gamma_{X_i} = a(s_i,t_i)$ and similarly for $X_k$ and $X_j$. Since the morphism from $X_k$ to $X_j$ is surjective, the marked point $t_k$ must be clockwise of $t_j$. The middle term in the short exact sequence given by this configuration of $X_k$ and $X_i$ is arc $a(s_i,t_k)$. Similarly, the middle term in the short exact sequence given by this configuration of $X_i$ and $X_j$ is arc $a(s_i,t_j)$.
From the diagram, we see that there is a morphism $g$ from $a(s_i,t_k)$ to $a(s_i,t_j)$ that must be a surjection as $t_k$ is clockwise of $t_j$. Let $E = \alpha_1\alpha_2\cdots\alpha_k$ and $E' = \alpha_1\alpha_2\cdots\alpha_l$ be the two string modules associated to $a(s_i,t_k)$ and $a(s_i,t_j)$ respectively. We have that the kernel of the morphism $g$ is precisely $\alpha_{l+1} \cdots \alpha_k$, which is the kernel of $f$. Thus, we get the following commuting diagram:

\begin{center}
\begin{tabular}{c}
\begin{xymatrix}{0 \ar[r] \ar[d]_0 & X_k \ar[d]_f \ar[r] & E \ar[r] \ar[d]_g & X_i \ar[r] \ar[d]_{1} & 0 \ar[d]_0 \\ 0 \ar[r] & X_j \ar[r] & E' \ar[r] & X_i \ar[r] & 0  }\end{xymatrix}
\end{tabular}
\end{center}

Therefore, there is no arrow in $Q^{\chi}$ from $X_i$ to $X_j$. The case in which the morphism from $X_k$ to $X_j$ is injective is similar.
}

Conversely, suppose there is no arrow from $X_i$ to $X_j$ in $Q^{\chi}$. If $X_i$ and $X_j$ are Hom-Ext orthogonal, then $\gamma_{X_i}$ and $\gamma_{X_j}$ do not intersect at their endpoints, hence are not immediate successors. Suppose that $X_i \rightarrow X_{i+1} \rightarrow \cdots \rightarrow X_j$ is a non-zero path in $Q^{\chi}$. Then by the forward direction, $X_i$ is immediately clockwise of $X_{i+1}$, which is immediately clockwise of $X_{i+2}$, so on. We conclude that $X_i$ is not immediately clockwise of $X_j$.
\end{proof}

\begin{exmp}\label{example: immidiately clockwise iff arrow}
   {\color{black}Let $Q$ be the following quiver.}

\[
\xymatrix{
1 \ar_{\alpha_1}[r] \ar@/^2pc/^{\alpha_4}[rrr] & 2 \ar_{\alpha_2}[r] & 3 \ar_{\alpha_3}[r] & 4 
	}
\]

As in the proof of Lemma \ref{lem: arrow and immediately clockwise}, let $X_i = a(1,2)[0], X_k = a(2,4)[0],$ and $X_j = a(2,3)[0]$. Then the corresponding arc diagram is given by:

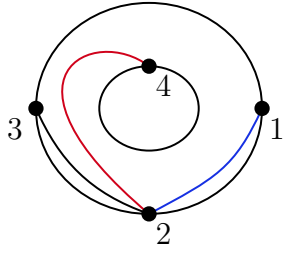
\begin{center}

\tikzset{every picture/.style={line width=0.75pt}} 

\begin{tikzpicture}[x=0.75pt,y=0.75pt,yscale=-1,xscale=1]

\draw   (279.91,124.75) .. controls (279.91,112.99) and (291.13,103.45) .. (304.98,103.45) .. controls (318.83,103.45) and (330.05,112.99) .. (330.05,124.75) .. controls (330.05,136.51) and (318.83,146.05) .. (304.98,146.05) .. controls (291.13,146.05) and (279.91,136.51) .. (279.91,124.75)(247.96,124.75) .. controls (247.96,95.34) and (273.49,71.5) .. (304.98,71.5) .. controls (336.47,71.5) and (362,95.34) .. (362,124.75) .. controls (362,154.16) and (336.47,178) .. (304.98,178) .. controls (273.49,178) and (247.96,154.16) .. (247.96,124.75) ;
\draw [color={rgb, 255:red, 208; green, 2; blue, 27 }  ,draw opacity=1 ]   (304.98,178) .. controls (222,108) and (277.47,82.17) .. (304.98,103.45) ;

\draw [color={rgb, 255:red, 0; green, 0; blue, 0 }  ,draw opacity=1 ]   (247.96,124.75) .. controls (259.79,151.22) and (275.12,167.54) .. (304.98,178) ;
\draw [color={rgb, 255:red, 22; green, 48; blue, 226 }  ,draw opacity=1 ]   (304.98,178) .. controls (338.76,160.01) and (348.19,154.98) .. (362,124.75) ;

\draw (364,128.15) node [anchor=north west][inner sep=0.75pt]    {$1$};
\draw (306.98,181.4) node [anchor=north west][inner sep=0.75pt]    {$2$};
\draw (232,128.15) node [anchor=north west][inner sep=0.75pt]    {$3$};
\draw (306.98,106.85) node [anchor=north west][inner sep=0.75pt]    {$4$};

\draw    (362,124.75) ;
\draw [shift={(362,124.75)}, rotate = 0] [color={rgb, 255:red, 0; green, 0; blue, 0 }  ][fill={rgb, 255:red, 0; green, 0; blue, 0 }  ][line width=0.75]      (0, 0) circle [x radius= 3.35, y radius= 3.35]   ;
\draw    (304.98,178) ;
\draw [shift={(304.98,178)}, rotate = 0] [color={rgb, 255:red, 0; green, 0; blue, 0 }  ][fill={rgb, 255:red, 0; green, 0; blue, 0 }  ][line width=0.75]      (0, 0) circle [x radius= 3.35, y radius= 3.35]   ;
\draw    (304.98,103.45) ;
\draw [shift={(304.98,103.45)}, rotate = 0] [color={rgb, 255:red, 0; green, 0; blue, 0 }  ][fill={rgb, 255:red, 0; green, 0; blue, 0 }  ][line width=0.75]      (0, 0) circle [x radius= 3.35, y radius= 3.35]   ;
\draw    (247.96,124.75) ;
\draw [shift={(247.96,124.75)}, rotate = 0] [color={rgb, 255:red, 0; green, 0; blue, 0 }  ][fill={rgb, 255:red, 0; green, 0; blue, 0 }  ][line width=0.75]      (0, 0) circle [x radius= 3.35, y radius= 3.35]   ;

\end{tikzpicture}

\end{center}

There is a quotient factorization $(e_3, e_3, \alpha_3)$ of the string $X_k$, which gives a surjective morphism onto $X_j$. Now, we get the following morphism between short exact sequences:

\begin{center}
\begin{tabular}{c}
\begin{xymatrix}{0 \ar[r] \ar[d]_0 & X_k \ar[d]_f \ar[r] & a(1,4)[0] \ar[r] \ar[d]_g & X_i \ar[r] \ar[d]_{1} & 0 \ar[d]_0 \\ 0 \ar[r] & X_j \ar[r] & a(1,3)[0] \ar[r] & X_i \ar[r] & 0  }\end{xymatrix}
\end{tabular}
\end{center}

In the corresponding arc diagram, the arcs $a(1,4)[0]$ and $a(1,3)[0]$ are depicted in orange and purple respectively: 

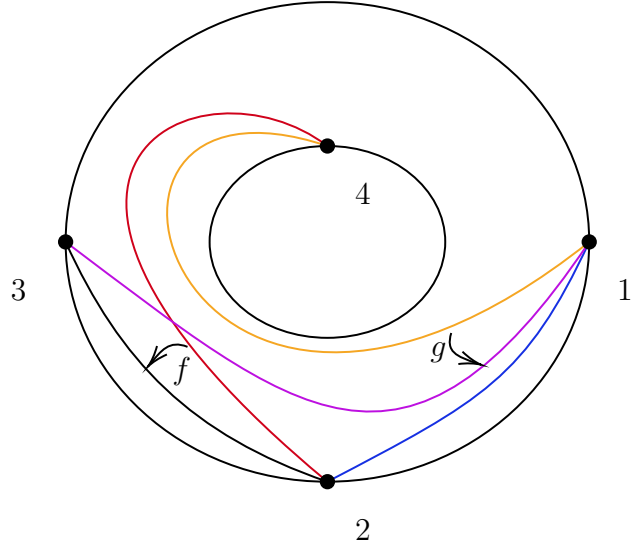
\begin{center}

\tikzset{every picture/.style={line width=0.75pt}} 

\begin{tikzpicture}[x=0.75pt,y=0.75pt,yscale=-1,xscale=1]

\draw   (252.14,139.03) .. controls (252.14,112.29) and (278.74,90.62) .. (311.56,90.62) .. controls (344.37,90.62) and (370.98,112.29) .. (370.98,139.03) .. controls (370.98,165.76) and (344.37,187.44) .. (311.56,187.44) .. controls (278.74,187.44) and (252.14,165.76) .. (252.14,139.03)(179.52,139.03) .. controls (179.52,72.18) and (238.64,18) .. (311.56,18) .. controls (384.48,18) and (443.59,72.18) .. (443.59,139.03) .. controls (443.59,205.87) and (384.48,260.05) .. (311.56,260.05) .. controls (238.64,260.05) and (179.52,205.87) .. (179.52,139.03) ;

\draw [color={rgb, 255:red, 208; green, 2; blue, 27 }  ,draw opacity=1 ]   (311.56,260.05) .. controls (119.42,100.96) and (247.86,42.25) .. (311.56,90.62) ;

\draw [color={rgb, 255:red, 0; green, 0; blue, 0 }  ,draw opacity=1 ]   (179.52,139.03) .. controls (206.93,199.18) and (242.41,236.27) .. (311.56,260.05) ;
\draw [color={rgb, 255:red, 22; green, 48; blue, 226 }  ,draw opacity=1 ]   (311.56,260.05) .. controls (389.78,219.15) and (411.61,207.74) .. (443.59,139.03) ;
\draw [color={rgb, 255:red, 245; green, 166; blue, 35 }  ,draw opacity=1 ]   (311.56,90.62) .. controls (158.78,39.59) and (233.52,305.95) .. (443.59,139.03) ;
\draw [color={rgb, 255:red, 189; green, 16; blue, 224 }  ,draw opacity=1 ]   (179.52,139.03) .. controls (295,226) and (350.64,277.48) .. (443.59,139.03) ;
\draw    (374,185) .. controls (371.15,194.5) and (376.42,196.78) .. (387.25,200.41) ;
\draw [shift={(389,201)}, rotate = 198.43] [color={rgb, 255:red, 0; green, 0; blue, 0 }  ][line width=0.75]    (10.93,-3.29) .. controls (6.95,-1.4) and (3.31,-0.3) .. (0,0) .. controls (3.31,0.3) and (6.95,1.4) .. (10.93,3.29)   ;
\draw    (241,192) .. controls (231.7,188.28) and (227.59,194.07) .. (222.23,200.53) ;
\draw [shift={(221,202)}, rotate = 310.6] [color={rgb, 255:red, 0; green, 0; blue, 0 }  ][line width=0.75]    (10.93,-3.29) .. controls (6.95,-1.4) and (3.31,-0.3) .. (0,0) .. controls (3.31,0.3) and (6.95,1.4) .. (10.93,3.29)   ;

\draw (456.12,156.43) node [anchor=north west][inner sep=0.75pt]    {$1$};
\draw (324.08,277.45) node [anchor=north west][inner sep=0.75pt]    {$2$};
\draw (150.47,156.43) node [anchor=north west][inner sep=0.75pt]    {$3$};
\draw (324.08,108.02) node [anchor=north west][inner sep=0.75pt]    {$4$};
\draw (362,188.4) node [anchor=north west][inner sep=0.75pt]    {$g$};
\draw (232,195.4) node [anchor=north west][inner sep=0.75pt]    {$f$};

\draw    (443.59,139.03) ;
\draw [shift={(443.59,139.03)}, rotate = 0] [color={rgb, 255:red, 0; green, 0; blue, 0 }  ][fill={rgb, 255:red, 0; green, 0; blue, 0 }  ][line width=0.75]      (0, 0) circle [x radius= 3.35, y radius= 3.35]   ;
\draw    (311.56,260.05) ;
\draw [shift={(311.56,260.05)}, rotate = 0] [color={rgb, 255:red, 0; green, 0; blue, 0 }  ][fill={rgb, 255:red, 0; green, 0; blue, 0 }  ][line width=0.75]      (0, 0) circle [x radius= 3.35, y radius= 3.35]   ;
\draw    (311.56,90.62) ;
\draw [shift={(311.56,90.62)}, rotate = 0] [color={rgb, 255:red, 0; green, 0; blue, 0 }  ][fill={rgb, 255:red, 0; green, 0; blue, 0 }  ][line width=0.75]      (0, 0) circle [x radius= 3.35, y radius= 3.35]   ;
\draw    (179.52,139.03) ;
\draw [shift={(179.52,139.03)}, rotate = 0] [color={rgb, 255:red, 0; green, 0; blue, 0 }  ][fill={rgb, 255:red, 0; green, 0; blue, 0 }  ][line width=0.75]      (0, 0) circle [x radius= 3.35, y radius= 3.35]   ;

\end{tikzpicture}

\end{center}
The morphism $g: a(1,4)[0] \rightarrow a(1,3)[0]$ can be seen directly from the arc diagram, however, it is also given by the quotient factorization $a(1,4)[0] = (e_2,\alpha_2, \alpha_3)$.
\end{exmp}

\begin{lem}\label{lem: two distinct paths}
Let $\chi = \{X_1, X_2, \dots, X_n\}$ be an exceptional collection. For $i,j\in\{1,2,\dots,n\}$, there are at most two non-zero paths from $X_i$ to $X_j$ in $(Q^{\chi},R)$, {\color{black} and if there are two paths}, their difference is not in $R$. 
\end{lem}

\begin{proof}
In the case $i=j$, there is only the trivial path, so assume $i \neq j$ and that there is a non-zero path from $X_i$ to $X_j$ in $Q^{\chi}$. Then this path tells us that that there is a non-trivial morphism from $X_i$ to $X_j$ or extension of $X_i$ by $X_j$. In either case, we have that in the associated diagram $D_{\chi}$, the arcs $\gamma_{X_i}$ and $\gamma_{X_j}$ must intersect at their endpoints. Since there are two ways for this to happen, there can be at most two non-zero paths from $X_i$ to $X_j$ in $Q^{\chi}$. {\color{black} If there are two non-zero paths from $X_i$ to $X_j$ in $Q^{\chi}$, then up to twist, the arcs corresponding to $X_i$ and $X_j$ appear as follows:}

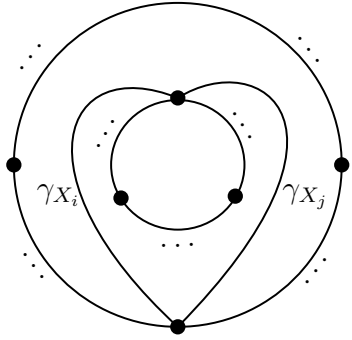
\begin{center}

\tikzset{every picture/.style={line width=0.75pt}} 

\begin{tikzpicture}[x=0.75pt,y=0.75pt,yscale=-1,xscale=1]

\draw   (284.01,99.25) .. controls (284.01,81.19) and (299.06,66.55) .. (317.64,66.55) .. controls (336.21,66.55) and (351.27,81.19) .. (351.27,99.25) .. controls (351.27,117.31) and (336.21,131.95) .. (317.64,131.95) .. controls (299.06,131.95) and (284.01,117.31) .. (284.01,99.25)(234.96,99.25) .. controls (234.96,54.1) and (271.97,17.5) .. (317.64,17.5) .. controls (363.3,17.5) and (400.32,54.1) .. (400.32,99.25) .. controls (400.32,144.4) and (363.3,181) .. (317.64,181) .. controls (271.97,181) and (234.96,144.4) .. (234.96,99.25) ;
\draw    (317.64,65.46) ;
\draw [shift={(317.64,65.46)}, rotate = 0] [color={rgb, 255:red, 0; green, 0; blue, 0 }  ][fill={rgb, 255:red, 0; green, 0; blue, 0 }  ][line width=0.75]      (0, 0) circle [x radius= 3.35, y radius= 3.35]   ;
\draw    (400.32,99.25) ;
\draw [shift={(400.32,99.25)}, rotate = 0] [color={rgb, 255:red, 0; green, 0; blue, 0 }  ][fill={rgb, 255:red, 0; green, 0; blue, 0 }  ][line width=0.75]      (0, 0) circle [x radius= 3.35, y radius= 3.35]   ;
\draw    (346.64,115.04) ;
\draw [shift={(346.64,115.04)}, rotate = 0] [color={rgb, 255:red, 0; green, 0; blue, 0 }  ][fill={rgb, 255:red, 0; green, 0; blue, 0 }  ][line width=0.75]      (0, 0) circle [x radius= 3.35, y radius= 3.35]   ;
\draw    (234.96,99.25) ;
\draw [shift={(234.96,99.25)}, rotate = 0] [color={rgb, 255:red, 0; green, 0; blue, 0 }  ][fill={rgb, 255:red, 0; green, 0; blue, 0 }  ][line width=0.75]      (0, 0) circle [x radius= 3.35, y radius= 3.35]   ;
\draw    (317.64,181) ;
\draw [shift={(317.64,181)}, rotate = 0] [color={rgb, 255:red, 0; green, 0; blue, 0 }  ][fill={rgb, 255:red, 0; green, 0; blue, 0 }  ][line width=0.75]      (0, 0) circle [x radius= 3.35, y radius= 3.35]   ;
\draw    (289,116) ;
\draw [shift={(289,116)}, rotate = 0] [color={rgb, 255:red, 0; green, 0; blue, 0 }  ][fill={rgb, 255:red, 0; green, 0; blue, 0 }  ][line width=0.75]      (0, 0) circle [x radius= 3.35, y radius= 3.35]   ;
\draw    (317.64,181) .. controls (235,105) and (259,42) .. (317.64,65.46) ;
\draw    (317.64,65.46) .. controls (360,39) and (417,80) .. (317.64,181) ;

\draw (378.97,159.47) node [anchor=north west][inner sep=0.75pt]  [rotate=-308.4]  {$\dotsc $};
\draw (240.42,140.03) node [anchor=north west][inner sep=0.75pt]  [rotate=-53.65]  {$\dotsc $};
\draw (346.25,68.96) node [anchor=north west][inner sep=0.75pt]  [rotate=-56.83]  {$\dotsc $};
\draw (306.86,137.65) node [anchor=north west][inner sep=0.75pt]  [rotate=-358.9]  {$\dotsc $};
\draw (273.95,91.51) node [anchor=north west][inner sep=0.75pt]  [rotate=-299.79]  {$\dotsc $};
\draw (245,107.4) node [anchor=north west][inner sep=0.75pt]    {$\gamma_{X_{i}}$};
\draw (369,107.4) node [anchor=north west][inner sep=0.75pt]    {$\gamma_{X_{j}}$};
\draw (234.96,51.45) node [anchor=north west][inner sep=0.75pt]  [rotate=-308.48]  {$\dotsc $};
\draw (378.75,31.11) node [anchor=north west][inner sep=0.75pt]  [rotate=-51.09]  {$\dotsc $};

\end{tikzpicture}

\end{center}

In this case, we have a two-dimensional Hom or Ext space, and hence the two paths give different morphisms or extensions.
\end{proof}
{\color{black} 
Combining the results so far in this section, we see that the Hom-Ext quiver associated to an exceptional collection in type $\tilde{\mathbb{A}}$ behaves very similarly to the tiling algebra associated to the corresponding exceptional arc diagram. A more general definition can be found as Definition 2.1 in \cite{baur2021geometric}; however, since here we are dealing with exceptional collections and there are no loops, we can make the following definition.

\begin{defn}
    The \textbf{tiling algebra} associated to the exceptional arc diagram $D_{\chi}$ is the algebra $T_{\chi} = \Bbbk Q_{\chi}/R_{\chi}$, where $(Q_{\chi},R_{\chi})$ is defined as follows:
    \begin{enumerate}
        \item The vertices are in bijection with the arcs in $\chi$.
        \item There is an arrow $\alpha:\gamma_{X_i} \rightarrow \gamma_{X_j}$ if the arcs $\gamma_{X_i}$ and $\gamma_{X_j}$ share an endpoint $p$ and $\gamma_{X_j}$ is immediately counterclockwise of $\gamma_{X_i}$.
        \item $R_{\chi}$ is generated by paths $\gamma_{X_i} \overset{\alpha}{\rightarrow} \gamma_{X_j} \overset{\beta}{\rightarrow} \gamma_{X_k}$ of length two such that the shared endpoint of $\gamma_{X_i}$ and $\gamma_{X_j}$ is different from the shared endpoint of $\gamma_{X_j}$ and $\gamma_{X_k}$. 
    \end{enumerate}
\end{defn}
} 

\begin{prop}\label{prop: H-E quiver is tiling algebra}
{\color{black} Let $\chi$ be an exceptional set}. We have an isomorphism of algebras $T_{\chi} \cong \Bbbk Q^{\chi}/R$.
\end{prop}
{\color{black}
\begin{proof}
We have a bijection between $(Q_{\chi})_0$ and $Q^{\chi}_0$. By Lemma \ref{lem: arrow and immediately clockwise}, we have an order-preserving bijection between the arrows of $Q_{\chi}$ and those of $Q^{\chi}$. To finish the proof, by Lemma \ref{lem: two distinct paths}, it suffices to show that both $Q_{\chi}$ and $Q^{\chi}$ have the same monomial relations; but, this is implied by the fact that both $Q_{\chi}$ and $Q^{\chi}$ have a monomial relation $\alpha_1\alpha_2$ if and only {\color{black} if there is a path $\gamma_1 \overset{\alpha_1}{\rightarrow} \gamma_2 \overset{\alpha_1}{\rightarrow} \gamma_3$  such that $\gamma_1$ and $\gamma_2$ share an endpoint $p$ that $\gamma_2$ and $\gamma_3$ don't.} 
\end{proof}
}
{\color{black} In Definition 2.1 in \cite{baur2021geometric}, when defining the tiling algebra, it was assumed that the collection of arcs on the surface formed a partial triangulation. That is, the arcs don't intersect themselves or each other in the interior of the surface. In light of Proposition \ref{prop: H-E quiver is tiling algebra}, it may be the case that Hom-Ext quivers are a generalization of tiling algebras to sets of arcs that do cross themselves. Moreover, we wonder:}

\begin{question}
Can a similar result be attained for other surfaces? We believe this result holds in type $\mathbb{A}$ as stated, however, we will likely need looser hypotheses on $\chi$ in non-hereditary cases.
\end{question}

\begin{cor}\label{cor: Hom-Ext quiver is gentle}
$\Bbbk Q^{\chi}/R$ is a gentle algebra.
\end{cor}

\begin{proof}
This follows from Proposition \ref{prop: H-E quiver is tiling algebra} along with Theorem 2.10 in \cite{baur2021geometric}.
\end{proof}

\begin{question}
Corollary \ref{cor: Hom-Ext quiver is gentle} shows that when we start with a quiver of type $\tilde{\mathbb{A}}$, the relations of the Hom-Ext quiver of an exceptional collection are generated by monomials of length 2. We believe this this true for gentle algebras in general. For an arbitrary finite acyclic quiver $Q$, are  monomial relations enough, or are there more? 
\end{question}

Now that we know the Hom-Ext quiver of an exceptional collection is encoded entirely in its arc diagram, we can use the arc diagrams to study Hom-Ext quivers. Namely, what we wish to do for the remainder of this section is classify exceptional sets in terms of {\color{black} isoclasses} of their Hom-Ext quivers. To do this, we will begin with some definitions and a pair of lemmas.

\begin{defn}\label{defn: heart}
For an exceptional collection $\chi$, there is a subset $\mathcal{H} \subset \chi$ such that $D_{\mathcal{H}}$ encloses the inner boundary of $A_Q$. The \textbf{heart} of $\chi$, denoted by $\mathcal{H}'$, is the smallest such subset. The \textbf{extended heart} of $\chi$, denoted by $\mathcal{H}\subset \chi$ is the collection of arcs in $\mathcal{H}'$ along with the set of arcs in $D_{\chi}$ through which each morphism/extension between modules in $\mathcal{H'}$ factors. 

{\color{black} Let $(Q^{\chi_1},R_1)$ and $(Q^{\chi_2},R_2)$ be two Hom-Ext quivers associated to two exceptional collections $\chi_1$ and $\chi_2$ respectively. Let $\varphi: (Q^{\chi_1},R_1) \rightarrow (Q^{\chi_2},R_2)$ be an isomorphism of quivers with relations. We call $\varphi$ an \textbf{isomorphism of Hom-Ext quivers} if $\varphi$ preserves the component of the AR quiver in which each module resides. That is, if $X_i$ is non-regular then $\varphi(X_i)$ is non-regular. Analogous statements hold for left and right regular modules. When there is an isomorphism between $(Q^{\chi_1},R_1)$ and $(Q^{\chi_2},R_2)$, we say that the two Hom-Ext quivers are \textbf{isomorphic}, denoted $(Q^{\chi_1},R_1) \cong (Q^{\chi_2},R_2)$. }
\end{defn}

\begin{rem}\label{rem: heart well-defined}
Since $\chi$ is a complete  exceptional collection, there is one and only one heart. It is of the form $\mathcal{H}' := \{ \gamma_1, \gamma_2, \dots, \gamma_k\}$ such that for all $i \in \{1,2,\dots,k\}$, we have that $\gamma_i$ ends where $\gamma_{i+1}$ starts ($\gamma_{k+1} := \gamma_1$). In turn, $\chi$ has one and only one extended heart $\mathcal{H}$. If there was no such $\mathcal{H}'$, $\chi$ wouldn't be complete. If there was another collection of arcs that enclose the inner boundary, then these arcs would cross those in $\mathcal{H}'$ and $\chi$ would not be exceptional. 
\end{rem}

\begin{rem}
The reason the set $\mathcal{H}'$ is called a heart is because the modules in this set form the heart of a $t$-structure on a subcategory of $\mathscr{D}^b(\text{rep}\Bbbk Q)$ that is of type $\tilde{\mathbb{A}}$. In a future paper, we will prove this, and moreover, use it to define hearts of exceptional sets in general. It is also nice that the set $\mathcal{H}'$ often looks like a heart when drawn on the annulus! 
\end{rem}

\begin{exmp}
Let $Q$ be the quiver from Example \ref{exam: complete fan}. Then we have two exceptional collections given by the following arc diagrams:

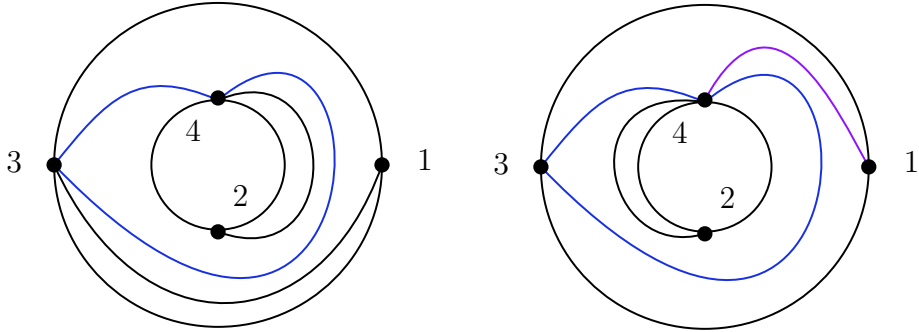
\begin{center}

\tikzset{every picture/.style={line width=0.75pt}} 

\begin{tikzpicture}[x=0.75pt,y=0.75pt,yscale=-1,xscale=1]

\draw   (96.01,99.25) .. controls (96.01,81.19) and (111.06,66.55) .. (129.64,66.55) .. controls (148.21,66.55) and (163.27,81.19) .. (163.27,99.25) .. controls (163.27,117.31) and (148.21,131.95) .. (129.64,131.95) .. controls (111.06,131.95) and (96.01,117.31) .. (96.01,99.25)(46.96,99.25) .. controls (46.96,54.1) and (83.97,17.5) .. (129.64,17.5) .. controls (175.3,17.5) and (212.32,54.1) .. (212.32,99.25) .. controls (212.32,144.4) and (175.3,181) .. (129.64,181) .. controls (83.97,181) and (46.96,144.4) .. (46.96,99.25) ;

\draw [color={rgb, 255:red, 22; green, 48; blue, 226 }  ,draw opacity=1 ]   (129.64,66.55) .. controls (222.21,-11) and (213.22,272.5) .. (46.96,99.25) ;
\draw [color={rgb, 255:red, 22; green, 48; blue, 226 }  ,draw opacity=1 ]   (46.96,99.25) .. controls (73.92,65.5) and (88.9,49) .. (129.64,66.55) ;
\draw [color={rgb, 255:red, 0; green, 0; blue, 0 }  ,draw opacity=1 ]   (129.64,66.55) .. controls (192.25,40) and (195.24,157) .. (129.64,133.04) ;
\draw [color={rgb, 255:red, 0; green, 0; blue, 0 }  ,draw opacity=1 ]   (46.96,99.25) .. controls (90.39,194.5) and (180.27,190) .. (212.32,99.25) ;
\draw   (341.52,100.25) .. controls (341.52,82.19) and (356.58,67.55) .. (375.15,67.55) .. controls (393.73,67.55) and (408.78,82.19) .. (408.78,100.25) .. controls (408.78,118.31) and (393.73,132.95) .. (375.15,132.95) .. controls (356.58,132.95) and (341.52,118.31) .. (341.52,100.25)(292.47,100.25) .. controls (292.47,55.1) and (329.49,18.5) .. (375.15,18.5) .. controls (420.82,18.5) and (457.83,55.1) .. (457.83,100.25) .. controls (457.83,145.4) and (420.82,182) .. (375.15,182) .. controls (329.49,182) and (292.47,145.4) .. (292.47,100.25) ;

\draw [color={rgb, 255:red, 22; green, 48; blue, 226 }  ,draw opacity=1 ]   (375.15,67.55) .. controls (467.72,-10) and (458.73,273.5) .. (292.47,100.25) ;
\draw [color={rgb, 255:red, 22; green, 48; blue, 226 }  ,draw opacity=1 ]   (292.47,100.25) .. controls (319.43,66.5) and (334.41,50) .. (375.15,67.55) ;
\draw [color={rgb, 255:red, 144; green, 19; blue, 254 }  ,draw opacity=1 ]   (375.15,66.46) .. controls (410,2) and (442,72) .. (457.83,100.25) ;
\draw [color={rgb, 255:red, 0; green, 0; blue, 0 }  ,draw opacity=1 ]   (375.15,67.55) .. controls (305,57) and (324,149) .. (375.15,134.04) ;

\draw (228.52,89.83) node [anchor=north west][inner sep=0.75pt]    {$1$};
\draw (135.62,108.15) node [anchor=north west][inner sep=0.75pt]    {$2$};
\draw (111.66,75.45) node [anchor=north west][inner sep=0.75pt]    {$4$};
\draw (21.48,90.9) node [anchor=north west][inner sep=0.75pt]    {$3$};
\draw (474.03,90.83) node [anchor=north west][inner sep=0.75pt]    {$1$};
\draw (381.14,109.15) node [anchor=north west][inner sep=0.75pt]    {$2$};
\draw (357.17,76.45) node [anchor=north west][inner sep=0.75pt]    {$4$};
\draw (267,91.9) node [anchor=north west][inner sep=0.75pt]    {$3$};

\draw    (129.64,65.46) ;
\draw [shift={(129.64,65.46)}, rotate = 0] [color={rgb, 255:red, 0; green, 0; blue, 0 }  ][fill={rgb, 255:red, 0; green, 0; blue, 0 }  ][line width=0.75]      (0, 0) circle [x radius= 3.35, y radius= 3.35]   ;
\draw    (212.32,99.25) ;
\draw [shift={(212.32,99.25)}, rotate = 0] [color={rgb, 255:red, 0; green, 0; blue, 0 }  ][fill={rgb, 255:red, 0; green, 0; blue, 0 }  ][line width=0.75]      (0, 0) circle [x radius= 3.35, y radius= 3.35]   ;
\draw    (129.64,133.04) ;
\draw [shift={(129.64,133.04)}, rotate = 0] [color={rgb, 255:red, 0; green, 0; blue, 0 }  ][fill={rgb, 255:red, 0; green, 0; blue, 0 }  ][line width=0.75]      (0, 0) circle [x radius= 3.35, y radius= 3.35]   ;
\draw    (46.96,99.25) ;
\draw [shift={(46.96,99.25)}, rotate = 0] [color={rgb, 255:red, 0; green, 0; blue, 0 }  ][fill={rgb, 255:red, 0; green, 0; blue, 0 }  ][line width=0.75]      (0, 0) circle [x radius= 3.35, y radius= 3.35]   ;
\draw    (375.15,66.46) ;
\draw [shift={(375.15,66.46)}, rotate = 0] [color={rgb, 255:red, 0; green, 0; blue, 0 }  ][fill={rgb, 255:red, 0; green, 0; blue, 0 }  ][line width=0.75]      (0, 0) circle [x radius= 3.35, y radius= 3.35]   ;
\draw    (457.83,100.25) ;
\draw [shift={(457.83,100.25)}, rotate = 0] [color={rgb, 255:red, 0; green, 0; blue, 0 }  ][fill={rgb, 255:red, 0; green, 0; blue, 0 }  ][line width=0.75]      (0, 0) circle [x radius= 3.35, y radius= 3.35]   ;
\draw    (375.15,134.04) ;
\draw [shift={(375.15,134.04)}, rotate = 0] [color={rgb, 255:red, 0; green, 0; blue, 0 }  ][fill={rgb, 255:red, 0; green, 0; blue, 0 }  ][line width=0.75]      (0, 0) circle [x radius= 3.35, y radius= 3.35]   ;
\draw    (292.47,100.25) ;
\draw [shift={(292.47,100.25)}, rotate = 0] [color={rgb, 255:red, 0; green, 0; blue, 0 }  ][fill={rgb, 255:red, 0; green, 0; blue, 0 }  ][line width=0.75]      (0, 0) circle [x radius= 3.35, y radius= 3.35]   ;

\end{tikzpicture}
\end{center}
\vspace{-2cm}
In the above exceptional sets, the blue arcs are in the heart $\mathcal{H}'$, and the purple arc is in the extended heart $\mathcal{H}$. The black arcs are in neither $\mathcal{H}$ nor $\mathcal{H}'$.
\end{exmp}

\begin{lem} \label{lem: affine A subquiver with non-regular source} 
$Q^{\mathcal{H}}$ is an acyclic quiver of type $\tilde{\mathbb{A}}$ in which all sources are non-regular modules.
\end{lem}

\begin{proof}
Since the arcs in $\mathcal{H}'$ enclose the inner boundary, the underlying graph of $Q^{\mathcal{H}'}$ is a band. Since $\chi$ is exceptional, by Lemma \ref{lem: Hom-Ext quiver acyclic}, we have that $Q^{\mathcal{H}'}$ is acyclic. Since the morphisms and extensions in $Q^{\mathcal{H}'}$ factor through the rest of the modules in $\mathcal{H}$, we have that $Q^{\mathcal{H}}$ is also an acyclic band. Thus $Q^{\mathcal{H}}$ is of type $\tilde{\mathbb{A}}$. 

Suppose now that $X \in \mathcal{H}$ is a source in $Q^{\mathcal{H}}$ and that $X$ is regular. Then in $Q^{\mathcal{H}}$, there are two arrows out of $X$ (the black arc in the diagrams below) and the diagram locally and up to homotopy, looks as one of the following:

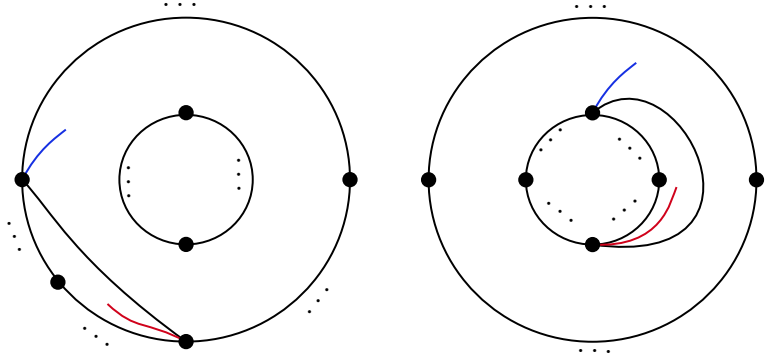
\begin{center}

\tikzset{every picture/.style={line width=0.75pt}} 

\begin{tikzpicture}[x=0.75pt,y=0.75pt,yscale=-1,xscale=1]

\draw   (95.01,100.25) .. controls (95.01,82.19) and (110.06,67.55) .. (128.64,67.55) .. controls (147.21,67.55) and (162.27,82.19) .. (162.27,100.25) .. controls (162.27,118.31) and (147.21,132.95) .. (128.64,132.95) .. controls (110.06,132.95) and (95.01,118.31) .. (95.01,100.25)(45.96,100.25) .. controls (45.96,55.1) and (82.97,18.5) .. (128.64,18.5) .. controls (174.3,18.5) and (211.32,55.1) .. (211.32,100.25) .. controls (211.32,145.4) and (174.3,182) .. (128.64,182) .. controls (82.97,182) and (45.96,145.4) .. (45.96,100.25) ;

\draw    (45.96,100.25) .. controls (61,116) and (72,141) .. (128.64,182) ;
\draw [color={rgb, 255:red, 208; green, 2; blue, 27 }  ,draw opacity=1 ]   (89,163) .. controls (104,177) and (108,171) .. (128.64,182) ;
\draw [color={rgb, 255:red, 22; green, 48; blue, 226 }  ,draw opacity=1 ]   (45.96,100.25) .. controls (54,85) and (63,79) .. (68,75) ;
\draw   (300.01,100.33) .. controls (300.01,82.27) and (315.06,67.63) .. (333.64,67.63) .. controls (352.21,67.63) and (367.27,82.27) .. (367.27,100.33) .. controls (367.27,118.39) and (352.21,133.03) .. (333.64,133.03) .. controls (315.06,133.03) and (300.01,118.39) .. (300.01,100.33)(250.96,100.33) .. controls (250.96,55.18) and (287.97,18.58) .. (333.64,18.58) .. controls (379.3,18.58) and (416.32,55.18) .. (416.32,100.33) .. controls (416.32,145.48) and (379.3,182.08) .. (333.64,182.08) .. controls (287.97,182.08) and (250.96,145.48) .. (250.96,100.33) ;

\draw    (333.64,66.54) .. controls (377,30) and (434,147) .. (333.64,133.03) ;
\draw [color={rgb, 255:red, 208; green, 2; blue, 27 }  ,draw opacity=1 ]   (333.64,133.03) .. controls (369,135) and (373,111) .. (376,104) ;
\draw [color={rgb, 255:red, 22; green, 48; blue, 226 }  ,draw opacity=1 ]   (333.64,66.54) .. controls (341.68,51.29) and (350.68,45.29) .. (355.68,41.29) ;

\draw (186.97,168.47) node [anchor=north west][inner sep=0.75pt]  [rotate=-308.4]  {$\dotsc $};
\draw (76.08,171.15) node [anchor=north west][inner sep=0.75pt]  [rotate=-36.54]  {$\dotsc $};
\draw (157.43,86.9) node [anchor=north west][inner sep=0.75pt]  [rotate=-89.26]  {$\dotsc $};
\draw (40.12,118.09) node [anchor=north west][inner sep=0.75pt]  [rotate=-67.78]  {$\dotsc $};
\draw (97.6,112.11) node [anchor=north west][inner sep=0.75pt]  [rotate=-269.13]  {$\dotsc $};
\draw (114.99,9.41) node [anchor=north west][inner sep=0.75pt]  [rotate=-359.95]  {$\dotsc $};
\draw (322.06,10.37) node [anchor=north west][inner sep=0.75pt]  [rotate=-0.47]  {$\dotsc $};
\draw (310.08,108.23) node [anchor=north west][inner sep=0.75pt]  [rotate=-36.54]  {$\dotsc $};
\draw (340.16,120.61) node [anchor=north west][inner sep=0.75pt]  [rotate=-319.99]  {$\dotsc $};
\draw (324.08,184.34) node [anchor=north west][inner sep=0.75pt]  [rotate=-0.59]  {$\dotsc $};
\draw (358.68,93.6) node [anchor=north west][inner sep=0.75pt]  [rotate=-223.21]  {$\dotsc $};
\draw (303.47,86.04) node [anchor=north west][inner sep=0.75pt]  [rotate=-314.45]  {$\dotsc $};

\draw    (128.64,66.46) ;
\draw [shift={(128.64,66.46)}, rotate = 0] [color={rgb, 255:red, 0; green, 0; blue, 0 }  ][fill={rgb, 255:red, 0; green, 0; blue, 0 }  ][line width=0.75]      (0, 0) circle [x radius= 3.35, y radius= 3.35]   ;
\draw    (211.32,100.25) ;
\draw [shift={(211.32,100.25)}, rotate = 0] [color={rgb, 255:red, 0; green, 0; blue, 0 }  ][fill={rgb, 255:red, 0; green, 0; blue, 0 }  ][line width=0.75]      (0, 0) circle [x radius= 3.35, y radius= 3.35]   ;
\draw    (128.64,132.95) ;
\draw [shift={(128.64,132.95)}, rotate = 0] [color={rgb, 255:red, 0; green, 0; blue, 0 }  ][fill={rgb, 255:red, 0; green, 0; blue, 0 }  ][line width=0.75]      (0, 0) circle [x radius= 3.35, y radius= 3.35]   ;
\draw    (45.96,100.25) ;
\draw [shift={(45.96,100.25)}, rotate = 0] [color={rgb, 255:red, 0; green, 0; blue, 0 }  ][fill={rgb, 255:red, 0; green, 0; blue, 0 }  ][line width=0.75]      (0, 0) circle [x radius= 3.35, y radius= 3.35]   ;
\draw    (128.64,182) ;
\draw [shift={(128.64,182)}, rotate = 0] [color={rgb, 255:red, 0; green, 0; blue, 0 }  ][fill={rgb, 255:red, 0; green, 0; blue, 0 }  ][line width=0.75]      (0, 0) circle [x radius= 3.35, y radius= 3.35]   ;
\draw    (64,152) ;
\draw [shift={(64,152)}, rotate = 0] [color={rgb, 255:red, 0; green, 0; blue, 0 }  ][fill={rgb, 255:red, 0; green, 0; blue, 0 }  ][line width=0.75]      (0, 0) circle [x radius= 3.35, y radius= 3.35]   ;

\draw    (333.64,66.54) ;
\draw [shift={(333.64,66.54)}, rotate = 0] [color={rgb, 255:red, 0; green, 0; blue, 0 }  ][fill={rgb, 255:red, 0; green, 0; blue, 0 }  ][line width=0.75]      (0, 0) circle [x radius= 3.35, y radius= 3.35]   ;
\draw    (416.32,100.33) ;
\draw [shift={(416.32,100.33)}, rotate = 0] [color={rgb, 255:red, 0; green, 0; blue, 0 }  ][fill={rgb, 255:red, 0; green, 0; blue, 0 }  ][line width=0.75]      (0, 0) circle [x radius= 3.35, y radius= 3.35]   ;
\draw    (333.64,133.03) ;
\draw [shift={(333.64,133.03)}, rotate = 0] [color={rgb, 255:red, 0; green, 0; blue, 0 }  ][fill={rgb, 255:red, 0; green, 0; blue, 0 }  ][line width=0.75]      (0, 0) circle [x radius= 3.35, y radius= 3.35]   ;
\draw    (250.96,100.33) ;
\draw [shift={(250.96,100.33)}, rotate = 0] [color={rgb, 255:red, 0; green, 0; blue, 0 }  ][fill={rgb, 255:red, 0; green, 0; blue, 0 }  ][line width=0.75]      (0, 0) circle [x radius= 3.35, y radius= 3.35]   ;
\draw    (300.01,100.33) ;
\draw [shift={(300.01,100.33)}, rotate = 0] [color={rgb, 255:red, 0; green, 0; blue, 0 }  ][fill={rgb, 255:red, 0; green, 0; blue, 0 }  ][line width=0.75]      (0, 0) circle [x radius= 3.35, y radius= 3.35]   ;
\draw    (367.27,100.33) ;
\draw [shift={(367.27,100.33)}, rotate = 0] [color={rgb, 255:red, 0; green, 0; blue, 0 }  ][fill={rgb, 255:red, 0; green, 0; blue, 0 }  ][line width=0.75]      (0, 0) circle [x radius= 3.35, y radius= 3.35]   ;

\end{tikzpicture}

\end{center}

From this we see that by Definition \ref{defn: heart}, the red arcs are not in $\mathcal{H}$, contradicting the fact that $X$ is a source {\color{black} in $Q^{\mathcal{H}}$}.  
\end{proof}

\begin{rem}\label{rem: every Hom-Ext quiver has a cycle}
Let $\gamma$ be a collection of non-self-intersecting arcs on $A_Q$. Then the underlying graph of $Q^{\gamma}$ is not a tree if and only if there are arcs $\gamma_1,\gamma_2, \dots, \gamma_j \in \gamma$ such that for all $i \in \{1,2,\dots,j\}$, we have that $\gamma_i$ ends where $\gamma_{i+1}$ starts ($\gamma_{j+1} := \gamma_1$). When we take $\chi$ to be exceptional, every such set of arcs in the arc diagram must enclose the inner boundary, for if not, there would be a cycle. Moreover, by Remark \ref{rem: heart well-defined}, we know that there can only be one such set of arcs. We conclude that for an exceptional collection $\chi$, the Hom-Ext quiver $Q^{\chi}$ is a gentle quiver with \textit{precisely} one acyclic band that corresponds to the {\color{black} extended} heart $\mathcal{H} \subset \chi$ (Definition \ref{defn: heart}).   
\end{rem}

The next important lemma is quite surprising and speaks to just how powerful knowing the Hom-Ext quiver is.

\begin{lem} \label{lem: Hom-Ext quiver is uniquely determined by one non-regular} 
Let {\color{black}$(Q',R)$} be a Hom-Ext quiver that can be associated to an exceptional collection of {\color{black} representations of} $Q^{\bm{\varepsilon}}$. Fix a labeling of the vertices of $Q'$ from the set $\{1,2,\dots, n\}$. Then there is an $i \in \{1,2,\dots, n\}$ such that the module at vertex $i$ must be non-regular {\color{black} and a potentially empty subset $\{j_1, j_2, \dots, j_k\}\subset \{1,2,\dots, n\}$ such that the modules at these vertices are regular.} Once we choose this non-regular module $X_i$ at vertex $i$ {\color{black} and the regular modules $X_{j_1}, \dots, X_{j_k}$ at vertices $j_1, \dots, j_k$,} there is a unique way to associate modules to the remaining vertices of $Q'$.
\end{lem}

\begin{proof}
 The existence of a non-regular module follows from the fact that any exceptional arc diagram contains at least two bridging arcs. To prove the uniqueness part, we will proceed by {\color{black} double induction. We will first induct on the number of vertices on the outer boundary of the annulus, then on the number of vertices on the inner boundary of the annulus.}

{\color{black} Suppose that there is one vertex on the inner boundary and one on the outer.} By Remark \ref{rem: every Hom-Ext quiver has a cycle}, we are studying exceptional collections of the Kronecker quiver: $1 \rightrightarrows 2$ up to duality. By Theorem 6.4 in \cite{igusa2024clusters}, up to $2\pi$ Dehn twist, there is one such exceptional collection:

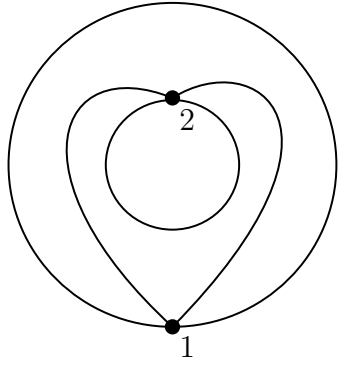
\begin{center}
\tikzset{every picture/.style={line width=0.75pt}} 
\begin{tikzpicture}[x=0.75pt,y=0.75pt,yscale=-1,xscale=1]

\draw   (284.01,99.25) .. controls (284.01,81.19) and (299.06,66.55) .. (317.64,66.55) .. controls (336.21,66.55) and (351.27,81.19) .. (351.27,99.25) .. controls (351.27,117.31) and (336.21,131.95) .. (317.64,131.95) .. controls (299.06,131.95) and (284.01,117.31) .. (284.01,99.25)(234.96,99.25) .. controls (234.96,54.1) and (271.97,17.5) .. (317.64,17.5) .. controls (363.3,17.5) and (400.32,54.1) .. (400.32,99.25) .. controls (400.32,144.4) and (363.3,181) .. (317.64,181) .. controls (271.97,181) and (234.96,144.4) .. (234.96,99.25) ;
\draw    (317.64,65.46) ;
\draw [shift={(317.64,65.46)}, rotate = 0] [color={rgb, 255:red, 0; green, 0; blue, 0 }  ][fill={rgb, 255:red, 0; green, 0; blue, 0 }  ][line width=0.75]      (0, 0) circle [x radius= 3.35, y radius= 3.35]   ;
\draw    (317.64,181) ;
\draw [shift={(317.64,181)}, rotate = 0] [color={rgb, 255:red, 0; green, 0; blue, 0 }  ][fill={rgb, 255:red, 0; green, 0; blue, 0 }  ][line width=0.75]      (0, 0) circle [x radius= 3.35, y radius= 3.35]   ;
\draw    (317.64,181) .. controls (235,105) and (259,42) .. (317.64,65.46) ;
\draw    (317.64,65.46) .. controls (360,39) and (417,80) .. (317.64,181) ;

\draw (319.64,184.4) node [anchor=north west][inner sep=0.75pt]    {$1$};
\draw (319.64,69.95) node [anchor=north west][inner sep=0.75pt]    {$2$};
\end{tikzpicture}
\end{center}

In turn, since Dehn twists preserve the corresponding tiling algebra and hence the corresponding Hom-Ext quiver by Proposition \ref{prop: H-E quiver is tiling algebra}, there is only one Hom-Ext quiver {\color{black}$(Q',R)$} up to {\color{black} isomorphism}: $a \rightrightarrows b$, where $R = \emptyset$. By Theorem \ref{thm: linear extensions and exceptional sequences}, there is only one possible exceptional ordering of these vertices: $(a,b)$. By placing a non-regular module $X_i$ at vertex $a$ or $b$, the other vertex is uniquely determined by Lemma \ref{lem: uniqueness in exceptional sequence}.

{\color{black} Suppose the result holds for all Hom-Ext quivers over all possible acyclic quivers of type $\tilde{\mathbb{A}}_{k}$ whose corresponding annulus has $k$ marked points on the exterior boundary and $1$ marked point on the interior. Consider a quiver of type $\tilde{\mathbb{A}}_{k+1}$ whose corresponding annulus has $k+1$ marked points on the exterior boundary and one on the interior. If there are no regular modules, all arcs are bridging, and in this case, there is only one such collection up to $2\pi$ Dehn twist of the inner boundary component. By Lemma \ref{lem: affine A subquiver with non-regular source}, we know that there is an $X_i$ that is a source. Once this source is chosen, the remainder of the diagram is uniquely determined. 

In the case there are regular modules, there must be a quasi-simple, say $X_l$. We then perform the following homotopy along the arc corresponding to the quasi-simple regular so that the outer boundary component now has one less marked point:}

\begin{center}
\tikzset{every picture/.style={line width=0.75pt}} 
\begin{tikzpicture}[x=0.75pt,y=0.75pt,yscale=-1,xscale=1]
\draw   (87.01,100.25) .. controls (87.01,82.19) and (102.06,67.55) .. (120.64,67.55) .. controls (139.21,67.55) and (154.27,82.19) .. (154.27,100.25) .. controls (154.27,118.31) and (139.21,132.95) .. (120.64,132.95) .. controls (102.06,132.95) and (87.01,118.31) .. (87.01,100.25)(37.96,100.25) .. controls (37.96,55.1) and (74.97,18.5) .. (120.64,18.5) .. controls (166.3,18.5) and (203.32,55.1) .. (203.32,100.25) .. controls (203.32,145.4) and (166.3,182) .. (120.64,182) .. controls (74.97,182) and (37.96,145.4) .. (37.96,100.25) ;
\draw    (120.64,182) ;
\draw [shift={(120.64,182)}, rotate = 0] [color={rgb, 255:red, 0; green, 0; blue, 0 }  ][fill={rgb, 255:red, 0; green, 0; blue, 0 }  ][line width=0.75]      (0, 0) circle [x radius= 3.35, y radius= 3.35]   ;
\draw    (120.64,67.55) ;
\draw [shift={(120.64,67.55)}, rotate = 0] [color={rgb, 255:red, 0; green, 0; blue, 0 }  ][fill={rgb, 255:red, 0; green, 0; blue, 0 }  ][line width=0.75]      (0, 0) circle [x radius= 3.35, y radius= 3.35]   ;
\draw    (120.64,18.5) ;
\draw [shift={(120.64,18.5)}, rotate = 0] [color={rgb, 255:red, 0; green, 0; blue, 0 }  ][fill={rgb, 255:red, 0; green, 0; blue, 0 }  ][line width=0.75]      (0, 0) circle [x radius= 3.35, y radius= 3.35]   ;
\draw    (149,176.5) ;
\draw [shift={(149,176.5)}, rotate = 0] [color={rgb, 255:red, 0; green, 0; blue, 0 }  ][fill={rgb, 255:red, 0; green, 0; blue, 0 }  ][line width=0.75]      (0, 0) circle [x radius= 3.35, y radius= 3.35]   ;
\draw    (220,101) -- (339,101.49) ;
\draw [shift={(341,101.5)}, rotate = 180.24] [color={rgb, 255:red, 0; green, 0; blue, 0 }  ][line width=0.75]    (10.93,-3.29) .. controls (6.95,-1.4) and (3.31,-0.3) .. (0,0) .. controls (3.31,0.3) and (6.95,1.4) .. (10.93,3.29)   ;
\draw   (416.01,100.25) .. controls (416.01,82.19) and (431.06,67.55) .. (449.64,67.55) .. controls (468.21,67.55) and (483.27,82.19) .. (483.27,100.25) .. controls (483.27,118.31) and (468.21,132.95) .. (449.64,132.95) .. controls (431.06,132.95) and (416.01,118.31) .. (416.01,100.25)(366.96,100.25) .. controls (366.96,55.1) and (403.97,18.5) .. (449.64,18.5) .. controls (495.3,18.5) and (532.32,55.1) .. (532.32,100.25) .. controls (532.32,145.4) and (495.3,182) .. (449.64,182) .. controls (403.97,182) and (366.96,145.4) .. (366.96,100.25) ;
\draw    (449.64,182) ;
\draw [shift={(449.64,182)}, rotate = 0] [color={rgb, 255:red, 0; green, 0; blue, 0 }  ][fill={rgb, 255:red, 0; green, 0; blue, 0 }  ][line width=0.75]      (0, 0) circle [x radius= 3.35, y radius= 3.35]   ;
\draw    (449.64,67.55) ;
\draw [shift={(449.64,67.55)}, rotate = 0] [color={rgb, 255:red, 0; green, 0; blue, 0 }  ][fill={rgb, 255:red, 0; green, 0; blue, 0 }  ][line width=0.75]      (0, 0) circle [x radius= 3.35, y radius= 3.35]   ;
\draw    (449.64,18.5) ;
\draw [shift={(449.64,18.5)}, rotate = 0] [color={rgb, 255:red, 0; green, 0; blue, 0 }  ][fill={rgb, 255:red, 0; green, 0; blue, 0 }  ][line width=0.75]      (0, 0) circle [x radius= 3.35, y radius= 3.35]   ;
\draw    (120.64,182) .. controls (122,168) and (149,167) .. (149,176.5) ;

\draw (31.51,83.95) node [anchor=north west][inner sep=0.75pt]  [rotate=-89.6]  {$\dotsc $};
\draw (211.48,78.93) node [anchor=north west][inner sep=0.75pt]  [rotate=-89.48]  {$\dotsc $};
\draw (149,184.4) node [anchor=north west][inner sep=0.75pt]  [font=\footnotesize]  {$b_{i}$};
\draw (106,193.4) node [anchor=north west][inner sep=0.75pt]  [font=\footnotesize]  {$b_{i+1}$};
\draw (360.51,83.95) node [anchor=north west][inner sep=0.75pt]  [rotate=-89.6]  {$\dotsc $};
\draw (540.48,78.93) node [anchor=north west][inner sep=0.75pt]  [rotate=-89.48]  {$\dotsc $};
\draw (439,186.4) node [anchor=north west][inner sep=0.75pt]  [font=\footnotesize]  {$b_{i+1}$};
\draw (274,83.4) node [anchor=north west][inner sep=0.75pt]  [font=\footnotesize]  {$\pi $};
\end{tikzpicture}
\end{center}

{\color{black} All the regulars and one bridging arc have been determined in the diagram on the right. By the inductive hypothesis, all the modules $\pi(X_1), \pi(X_2), \dots, 
\pi(X_{l-1}), \pi(X_{l+1}), \dots \pi(X_{k+1})$, and their corresponding arcs, are uniquely determined. Now consider the arcs $X_1, X_2,
$ $\dots, X_{l-1}, X_l, X_{l+1}, \dots X_{k+1}$. All of the arcs $\pi(X_1), \pi(X_2), \dots, 
\pi(X_{l-1}), \pi(X_{l+1}), \dots \pi(X_{k+1})$ whose endpoint is not $b_{i+1}$ remain unchanged and hence are uniquely determined. Those which have an endpoint at $b_{i+1}$ are then determined by the existence of a relation avoiding path in $Q'$ either to $X_l$ or from $X_l$. So therefore, all the arcs are uniquely determined.}

{\color{black} Now the suppose the result holds for all Hom-Ext quivers over all acyclic quivers of type $\tilde{\mathbb{A}}_m$ with $p$ vertices on the outer boundary component and $k$ on the inner boundary component, and suppose we have a quiver of type $\tilde{\mathbb{A}}_{m+1}$ with $p$ vertices on the outer boundary component and $k+1$ on the inner boundary component. Suppose first that there are no regular modules. Then again, by Lemma \ref{lem: affine A subquiver with non-regular source}, suppose we fixed a source in the Hom-Ext quiver. Since this is a source in the Hom-Ext quiver, we can perform the following Homotopy:}

\begin{center}

\tikzset{every picture/.style={line width=0.75pt}} 

\begin{tikzpicture}[x=0.75pt,y=0.75pt,yscale=-1,xscale=1]

\draw   (87.01,100.25) .. controls (87.01,82.19) and (102.06,67.55) .. (120.64,67.55) .. controls (139.21,67.55) and (154.27,82.19) .. (154.27,100.25) .. controls (154.27,118.31) and (139.21,132.95) .. (120.64,132.95) .. controls (102.06,132.95) and (87.01,118.31) .. (87.01,100.25)(37.96,100.25) .. controls (37.96,55.1) and (74.97,18.5) .. (120.64,18.5) .. controls (166.3,18.5) and (203.32,55.1) .. (203.32,100.25) .. controls (203.32,145.4) and (166.3,182) .. (120.64,182) .. controls (74.97,182) and (37.96,145.4) .. (37.96,100.25) ;
\draw    (120.64,182) ;
\draw [shift={(120.64,182)}, rotate = 0] [color={rgb, 255:red, 0; green, 0; blue, 0 }  ][fill={rgb, 255:red, 0; green, 0; blue, 0 }  ][line width=0.75]      (0, 0) circle [x radius= 3.35, y radius= 3.35]   ;
\draw    (87.01,100.25) ;
\draw [shift={(87.01,100.25)}, rotate = 0] [color={rgb, 255:red, 0; green, 0; blue, 0 }  ][fill={rgb, 255:red, 0; green, 0; blue, 0 }  ][line width=0.75]      (0, 0) circle [x radius= 3.35, y radius= 3.35]   ;
\draw    (154.27,100.25) ;
\draw [shift={(154.27,100.25)}, rotate = 0] [color={rgb, 255:red, 0; green, 0; blue, 0 }  ][fill={rgb, 255:red, 0; green, 0; blue, 0 }  ][line width=0.75]      (0, 0) circle [x radius= 3.35, y radius= 3.35]   ;
\draw    (120.64,132.95) ;
\draw [shift={(120.64,132.95)}, rotate = 0] [color={rgb, 255:red, 0; green, 0; blue, 0 }  ][fill={rgb, 255:red, 0; green, 0; blue, 0 }  ][line width=0.75]      (0, 0) circle [x radius= 3.35, y radius= 3.35]   ;
\draw    (120.64,182) .. controls (148,165) and (157,129) .. (143,125.5) ;
\draw    (120.64,18.5) ;
\draw [shift={(120.64,18.5)}, rotate = 0] [color={rgb, 255:red, 0; green, 0; blue, 0 }  ][fill={rgb, 255:red, 0; green, 0; blue, 0 }  ][line width=0.75]      (0, 0) circle [x radius= 3.35, y radius= 3.35]   ;
\draw    (143,125.5) ;
\draw [shift={(143,125.5)}, rotate = 0] [color={rgb, 255:red, 0; green, 0; blue, 0 }  ][fill={rgb, 255:red, 0; green, 0; blue, 0 }  ][line width=0.75]      (0, 0) circle [x radius= 3.35, y radius= 3.35]   ;
\draw    (220,101) -- (339,101.49) ;
\draw [shift={(341,101.5)}, rotate = 180.24] [color={rgb, 255:red, 0; green, 0; blue, 0 }  ][line width=0.75]    (10.93,-3.29) .. controls (6.95,-1.4) and (3.31,-0.3) .. (0,0) .. controls (3.31,0.3) and (6.95,1.4) .. (10.93,3.29)   ;
\draw    (120.64,132.95) -- (120.64,182) ;
\draw    (66,161) .. controls (82,157) and (103,154) .. (120.64,132.95) ;
\draw    (66,161) ;
\draw [shift={(66,161)}, rotate = 0] [color={rgb, 255:red, 0; green, 0; blue, 0 }  ][fill={rgb, 255:red, 0; green, 0; blue, 0 }  ][line width=0.75]      (0, 0) circle [x radius= 3.35, y radius= 3.35]   ;
\draw   (427.01,99.25) .. controls (427.01,81.19) and (442.06,66.55) .. (460.64,66.55) .. controls (479.21,66.55) and (494.27,81.19) .. (494.27,99.25) .. controls (494.27,117.31) and (479.21,131.95) .. (460.64,131.95) .. controls (442.06,131.95) and (427.01,117.31) .. (427.01,99.25)(377.96,99.25) .. controls (377.96,54.1) and (414.97,17.5) .. (460.64,17.5) .. controls (506.3,17.5) and (543.32,54.1) .. (543.32,99.25) .. controls (543.32,144.4) and (506.3,181) .. (460.64,181) .. controls (414.97,181) and (377.96,144.4) .. (377.96,99.25) ;
\draw    (460.64,181) ;
\draw [shift={(460.64,181)}, rotate = 0] [color={rgb, 255:red, 0; green, 0; blue, 0 }  ][fill={rgb, 255:red, 0; green, 0; blue, 0 }  ][line width=0.75]      (0, 0) circle [x radius= 3.35, y radius= 3.35]   ;
\draw    (427.01,99.25) ;
\draw [shift={(427.01,99.25)}, rotate = 0] [color={rgb, 255:red, 0; green, 0; blue, 0 }  ][fill={rgb, 255:red, 0; green, 0; blue, 0 }  ][line width=0.75]      (0, 0) circle [x radius= 3.35, y radius= 3.35]   ;
\draw    (494.27,99.25) ;
\draw [shift={(494.27,99.25)}, rotate = 0] [color={rgb, 255:red, 0; green, 0; blue, 0 }  ][fill={rgb, 255:red, 0; green, 0; blue, 0 }  ][line width=0.75]      (0, 0) circle [x radius= 3.35, y radius= 3.35]   ;
\draw    (460.64,131.95) ;
\draw [shift={(460.64,131.95)}, rotate = 0] [color={rgb, 255:red, 0; green, 0; blue, 0 }  ][fill={rgb, 255:red, 0; green, 0; blue, 0 }  ][line width=0.75]      (0, 0) circle [x radius= 3.35, y radius= 3.35]   ;
\draw    (460.64,17.5) ;
\draw [shift={(460.64,17.5)}, rotate = 0] [color={rgb, 255:red, 0; green, 0; blue, 0 }  ][fill={rgb, 255:red, 0; green, 0; blue, 0 }  ][line width=0.75]      (0, 0) circle [x radius= 3.35, y radius= 3.35]   ;
\draw    (460.64,131.95) -- (460.64,181) ;
\draw    (406,160) .. controls (422,156) and (443,153) .. (460.64,131.95) ;
\draw    (406,160) ;
\draw [shift={(406,160)}, rotate = 0] [color={rgb, 255:red, 0; green, 0; blue, 0 }  ][fill={rgb, 255:red, 0; green, 0; blue, 0 }  ][line width=0.75]      (0, 0) circle [x radius= 3.35, y radius= 3.35]   ;

\draw (106.94,50.5) node [anchor=north west][inner sep=0.75pt]  [rotate=-359.56]  {$\dotsc $};
\draw (31.51,83.95) node [anchor=north west][inner sep=0.75pt]  [rotate=-89.6]  {$\dotsc $};
\draw (211.48,78.93) node [anchor=north west][inner sep=0.75pt]  [rotate=-89.48]  {$\dotsc $};
\draw (132,108.4) node [anchor=north west][inner sep=0.75pt]  [font=\footnotesize]  {$b_{i}$};
\draw (100,114.4) node [anchor=north west][inner sep=0.75pt]  [font=\footnotesize]  {$b_{i+1}$};
\draw (446.94,49.5) node [anchor=north west][inner sep=0.75pt]  [rotate=-359.56]  {$\dotsc $};
\draw (371.51,82.95) node [anchor=north west][inner sep=0.75pt]  [rotate=-89.6]  {$\dotsc $};
\draw (551.48,77.93) node [anchor=north west][inner sep=0.75pt]  [rotate=-89.48]  {$\dotsc $};
\draw (450,111.4) node [anchor=north west][inner sep=0.75pt]  [font=\footnotesize]  {$b_{i+1}$};
\draw (274,83.4) node [anchor=north west][inner sep=0.75pt]  [font=\footnotesize]  {$\pi $};

\end{tikzpicture}

\end{center}

{\color{black} The result follows by induction analogous to the previous case. Finally, if there are regular modules, then there must exist a quasi-simple regular module on the inner boundary component and the proof is again analogous. }
\end{proof}

{\color{black} Before proving the main result of this section, we need a pair of lemmas.

\begin{lem}\label{lem: iso sends simple to simple}
    Let $\varphi:(Q^{\chi_1},R_1) \rightarrow (Q^{\chi_2},R_2)$ be an isomorphism of Hom-Ext quivers where $\chi_1$ and $\chi_2$ are two exceptional collections of a quiver of type $\tilde{\mathbb{A}}_{n-1}$. If $X_i \in \chi_1$ is a quasi-simple module, so is $\varphi(X_i)$.
\end{lem}

\begin{proof}
    Suppose that $\varphi$ is an isomorphism of Hom-Ext quivers, $X_i\in \chi_1$ is a quasi-simple left regular module, and for contradiction, that $\varphi(X_i)$ is not quasi-simple. The proof when $X_i$ is right-regular is analogous. Since $\chi_2$ is complete, there exists a left regular module $\varphi(X_{i_1})$ such that there is a surjection $\varphi(X_i) \rightarrow \varphi(X_{i_1})$, or an injection $\varphi(X_{i_1}) \rightarrow \varphi(X_i)$. We will assume the later case, as the proof in the former case is analogous. Again, by completeness of $\chi_2$, there is a chain of morphisms $\varphi(X_{i_k}) \leftrightarrow \cdots \leftrightarrow \varphi(X_{i_2}) \leftrightarrow \varphi(X_{i_1}) \rightarrow \varphi(X_i)$ in $(Q^{\chi_2},R_2)$ where $\varphi(X_{i_k})$ is a quasi-simple module and the two headed arrows indicate that the morphism can go either left or right. Since $\varphi$ is an isomorphism of Hom-Ext quivers, we have an analogous chain $X_{i_k} \leftrightarrow \cdots \leftrightarrow X_{i_2} \leftrightarrow X_{i_1} \rightarrow X_i$  in $(Q^{\chi_1},R_1)$ as in the following diagram.

\begin{center}

\tikzset{every picture/.style={line width=0.75pt}} 

\begin{tikzpicture}[x=0.75pt,y=0.75pt,yscale=-1,xscale=1]

\draw   (87.01,100.25) .. controls (87.01,82.19) and (102.06,67.55) .. (120.64,67.55) .. controls (139.21,67.55) and (154.27,82.19) .. (154.27,100.25) .. controls (154.27,118.31) and (139.21,132.95) .. (120.64,132.95) .. controls (102.06,132.95) and (87.01,118.31) .. (87.01,100.25)(37.96,100.25) .. controls (37.96,55.1) and (74.97,18.5) .. (120.64,18.5) .. controls (166.3,18.5) and (203.32,55.1) .. (203.32,100.25) .. controls (203.32,145.4) and (166.3,182) .. (120.64,182) .. controls (74.97,182) and (37.96,145.4) .. (37.96,100.25) ;
\draw    (167.64,167) ;
\draw [shift={(167.64,167)}, rotate = 0] [color={rgb, 255:red, 0; green, 0; blue, 0 }  ][fill={rgb, 255:red, 0; green, 0; blue, 0 }  ][line width=0.75]      (0, 0) circle [x radius= 3.35, y radius= 3.35]   ;
\draw    (188.27,147.25) ;
\draw [shift={(188.27,147.25)}, rotate = 0] [color={rgb, 255:red, 0; green, 0; blue, 0 }  ][fill={rgb, 255:red, 0; green, 0; blue, 0 }  ][line width=0.75]      (0, 0) circle [x radius= 3.35, y radius= 3.35]   ;
\draw    (203.32,100.25) ;
\draw [shift={(203.32,100.25)}, rotate = 0] [color={rgb, 255:red, 0; green, 0; blue, 0 }  ][fill={rgb, 255:red, 0; green, 0; blue, 0 }  ][line width=0.75]      (0, 0) circle [x radius= 3.35, y radius= 3.35]   ;
\draw    (82.64,172.5) ;
\draw [shift={(82.64,172.5)}, rotate = 0] [color={rgb, 255:red, 0; green, 0; blue, 0 }  ][fill={rgb, 255:red, 0; green, 0; blue, 0 }  ][line width=0.75]      (0, 0) circle [x radius= 3.35, y radius= 3.35]   ;
\draw    (167.64,167) .. controls (168,154) and (174,151) .. (188.27,147.25) ;
\draw    (120.64,18.5) ;
\draw [shift={(120.64,18.5)}, rotate = 0] [color={rgb, 255:red, 0; green, 0; blue, 0 }  ][fill={rgb, 255:red, 0; green, 0; blue, 0 }  ][line width=0.75]      (0, 0) circle [x radius= 3.35, y radius= 3.35]   ;
\draw    (82.64,172.5) .. controls (122.64,142.5) and (141,148.5) .. (188.27,147.25) ;
\draw    (47,137.5) ;
\draw [shift={(47,137.5)}, rotate = 0] [color={rgb, 255:red, 0; green, 0; blue, 0 }  ][fill={rgb, 255:red, 0; green, 0; blue, 0 }  ][line width=0.75]      (0, 0) circle [x radius= 3.35, y radius= 3.35]   ;
\draw    (47,137.5) .. controls (98,164.5) and (191,127.5) .. (203.32,100.25) ;
\draw   (358.01,99.25) .. controls (358.01,81.19) and (373.06,66.55) .. (391.64,66.55) .. controls (410.21,66.55) and (425.27,81.19) .. (425.27,99.25) .. controls (425.27,117.31) and (410.21,131.95) .. (391.64,131.95) .. controls (373.06,131.95) and (358.01,117.31) .. (358.01,99.25)(308.96,99.25) .. controls (308.96,54.1) and (345.97,17.5) .. (391.64,17.5) .. controls (437.3,17.5) and (474.32,54.1) .. (474.32,99.25) .. controls (474.32,144.4) and (437.3,181) .. (391.64,181) .. controls (345.97,181) and (308.96,144.4) .. (308.96,99.25) ;
\draw    (391.64,181) ;
\draw [shift={(391.64,181)}, rotate = 0] [color={rgb, 255:red, 0; green, 0; blue, 0 }  ][fill={rgb, 255:red, 0; green, 0; blue, 0 }  ][line width=0.75]      (0, 0) circle [x radius= 3.35, y radius= 3.35]   ;
\draw    (459.27,146.25) ;
\draw [shift={(459.27,146.25)}, rotate = 0] [color={rgb, 255:red, 0; green, 0; blue, 0 }  ][fill={rgb, 255:red, 0; green, 0; blue, 0 }  ][line width=0.75]      (0, 0) circle [x radius= 3.35, y radius= 3.35]   ;
\draw    (474.32,99.25) ;
\draw [shift={(474.32,99.25)}, rotate = 0] [color={rgb, 255:red, 0; green, 0; blue, 0 }  ][fill={rgb, 255:red, 0; green, 0; blue, 0 }  ][line width=0.75]      (0, 0) circle [x radius= 3.35, y radius= 3.35]   ;
\draw    (353.64,171.5) ;
\draw [shift={(353.64,171.5)}, rotate = 0] [color={rgb, 255:red, 0; green, 0; blue, 0 }  ][fill={rgb, 255:red, 0; green, 0; blue, 0 }  ][line width=0.75]      (0, 0) circle [x radius= 3.35, y radius= 3.35]   ;
\draw    (353.64,171.5) .. controls (374,155.5) and (388,169.5) .. (391.64,181) ;
\draw    (391.64,17.5) ;
\draw [shift={(391.64,17.5)}, rotate = 0] [color={rgb, 255:red, 0; green, 0; blue, 0 }  ][fill={rgb, 255:red, 0; green, 0; blue, 0 }  ][line width=0.75]      (0, 0) circle [x radius= 3.35, y radius= 3.35]   ;
\draw    (318,136.5) .. controls (368,168.5) and (439,164.5) .. (459.27,146.25) ;
\draw    (318,136.5) ;
\draw [shift={(318,136.5)}, rotate = 0] [color={rgb, 255:red, 0; green, 0; blue, 0 }  ][fill={rgb, 255:red, 0; green, 0; blue, 0 }  ][line width=0.75]      (0, 0) circle [x radius= 3.35, y radius= 3.35]   ;
\draw    (318,136.5) .. controls (369,163.5) and (462,126.5) .. (474.32,99.25) ;
\draw    (120.64,67.55) ;
\draw [shift={(120.64,67.55)}, rotate = 0] [color={rgb, 255:red, 0; green, 0; blue, 0 }  ][fill={rgb, 255:red, 0; green, 0; blue, 0 }  ][line width=0.75]      (0, 0) circle [x radius= 3.35, y radius= 3.35]   ;
\draw    (391.64,66.55) ;
\draw [shift={(391.64,66.55)}, rotate = 0] [color={rgb, 255:red, 0; green, 0; blue, 0 }  ][fill={rgb, 255:red, 0; green, 0; blue, 0 }  ][line width=0.75]      (0, 0) circle [x radius= 3.35, y radius= 3.35]   ;

\draw (186.58,30.01) node [anchor=north west][inner sep=0.75pt]  [rotate=-54.26]  {$\dotsc $};
\draw (49.43,150.65) node [anchor=north west][inner sep=0.75pt]  [rotate=-42.02]  {$\dotsc $};
\draw (151,151.4) node [anchor=north west][inner sep=0.75pt]  [font=\scriptsize]  {$X_{i}$};
\draw (97,161.4) node [anchor=north west][inner sep=0.75pt]  [font=\scriptsize]  {$X_{i_{1}}$};
\draw (217.29,120.04) node [anchor=north west][inner sep=0.75pt]  [rotate=-108.26]  {$\dotsc $};
\draw (60,129.4) node [anchor=north west][inner sep=0.75pt]  [font=\scriptsize]  {$X_{i_{k}}$};
\draw (437.97,182.62) node [anchor=north west][inner sep=0.75pt]  [rotate=-155.85]  {$\dotsc $};
\draw (454.58,31.01) node [anchor=north west][inner sep=0.75pt]  [rotate=-54.26]  {$\dotsc $};
\draw (322.43,149.65) node [anchor=north west][inner sep=0.75pt]  [rotate=-42.02]  {$\dotsc $};
\draw (387,161.4) node [anchor=north west][inner sep=0.75pt]  [font=\scriptsize]  {$\varphi ( X_{i_{k}})$};
\draw (415,138.4) node [anchor=north west][inner sep=0.75pt]  [font=\scriptsize]  {$\varphi ( X_{i_{1}})$};
\draw (488.29,118.04) node [anchor=north west][inner sep=0.75pt]  [rotate=-108.26]  {$\dotsc $};
\draw (331,128.4) node [anchor=north west][inner sep=0.75pt]  [font=\scriptsize]  {$\varphi ( X_{i})$};
\draw (298.12,68.4) node [anchor=north west][inner sep=0.75pt]  [rotate=-298.49]  {$\dotsc $};
\draw (35.12,64.4) node [anchor=north west][inner sep=0.75pt]  [rotate=-298.49]  {$\dotsc $};
\draw (158.56,77.37) node [anchor=north west][inner sep=0.75pt]  [rotate=-73.38]  {$\dotsc $};
\draw (428.56,71.37) node [anchor=north west][inner sep=0.75pt]  [rotate=-73.38]  {$\dotsc $};

\end{tikzpicture}

\end{center}

    In order to be complete exceptional collections, there must be a path from $X_{i_k}$ to a non-regular module in $(Q^{\chi_1},R_1)$. Since $\varphi$ is an isomorphism, this implies that there is a path connecting a non-regular module and $\varphi(X_{i_k})$ in $(Q^{\chi_2},R_2)$. But this path is shorter in $(Q^{\chi_1},R_1)$ than it is in $(Q^{\chi_2},R_2)$, contradicting the fact that $\varphi$ is an isomorphism of Hom-Ext quivers.
\end{proof}
}

\begin{lem}\label{lem: regulars twist to eachother}
    Let $Q$ be a quiver of type $\tilde{\mathbb{A}}_{n-1}$ and $\chi$ and $\chi'$ be two exceptional collections. Let $\{X_{i_1}, X_{i_2}, \dots, X_{i_k}\}$ and $\{X_{j_1}, X_{j_2}, \dots, X_{j_l}\}$ be the left regular and right regular modules in $\chi$ respectively and suppose that $\varphi: (Q^{\chi},R) \rightarrow (Q^{
\chi'},R')$ is an isomorphism of Hom-Ext quivers. Then there exist $a$ and $b$ such that $$\{T_{\mathcal{L}}^a(\varphi
(X_{i_1})), T_{\mathcal{L}}^a(\varphi(X_{i_2})), \dots, T_{\mathcal{L}}^a(\varphi(X_{i_k}))\} = \{X_{i_1}, X_{i_2}, \dots, X_{i_k}\}$$ and $$\{T_{\mathcal{R}}^b(\varphi(X_{j_1})), T_{\mathcal{R}}^b(\varphi
(X_{j_2})), \dots, T_{\mathcal{R}}^b(\varphi(X_{j_l}))\} = \{X_{j_1}, X_{j_2}, \dots, X_{j_l}\}$$

\noindent
where $a,b \in \mathbb{Z}$ {\color{black} and $T^a_{\mathcal{L}}$ ($T^b_{\mathcal{R}}$) are $a$-fold ($b$-fold) elementary twists of the outer (inner) boundary component of $A_{Q}$.} 
\end{lem}

\begin{proof}
    We proceed by induction on the number of marked points on each boundary component. When there is one marked point on the boundary component, no quasi-simple module can occur in an exceptional collection, so there is nothing to show. In the case there are two marked points, the quasi-simples differ by an elementary twist, so the result holds. Suppose that if there are $k$ marked points on the boundary component, that the result holds and assume we are in the case that there are $k+1$ marked points on the boundary component. If the set of left-regular modules is empty, the statement is vacuously true. If the set of left-regular modules is nonempty, by completeness of the exceptional collections $\chi$ and $\chi'$, there must exist a quasi-simple $X_i\in\chi$ and $\varphi(X_i) \in \chi'$ by Lemma \ref{lem: iso sends simple to simple}. As in the proof of Lemma \ref{lem: Hom-Ext quiver is uniquely determined by one non-regular}, we can homotope along these quasi-simple arcs in each of the exceptional arc diagrams. Since the resulting Hom-Ext quivers are attained by removing this quasi simple from the vertex sets, along with any monomial relation passing through through it from the set of relations, we have that the isomorphism $\varphi$ descends and the result follows from induction. 
\end{proof}

\begin{thm}\label{thm: iso H-E quiver iff Dehn twist}
Let $Q$ be a quiver of type $\tilde{\mathbb{A}}_{n-1}$ and $\chi$ and $\chi'$ be two exceptional collections. Then $(Q^{\chi},R) \cong (Q^{\chi'},R')$ if and only if $D_{\chi} = T_{\mathcal{L}}^a T_{\mathcal{R}}^b(D_{\chi'})$, where $a,b \in \mathbb{Z}$ {\color{black} and $T^a_{\mathcal{L}}$ ($T^b_{\mathcal{R}}$) are $a$-fold ($b$-fold) elementary twists of the outer (inner) boundary component of $A_{Q}$.}
\end{thm}

\begin{proof}
If $D_{\chi} = T_{\mathcal{L}}^a T_{\mathcal{R}}^b(D_{\chi'})$, then we have an isomorphism of tiling algebras $T_{\chi} \cong T_{\chi'}$. By Proposition \ref{prop: H-E quiver is tiling algebra}, we have that $(Q^{\chi},R) \cong (Q^{\chi'},R')$.

Suppose that $\varphi: (Q^{\chi},R) \rightarrow (Q^{
\chi'},R')$ is an isomorphism of Hom-Ext quivers. If there are no regular modules in $\chi$, since $\varphi$ is an isomorphism, there are no regular modules in $\chi'$. Fix a bridging arc $X$ in $\chi$. Then there exist $a,b \in \mathbb{Z}$ such that $T^a_{\mathscr{L}}T^b_{\mathscr{R}}(\varphi(X)) = X$. By Lemma \ref{lem: Hom-Ext quiver is uniquely determined by one non-regular}, the remainder of the diagram is uniquely determined and $D_{\chi} = T^a_{\mathscr{L}}T^b_{\mathscr{R}}(D_{\chi'})$. Now suppose there are regulars. In the exceptional collection $\chi$, there exists a non-regular $X_l$ that intersects a regular $X_{i_l}$ at a vertex $i$ on the outer boundary component and a regular $X_{j_l}$ at a vertex $j$ on the inner boundary component. Since $\varphi$ is an isomorphism, $\varphi(X_l)$ also intersects $\varphi(X_{i_l})$ and $\varphi(X_{j_l})$ on their respective boundary components, so  by Lemma \ref{lem: regulars twist to eachother}, there exist integers $a$ and $b$ such that  $T_{\mathcal{L}}^a T_{\mathcal{R}}^b(\varphi(X_l))$ has the same endpoints as $X_l$. So, after some number $y$ of $2\pi$ Dehn twists, $T_{\mathcal{L}}^{a+y} T_{\mathcal{R}}^b(\varphi(X_l)) = X_l$. By Lemma \ref{lem: Hom-Ext quiver is uniquely determined by one non-regular}, the other modules in $T_{\mathcal{L}}^{a+y} T_{\mathcal{R}}^b(Q^{\chi'})$ are uniquely determined and $D_{\chi} = T_{\mathcal{L}}^{a+y} T_{\mathcal{R}}^b(D_{\chi'})$.
\end{proof}

\begin{rem}
    The reference \cite{gross2001topological} contains similar results using the language of graph embeddings and rotation systems. 
\end{rem}

\begin{exmp}
{\color{black} We do in fact need isomorphism of Hom-Ext quivers for Theorem \ref{thm: iso H-E quiver iff Dehn twist} to hold. Consider the following quiver.
\vspace{.5cm}
\[
\xymatrix{
1 \ar[r] \ar@/^2pc/[rrr] & 2 \ar[r] & 3  & 4 \ar[l]
	}
\]
Consider the following two exceptional sets along with the underlying quiver of their corresponding Hom-Ext quivers.}

\begin{center}

\tikzset{every picture/.style={line width=0.75pt}} 

\begin{tikzpicture}[x=0.75pt,y=0.75pt,yscale=-1,xscale=1]

\draw   (121.01,107.25) .. controls (121.01,89.19) and (136.06,74.55) .. (154.64,74.55) .. controls (173.21,74.55) and (188.27,89.19) .. (188.27,107.25) .. controls (188.27,125.31) and (173.21,139.95) .. (154.64,139.95) .. controls (136.06,139.95) and (121.01,125.31) .. (121.01,107.25)(71.96,107.25) .. controls (71.96,62.1) and (108.97,25.5) .. (154.64,25.5) .. controls (200.3,25.5) and (237.32,62.1) .. (237.32,107.25) .. controls (237.32,152.4) and (200.3,189) .. (154.64,189) .. controls (108.97,189) and (71.96,152.4) .. (71.96,107.25) ;

\draw    (121.01,107.25) .. controls (101,105) and (85,138) .. (95.96,165.25) ;
\draw [color={rgb, 255:red, 208; green, 2; blue, 27 }  ,draw opacity=1 ]   (188.27,107.25) .. controls (206,102) and (230,128) .. (219.32,157.25) ;
\draw [color={rgb, 255:red, 245; green, 166; blue, 35 }  ,draw opacity=1 ]   (121.01,107.25) .. controls (94,138) and (161,195) .. (219.32,157.25) ;
\draw [color={rgb, 255:red, 22; green, 48; blue, 226 }  ,draw opacity=1 ]   (95.96,165.25) .. controls (13,37) and (271,-14) .. (219.32,157.25) ;
\draw   (435.01,106.25) .. controls (435.01,88.19) and (450.06,73.55) .. (468.64,73.55) .. controls (487.21,73.55) and (502.27,88.19) .. (502.27,106.25) .. controls (502.27,124.31) and (487.21,138.95) .. (468.64,138.95) .. controls (450.06,138.95) and (435.01,124.31) .. (435.01,106.25)(385.96,106.25) .. controls (385.96,61.1) and (422.97,24.5) .. (468.64,24.5) .. controls (514.3,24.5) and (551.32,61.1) .. (551.32,106.25) .. controls (551.32,151.4) and (514.3,188) .. (468.64,188) .. controls (422.97,188) and (385.96,151.4) .. (385.96,106.25) ;

\draw    (435.01,106.25) .. controls (415,104) and (399,137) .. (409.96,164.25) ;
\draw [color={rgb, 255:red, 208; green, 2; blue, 27 }  ,draw opacity=1 ]   (502.27,106.25) .. controls (520,101) and (544,127) .. (533.32,156.25) ;
\draw [color={rgb, 255:red, 245; green, 166; blue, 35 }  ,draw opacity=1 ]   (502.27,106.25) .. controls (508,-5) and (349,77) .. (409.96,164.25) ;
\draw [color={rgb, 255:red, 22; green, 48; blue, 226 }  ,draw opacity=1 ]   (435.01,106.25) .. controls (406,166) and (528,171) .. (502.27,106.25) ;
\draw    (319,218) ;
\draw [shift={(319,218)}, rotate = 0] [color={rgb, 255:red, 0; green, 0; blue, 0 }  ][fill={rgb, 255:red, 0; green, 0; blue, 0 }  ][line width=0.75]      (0, 0) circle [x radius= 3.35, y radius= 3.35]   ;
\draw    (312,222) -- (275.69,244.93) ;
\draw [shift={(274,246)}, rotate = 327.72] [color={rgb, 255:red, 0; green, 0; blue, 0 }  ][line width=0.75]    (10.93,-3.29) .. controls (6.95,-1.4) and (3.31,-0.3) .. (0,0) .. controls (3.31,0.3) and (6.95,1.4) .. (10.93,3.29)   ;
\draw [color={rgb, 255:red, 22; green, 48; blue, 226 }  ,draw opacity=1 ]   (269,251) ;
\draw [shift={(269,251)}, rotate = 0] [color={rgb, 255:red, 22; green, 48; blue, 226 }  ,draw opacity=1 ][fill={rgb, 255:red, 22; green, 48; blue, 226 }  ,fill opacity=1 ][line width=0.75]      (0, 0) circle [x radius= 3.35, y radius= 3.35]   ;
\draw    (279,250) -- (312,250) ;
\draw [shift={(314,250)}, rotate = 180] [color={rgb, 255:red, 0; green, 0; blue, 0 }  ][line width=0.75]    (10.93,-3.29) .. controls (6.95,-1.4) and (3.31,-0.3) .. (0,0) .. controls (3.31,0.3) and (6.95,1.4) .. (10.93,3.29)   ;
\draw [color={rgb, 255:red, 208; green, 2; blue, 27 }  ,draw opacity=1 ]   (321,251) ;
\draw [shift={(321,251)}, rotate = 0] [color={rgb, 255:red, 208; green, 2; blue, 27 }  ,draw opacity=1 ][fill={rgb, 255:red, 208; green, 2; blue, 27 }  ,fill opacity=1 ][line width=0.75]      (0, 0) circle [x radius= 3.35, y radius= 3.35]   ;
\draw    (331,251) -- (364,251) ;
\draw [shift={(366,251)}, rotate = 180] [color={rgb, 255:red, 0; green, 0; blue, 0 }  ][line width=0.75]    (10.93,-3.29) .. controls (6.95,-1.4) and (3.31,-0.3) .. (0,0) .. controls (3.31,0.3) and (6.95,1.4) .. (10.93,3.29)   ;
\draw [color={rgb, 255:red, 245; green, 166; blue, 35 }  ,draw opacity=1 ]   (372,251) ;
\draw [shift={(372,251)}, rotate = 0] [color={rgb, 255:red, 245; green, 166; blue, 35 }  ,draw opacity=1 ][fill={rgb, 255:red, 245; green, 166; blue, 35 }  ,fill opacity=1 ][line width=0.75]      (0, 0) circle [x radius= 3.35, y radius= 3.35]   ;
\draw    (325,221) -- (365.26,244.01) ;
\draw [shift={(367,245)}, rotate = 209.74] [color={rgb, 255:red, 0; green, 0; blue, 0 }  ][line width=0.75]    (10.93,-3.29) .. controls (6.95,-1.4) and (3.31,-0.3) .. (0,0) .. controls (3.31,0.3) and (6.95,1.4) .. (10.93,3.29)   ;
\draw  [dash pattern={on 0.84pt off 2.51pt}]  (296.5,250) .. controls (251,304) and (226,209) .. (293,234) ;

\draw (221.32,160.65) node [anchor=north west][inner sep=0.75pt]    {$1$};
\draw (81.62,167.15) node [anchor=north west][inner sep=0.75pt]    {$2$};
\draw (170.66,98.45) node [anchor=north west][inner sep=0.75pt]    {$4$};
\draw (127.48,99.9) node [anchor=north west][inner sep=0.75pt]    {$3$};
\draw (535.32,159.65) node [anchor=north west][inner sep=0.75pt]    {$1$};
\draw (395.62,166.15) node [anchor=north west][inner sep=0.75pt]    {$2$};
\draw (484.66,97.45) node [anchor=north west][inner sep=0.75pt]    {$4$};
\draw (441.48,98.9) node [anchor=north west][inner sep=0.75pt]    {$3$};

\draw    (121.01,107.25) ;
\draw [shift={(121.01,107.25)}, rotate = 0] [color={rgb, 255:red, 0; green, 0; blue, 0 }  ][fill={rgb, 255:red, 0; green, 0; blue, 0 }  ][line width=0.75]      (0, 0) circle [x radius= 3.35, y radius= 3.35]   ;
\draw    (219.32,157.25) ;
\draw [shift={(219.32,157.25)}, rotate = 0] [color={rgb, 255:red, 0; green, 0; blue, 0 }  ][fill={rgb, 255:red, 0; green, 0; blue, 0 }  ][line width=0.75]      (0, 0) circle [x radius= 3.35, y radius= 3.35]   ;
\draw    (188.27,107.25) ;
\draw [shift={(188.27,107.25)}, rotate = 0] [color={rgb, 255:red, 0; green, 0; blue, 0 }  ][fill={rgb, 255:red, 0; green, 0; blue, 0 }  ][line width=0.75]      (0, 0) circle [x radius= 3.35, y radius= 3.35]   ;
\draw    (95.96,165.25) ;
\draw [shift={(95.96,165.25)}, rotate = 0] [color={rgb, 255:red, 0; green, 0; blue, 0 }  ][fill={rgb, 255:red, 0; green, 0; blue, 0 }  ][line width=0.75]      (0, 0) circle [x radius= 3.35, y radius= 3.35]   ;

\draw    (435.01,106.25) ;
\draw [shift={(435.01,106.25)}, rotate = 0] [color={rgb, 255:red, 0; green, 0; blue, 0 }  ][fill={rgb, 255:red, 0; green, 0; blue, 0 }  ][line width=0.75]      (0, 0) circle [x radius= 3.35, y radius= 3.35]   ;
\draw    (533.32,156.25) ;
\draw [shift={(533.32,156.25)}, rotate = 0] [color={rgb, 255:red, 0; green, 0; blue, 0 }  ][fill={rgb, 255:red, 0; green, 0; blue, 0 }  ][line width=0.75]      (0, 0) circle [x radius= 3.35, y radius= 3.35]   ;
\draw    (502.27,106.25) ;
\draw [shift={(502.27,106.25)}, rotate = 0] [color={rgb, 255:red, 0; green, 0; blue, 0 }  ][fill={rgb, 255:red, 0; green, 0; blue, 0 }  ][line width=0.75]      (0, 0) circle [x radius= 3.35, y radius= 3.35]   ;
\draw    (409.96,164.25) ;
\draw [shift={(409.96,164.25)}, rotate = 0] [color={rgb, 255:red, 0; green, 0; blue, 0 }  ][fill={rgb, 255:red, 0; green, 0; blue, 0 }  ][line width=0.75]      (0, 0) circle [x radius= 3.35, y radius= 3.35]   ;

\end{tikzpicture}

\end{center}

{\color{black} Although the underlying quivers of their Hom-Ext quivers are isomorphic as quivers with relations, their Hom-Ext quivers are not isomorphic because the module corresponding to the blue arc is left regular in one exceptional collection and right regular in the other, and clearly one diagram can not be twisted to the other.}
\end{exmp}

\begin{exmp}\label{exmp: classification of H-E quivers}
   {\color{black} Consider the quiver 

    \[
\xymatrix{
1 \ar[r] \ar@/^2pc/[rr] & 2 \ar[r] & 3
	}
\]

\noindent
Since there are 8 exceptional arc diagrams up to $2\pi$ twist of the inner boundary component (Theorem 6.4 in \cite{igusa2024clusters}), there are 4 exceptional arc diagrams up to any twist of any boundary component. From the previous theorem, there are 4 {\color{black} non-isomorphic} Hom-Ext quivers: }

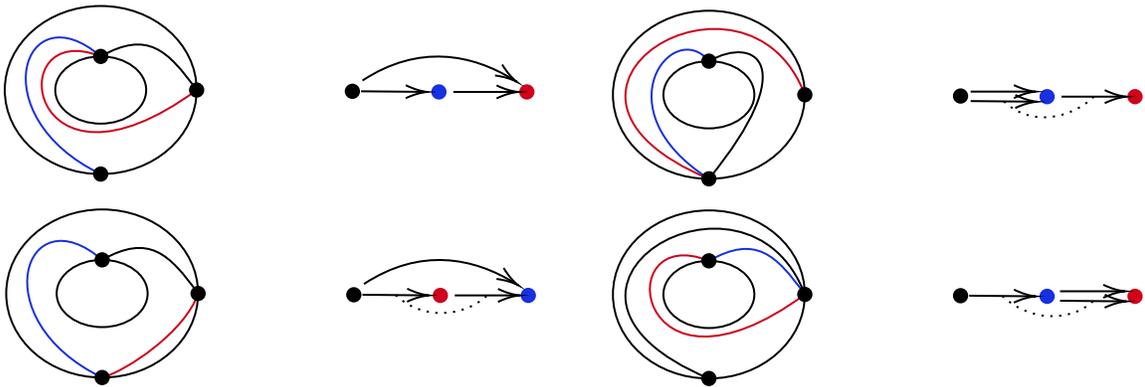
\begin{center}

\tikzset{every picture/.style={line width=0.75pt}} 

\begin{tikzpicture}[x=0.75pt,y=0.75pt,yscale=-1,xscale=1]

\draw   (94.41,64.92) .. controls (94.41,55.55) and (104.67,47.95) .. (117.34,47.95) .. controls (130,47.95) and (140.27,55.55) .. (140.27,64.92) .. controls (140.27,74.29) and (130,81.88) .. (117.34,81.88) .. controls (104.67,81.88) and (94.41,74.29) .. (94.41,64.92)(68.96,64.92) .. controls (68.96,41.49) and (90.62,22.5) .. (117.34,22.5) .. controls (144.06,22.5) and (165.72,41.49) .. (165.72,64.92) .. controls (165.72,88.34) and (144.06,107.33) .. (117.34,107.33) .. controls (90.62,107.33) and (68.96,88.34) .. (68.96,64.92) ;

\draw    (117.34,47.95) .. controls (140.99,34.62) and (151.17,44.82) .. (165.72,64.92) ;
\draw [color={rgb, 255:red, 208; green, 2; blue, 27 }  ,draw opacity=1 ]   (117.34,47.95) .. controls (66.81,27.6) and (79.17,128.38) .. (165.72,64.92) ;
\draw [color={rgb, 255:red, 22; green, 48; blue, 226 }  ,draw opacity=1 ]   (117.34,107.33) .. controls (56.62,77.99) and (78.44,14.21) .. (117.34,47.95) ;
\draw    (244.27,65.55) ;
\draw [shift={(244.27,65.55)}, rotate = 0] [color={rgb, 255:red, 0; green, 0; blue, 0 }  ][fill={rgb, 255:red, 0; green, 0; blue, 0 }  ][line width=0.75]      (0, 0) circle [x radius= 3.35, y radius= 3.35]   ;
\draw [color={rgb, 255:red, 22; green, 48; blue, 226 }  ,draw opacity=1 ]   (287.91,65.87) ;
\draw [shift={(287.91,65.87)}, rotate = 0] [color={rgb, 255:red, 22; green, 48; blue, 226 }  ,draw opacity=1 ][fill={rgb, 255:red, 22; green, 48; blue, 226 }  ,fill opacity=1 ][line width=0.75]      (0, 0) circle [x radius= 3.35, y radius= 3.35]   ;
\draw [color={rgb, 255:red, 208; green, 2; blue, 27 }  ,draw opacity=1 ]   (332.27,65.87) ;
\draw [shift={(332.27,65.87)}, rotate = 0] [color={rgb, 255:red, 208; green, 2; blue, 27 }  ,draw opacity=1 ][fill={rgb, 255:red, 208; green, 2; blue, 27 }  ,fill opacity=1 ][line width=0.75]      (0, 0) circle [x radius= 3.35, y radius= 3.35]   ;
\draw    (248.63,65.55) -- (280.09,65.85) ;
\draw [shift={(282.09,65.87)}, rotate = 180.55] [color={rgb, 255:red, 0; green, 0; blue, 0 }  ][line width=0.75]    (10.93,-3.29) .. controls (6.95,-1.4) and (3.31,-0.3) .. (0,0) .. controls (3.31,0.3) and (6.95,1.4) .. (10.93,3.29)   ;
\draw    (295.18,65.87) -- (325.91,65.87) ;
\draw [shift={(327.91,65.87)}, rotate = 180] [color={rgb, 255:red, 0; green, 0; blue, 0 }  ][line width=0.75]    (10.93,-3.29) .. controls (6.95,-1.4) and (3.31,-0.3) .. (0,0) .. controls (3.31,0.3) and (6.95,1.4) .. (10.93,3.29)   ;
\draw    (249.36,60.13) .. controls (277.72,41.48) and (304.71,46.77) .. (325.59,59.76) ;
\draw [shift={(327.18,60.77)}, rotate = 213.03] [color={rgb, 255:red, 0; green, 0; blue, 0 }  ][line width=0.75]    (10.93,-3.29) .. controls (6.95,-1.4) and (3.31,-0.3) .. (0,0) .. controls (3.31,0.3) and (6.95,1.4) .. (10.93,3.29)   ;
\draw   (95.13,167.61) .. controls (95.13,158.24) and (105.4,150.64) .. (118.07,150.64) .. controls (130.73,150.64) and (141,158.24) .. (141,167.61) .. controls (141,176.98) and (130.73,184.58) .. (118.07,184.58) .. controls (105.4,184.58) and (95.13,176.98) .. (95.13,167.61)(69.68,167.61) .. controls (69.68,144.18) and (91.35,125.19) .. (118.07,125.19) .. controls (144.79,125.19) and (166.45,144.18) .. (166.45,167.61) .. controls (166.45,191.04) and (144.79,210.03) .. (118.07,210.03) .. controls (91.35,210.03) and (69.68,191.04) .. (69.68,167.61) ;

\draw    (118.07,150.64) .. controls (141.72,137.31) and (151.9,147.52) .. (166.45,167.61) ;
\draw [color={rgb, 255:red, 208; green, 2; blue, 27 }  ,draw opacity=1 ]   (118.07,210.03) .. controls (130.08,206.84) and (159.9,187.06) .. (166.45,167.61) ;
\draw [color={rgb, 255:red, 22; green, 48; blue, 226 }  ,draw opacity=1 ]   (118.07,210.03) .. controls (57.35,180.69) and (79.17,116.9) .. (118.07,150.64) ;
\draw [color={rgb, 255:red, 0; green, 0; blue, 0 }  ,draw opacity=1 ]   (245,168.25) ;
\draw [shift={(245,168.25)}, rotate = 0] [color={rgb, 255:red, 0; green, 0; blue, 0 }  ,draw opacity=1 ][fill={rgb, 255:red, 0; green, 0; blue, 0 }  ,fill opacity=1 ][line width=0.75]      (0, 0) circle [x radius= 3.35, y radius= 3.35]   ;
\draw [color={rgb, 255:red, 208; green, 2; blue, 27 }  ,draw opacity=1 ]   (288.63,168.57) ;
\draw [shift={(288.63,168.57)}, rotate = 0] [color={rgb, 255:red, 208; green, 2; blue, 27 }  ,draw opacity=1 ][fill={rgb, 255:red, 208; green, 2; blue, 27 }  ,fill opacity=1 ][line width=0.75]      (0, 0) circle [x radius= 3.35, y radius= 3.35]   ;
\draw [color={rgb, 255:red, 22; green, 48; blue, 226 }  ,draw opacity=1 ]   (333,168.57) ;
\draw [shift={(333,168.57)}, rotate = 0] [color={rgb, 255:red, 22; green, 48; blue, 226 }  ,draw opacity=1 ][fill={rgb, 255:red, 22; green, 48; blue, 226 }  ,fill opacity=1 ][line width=0.75]      (0, 0) circle [x radius= 3.35, y radius= 3.35]   ;
\draw    (249.36,168.25) -- (280.82,168.55) ;
\draw [shift={(282.82,168.57)}, rotate = 180.55] [color={rgb, 255:red, 0; green, 0; blue, 0 }  ][line width=0.75]    (10.93,-3.29) .. controls (6.95,-1.4) and (3.31,-0.3) .. (0,0) .. controls (3.31,0.3) and (6.95,1.4) .. (10.93,3.29)   ;
\draw    (295.91,168.57) -- (326.64,168.57) ;
\draw [shift={(328.64,168.57)}, rotate = 180] [color={rgb, 255:red, 0; green, 0; blue, 0 }  ][line width=0.75]    (10.93,-3.29) .. controls (6.95,-1.4) and (3.31,-0.3) .. (0,0) .. controls (3.31,0.3) and (6.95,1.4) .. (10.93,3.29)   ;
\draw    (250.09,162.83) .. controls (278.45,144.17) and (305.43,149.46) .. (326.31,162.45) ;
\draw [shift={(327.91,163.46)}, rotate = 213.03] [color={rgb, 255:red, 0; green, 0; blue, 0 }  ][line width=0.75]    (10.93,-3.29) .. controls (6.95,-1.4) and (3.31,-0.3) .. (0,0) .. controls (3.31,0.3) and (6.95,1.4) .. (10.93,3.29)   ;
\draw   (401.13,67.3) .. controls (401.13,57.93) and (411.4,50.34) .. (424.07,50.34) .. controls (436.73,50.34) and (447,57.93) .. (447,67.3) .. controls (447,76.67) and (436.73,84.27) .. (424.07,84.27) .. controls (411.4,84.27) and (401.13,76.67) .. (401.13,67.3)(375.68,67.3) .. controls (375.68,43.88) and (397.35,24.89) .. (424.07,24.89) .. controls (450.79,24.89) and (472.45,43.88) .. (472.45,67.3) .. controls (472.45,90.73) and (450.79,109.72) .. (424.07,109.72) .. controls (397.35,109.72) and (375.68,90.73) .. (375.68,67.3) ;

\draw    (424.07,50.34) .. controls (447.72,37.01) and (470.99,52.31) .. (424.07,109.72) ;
\draw [color={rgb, 255:red, 208; green, 2; blue, 27 }  ,draw opacity=1 ]   (472.45,67.3) .. controls (443.35,-9.56) and (317.53,65.07) .. (424.07,109.72) ;
\draw [color={rgb, 255:red, 22; green, 48; blue, 226 }  ,draw opacity=1 ]   (424.07,109.72) .. controls (369,74.5) and (406.26,28.08) .. (424.07,50.34) ;
\draw    (551,67.94) ;
\draw [shift={(551,67.94)}, rotate = 0] [color={rgb, 255:red, 0; green, 0; blue, 0 }  ][fill={rgb, 255:red, 0; green, 0; blue, 0 }  ][line width=0.75]      (0, 0) circle [x radius= 3.35, y radius= 3.35]   ;
\draw [color={rgb, 255:red, 22; green, 48; blue, 226 }  ,draw opacity=1 ]   (594.63,68.26) ;
\draw [shift={(594.63,68.26)}, rotate = 0] [color={rgb, 255:red, 22; green, 48; blue, 226 }  ,draw opacity=1 ][fill={rgb, 255:red, 22; green, 48; blue, 226 }  ,fill opacity=1 ][line width=0.75]      (0, 0) circle [x radius= 3.35, y radius= 3.35]   ;
\draw [color={rgb, 255:red, 208; green, 2; blue, 27 }  ,draw opacity=1 ]   (639,68.26) ;
\draw [shift={(639,68.26)}, rotate = 0] [color={rgb, 255:red, 208; green, 2; blue, 27 }  ,draw opacity=1 ][fill={rgb, 255:red, 208; green, 2; blue, 27 }  ,fill opacity=1 ][line width=0.75]      (0, 0) circle [x radius= 3.35, y radius= 3.35]   ;
\draw    (556.09,70.49) -- (587.54,70.79) ;
\draw [shift={(589.54,70.81)}, rotate = 180.55] [color={rgb, 255:red, 0; green, 0; blue, 0 }  ][line width=0.75]    (10.93,-3.29) .. controls (6.95,-1.4) and (3.31,-0.3) .. (0,0) .. controls (3.31,0.3) and (6.95,1.4) .. (10.93,3.29)   ;
\draw    (601.91,68.26) -- (632.64,68.26) ;
\draw [shift={(634.64,68.26)}, rotate = 180] [color={rgb, 255:red, 0; green, 0; blue, 0 }  ][line width=0.75]    (10.93,-3.29) .. controls (6.95,-1.4) and (3.31,-0.3) .. (0,0) .. controls (3.31,0.3) and (6.95,1.4) .. (10.93,3.29)   ;
\draw   (401.13,168.08) .. controls (401.13,158.71) and (411.4,151.12) .. (424.07,151.12) .. controls (436.73,151.12) and (447,158.71) .. (447,168.08) .. controls (447,177.45) and (436.73,185.05) .. (424.07,185.05) .. controls (411.4,185.05) and (401.13,177.45) .. (401.13,168.08)(375.68,168.08) .. controls (375.68,144.66) and (397.35,125.67) .. (424.07,125.67) .. controls (450.79,125.67) and (472.45,144.66) .. (472.45,168.08) .. controls (472.45,191.51) and (450.79,210.5) .. (424.07,210.5) .. controls (397.35,210.5) and (375.68,191.51) .. (375.68,168.08) ;

\draw [color={rgb, 255:red, 22; green, 48; blue, 226 }  ,draw opacity=1 ]   (424.07,151.12) .. controls (447.72,137.79) and (457.9,147.99) .. (472.45,168.08) ;
\draw [color={rgb, 255:red, 208; green, 2; blue, 27 }  ,draw opacity=1 ]   (424.07,151.12) .. controls (373.53,130.77) and (385.9,231.55) .. (472.45,168.08) ;
\draw    (551,168.72) ;
\draw [shift={(551,168.72)}, rotate = 0] [color={rgb, 255:red, 0; green, 0; blue, 0 }  ][fill={rgb, 255:red, 0; green, 0; blue, 0 }  ][line width=0.75]      (0, 0) circle [x radius= 3.35, y radius= 3.35]   ;
\draw [color={rgb, 255:red, 22; green, 48; blue, 226 }  ,draw opacity=1 ]   (594.63,169.04) ;
\draw [shift={(594.63,169.04)}, rotate = 0] [color={rgb, 255:red, 22; green, 48; blue, 226 }  ,draw opacity=1 ][fill={rgb, 255:red, 22; green, 48; blue, 226 }  ,fill opacity=1 ][line width=0.75]      (0, 0) circle [x radius= 3.35, y radius= 3.35]   ;
\draw [color={rgb, 255:red, 208; green, 2; blue, 27 }  ,draw opacity=1 ]   (639,169.04) ;
\draw [shift={(639,169.04)}, rotate = 0] [color={rgb, 255:red, 208; green, 2; blue, 27 }  ,draw opacity=1 ][fill={rgb, 255:red, 208; green, 2; blue, 27 }  ,fill opacity=1 ][line width=0.75]      (0, 0) circle [x radius= 3.35, y radius= 3.35]   ;
\draw    (555.36,168.72) -- (586.82,169.02) ;
\draw [shift={(588.82,169.04)}, rotate = 180.55] [color={rgb, 255:red, 0; green, 0; blue, 0 }  ][line width=0.75]    (10.93,-3.29) .. controls (6.95,-1.4) and (3.31,-0.3) .. (0,0) .. controls (3.31,0.3) and (6.95,1.4) .. (10.93,3.29)   ;
\draw  [dash pattern={on 0.84pt off 2.51pt}]  (266.09,168.41) .. controls (281.36,182.6) and (302.45,178.13) .. (312.27,168.57) ;
\draw    (556.09,65.71) -- (586.82,65.71) ;
\draw [shift={(588.82,65.71)}, rotate = 180] [color={rgb, 255:red, 0; green, 0; blue, 0 }  ][line width=0.75]    (10.93,-3.29) .. controls (6.95,-1.4) and (3.31,-0.3) .. (0,0) .. controls (3.31,0.3) and (6.95,1.4) .. (10.93,3.29)   ;
\draw  [dash pattern={on 0.84pt off 2.51pt}]  (572.82,70.65) .. controls (588.09,84.84) and (608.45,77.83) .. (618.27,68.26) ;
\draw [color={rgb, 255:red, 0; green, 0; blue, 0 }  ,draw opacity=1 ]   (472.45,168.08) .. controls (443.35,91.22) and (317.53,165.85) .. (424.07,210.5) ;
\draw    (601.18,171.27) -- (632.64,171.57) ;
\draw [shift={(634.64,171.59)}, rotate = 180.55] [color={rgb, 255:red, 0; green, 0; blue, 0 }  ][line width=0.75]    (10.93,-3.29) .. controls (6.95,-1.4) and (3.31,-0.3) .. (0,0) .. controls (3.31,0.3) and (6.95,1.4) .. (10.93,3.29)   ;
\draw    (601.18,166.49) -- (631.91,166.49) ;
\draw [shift={(633.91,166.49)}, rotate = 180] [color={rgb, 255:red, 0; green, 0; blue, 0 }  ][line width=0.75]    (10.93,-3.29) .. controls (6.95,-1.4) and (3.31,-0.3) .. (0,0) .. controls (3.31,0.3) and (6.95,1.4) .. (10.93,3.29)   ;
\draw  [dash pattern={on 0.84pt off 2.51pt}]  (572.09,168.88) .. controls (587.36,183.07) and (608.09,181) .. (617.91,171.43) ;

\draw    (165.72,64.92) ;
\draw [shift={(165.72,64.92)}, rotate = 0] [color={rgb, 255:red, 0; green, 0; blue, 0 }  ][fill={rgb, 255:red, 0; green, 0; blue, 0 }  ][line width=0.75]      (0, 0) circle [x radius= 3.35, y radius= 3.35]   ;
\draw    (117.34,107.33) ;
\draw [shift={(117.34,107.33)}, rotate = 0] [color={rgb, 255:red, 0; green, 0; blue, 0 }  ][fill={rgb, 255:red, 0; green, 0; blue, 0 }  ][line width=0.75]      (0, 0) circle [x radius= 3.35, y radius= 3.35]   ;
\draw    (117.34,47.95) ;
\draw [shift={(117.34,47.95)}, rotate = 0] [color={rgb, 255:red, 0; green, 0; blue, 0 }  ][fill={rgb, 255:red, 0; green, 0; blue, 0 }  ][line width=0.75]      (0, 0) circle [x radius= 3.35, y radius= 3.35]   ;

\draw    (166.45,167.61) ;
\draw [shift={(166.45,167.61)}, rotate = 0] [color={rgb, 255:red, 0; green, 0; blue, 0 }  ][fill={rgb, 255:red, 0; green, 0; blue, 0 }  ][line width=0.75]      (0, 0) circle [x radius= 3.35, y radius= 3.35]   ;
\draw    (118.07,210.03) ;
\draw [shift={(118.07,210.03)}, rotate = 0] [color={rgb, 255:red, 0; green, 0; blue, 0 }  ][fill={rgb, 255:red, 0; green, 0; blue, 0 }  ][line width=0.75]      (0, 0) circle [x radius= 3.35, y radius= 3.35]   ;
\draw    (118.07,150.64) ;
\draw [shift={(118.07,150.64)}, rotate = 0] [color={rgb, 255:red, 0; green, 0; blue, 0 }  ][fill={rgb, 255:red, 0; green, 0; blue, 0 }  ][line width=0.75]      (0, 0) circle [x radius= 3.35, y radius= 3.35]   ;

\draw    (472.45,67.3) ;
\draw [shift={(472.45,67.3)}, rotate = 0] [color={rgb, 255:red, 0; green, 0; blue, 0 }  ][fill={rgb, 255:red, 0; green, 0; blue, 0 }  ][line width=0.75]      (0, 0) circle [x radius= 3.35, y radius= 3.35]   ;
\draw    (424.07,109.72) ;
\draw [shift={(424.07,109.72)}, rotate = 0] [color={rgb, 255:red, 0; green, 0; blue, 0 }  ][fill={rgb, 255:red, 0; green, 0; blue, 0 }  ][line width=0.75]      (0, 0) circle [x radius= 3.35, y radius= 3.35]   ;
\draw    (424.07,50.34) ;
\draw [shift={(424.07,50.34)}, rotate = 0] [color={rgb, 255:red, 0; green, 0; blue, 0 }  ][fill={rgb, 255:red, 0; green, 0; blue, 0 }  ][line width=0.75]      (0, 0) circle [x radius= 3.35, y radius= 3.35]   ;

\draw    (472.45,168.08) ;
\draw [shift={(472.45,168.08)}, rotate = 0] [color={rgb, 255:red, 0; green, 0; blue, 0 }  ][fill={rgb, 255:red, 0; green, 0; blue, 0 }  ][line width=0.75]      (0, 0) circle [x radius= 3.35, y radius= 3.35]   ;
\draw    (424.07,210.5) ;
\draw [shift={(424.07,210.5)}, rotate = 0] [color={rgb, 255:red, 0; green, 0; blue, 0 }  ][fill={rgb, 255:red, 0; green, 0; blue, 0 }  ][line width=0.75]      (0, 0) circle [x radius= 3.35, y radius= 3.35]   ;
\draw    (424.07,151.12) ;
\draw [shift={(424.07,151.12)}, rotate = 0] [color={rgb, 255:red, 0; green, 0; blue, 0 }  ][fill={rgb, 255:red, 0; green, 0; blue, 0 }  ][line width=0.75]      (0, 0) circle [x radius= 3.35, y radius= 3.35]   ;

\end{tikzpicture}

\end{center}
\end{exmp}

\section{Hom-Ext quiver and derived autoequivalences}\label{sec: H-E quiver and autoequivalences}

Our main goal for this section is to show that two exceptional collections in type $\tilde{\mathbb{A}}$ have {\color{black} isomorphic} Hom-Ext quivers if and only if they differ by an autoequivalence of the bounded derived category. Until stated otherwise, let $\mathscr{D}$ denote the bounded derived category of representations of a finite acyclic quiver.

\begin{defn}
\leavevmode
\begin{enumerate}
\item An object $E \in \mathscr{D}$ is called exceptional if Hom$^{\bullet}(E,E) = \Bbbk \cdot \mathbbm{1}_{E}$, where $$\text{Hom}^{\bullet}(X,Y) = \displaystyle\bigoplus_{i \in \mathbb{Z}} \Sigma^{-i}\,\text{Hom}(X,\Sigma^{i}\,Y).$$

\item (\cite{broomhead2017discrete}) A sequence of objects $E_* = (E_1, E_2, \dots, E_k)$ of $\mathscr{D}$ is called an \textbf{exceptional $k$-cycle} ($k \geq 2$) if 

\begin{enumerate}
\item $E_i$ is exceptional for all $i$.
\item There exist $k_i \in \mathbb{Z}$ such that $\mathbb{S}(E_i) \cong \Sigma^{k_i}(E_{i+1})$ for all $i$, where $\mathbb{S} := \Sigma \tau_{\mathscr{D}}$ is the Serre functor and we take the convention $E_{k+1} := E_1$.
\item Hom$^{\bullet}(E_i,E_j) = 0$ unless $j = i$ or $j = i+1$.
\end{enumerate}
\end{enumerate}
\end{defn}

\begin{rem}
    We also can define an \textbf{exceptional 1-cycle} as an object $E$ such that there is an integer $l$ so that $\mathbb{S}(E) \cong \Sigma^l E$ and Hom$^{\bullet}(E,E) = \Bbbk \oplus \Sigma^{-l}\Bbbk$.
\end{rem}

\begin{exmp}
Consider the quiver 
\[
\xymatrixcolsep{10pt}
\xymatrix{
1 \ar[r]_{\alpha_1} \ar@/^1pc/[rr]^{\alpha_3} & 2 \ar[r]_{\alpha_2} & 3 }\]

Then the tuple of string modules $(e_2,\alpha_3)$ forms an exceptional 2-cycle when these modules are viewed as stalk complexes centered at $0$ in $\mathscr{D}$. For any choice of $\lambda \in \Bbbk^{\times}$, the band module \[
\xymatrixcolsep{10pt}
\xymatrix{
\Bbbk \ar[r]_1 \ar@/^1pc/[rr]^{\lambda} & \Bbbk \ar[r]_1 & \Bbbk }\]

associated to the walk $\alpha_1\alpha_2\alpha_3^{-1}$ is an exceptional 1-cycle. When $\lambda = 0$, the above module is the string module associated to the walk $\alpha_1\alpha_2$, and is also an exceptional 1-cycle.
\end{exmp}

It turns out that in type $\tilde{\mathbb{A}}$, all exceptional $k$-cycles are completely classified.

\begin{thm}[\cite {guo2020exceptional} Theorem 1.3]

Let $Q$ be a Euclidean quiver. {\color{black} Up to shift, the quasi-simples at the mouths of homogeneous tubes give all the exceptional 1-cycles.} For each non-homogeneous tube $T$ of $\Bbbk Q$ with rank $r \geq 2$, we have that $(E, \tau E, \cdots, \tau^{r-1} E)$ is an exceptional cycle in $\mathscr{D}$, where $E$ is a quasi-simple in $T$. When $T$ runs over all non-homogeneous tubes, this gives all the exceptional cycles of length at least 2 in $\mathscr{D}$ up to shift {\color{black} of each position} and rotation. $\hfill \qed$
\end{thm}

In order to prove the main result of this section, we will first provide a generating set for Aut$(\mathscr{D})$. Note that this has already been done in \cite{opper2019auto} for gentle algebras; however, here we will provide an explicit list of twist functors which generate the orientation-preserving homeomorphisms of the annulus. To do this, we will need to introduce a few important automorphisms.

\begin{defn}
For any $X \in \mathscr{D}$, define a functor $F_X: \mathscr{D} \rightarrow \mathscr{D}$ by $Y \mapsto \text{Hom}^{\bullet}(X,Y) \otimes_{\Bbbk} X$. There is a canonical evaluation morphism of functors $F_X \rightarrow \mathbbm{1}_{\mathscr{D}}$, which extends to direct sums in the following sense. For a sequence of objects $X_* = (X_1, X_2, \dots, X_k)$, we define the associated \textbf{twist functor} $T_{X_*}$ as the cone of this evaluation morphism: $$F_{X_*} \rightarrow \mathbbm{1}_{\mathscr{D}} \rightarrow T_{X_*} \rightarrow \Sigma F_{X_*} \,\,\,\,\,\, \text{with} \,\,\,\,\,\, F_{X_*} = \bigoplus_{i=1}^k F_{X_i}$$
\end{defn}

Twist functors satisfy some nice properties.

\begin{thm}
\leavevmode
\begin{enumerate}
\item (\cite{broomhead2017discrete} Theorem 4.5)  If $E_* = (E_1, E_2, \dots, E_k)$ is an exceptional $k$-cycle in $\mathscr{D}$, then $T_{E_*}$ is an autoequivalence of $\mathscr{D}$.

\item (\cite{huybrechts2006fourier}, Lemma 1.30) The Serre functor commutes with all auto equivalences, including $T_{E_*}$.
\end{enumerate}
\end{thm}

\begin{rem}\label{rem: twist functor commutes with}
Since all autoequivalences commute with $\mathbb{S}$ and $\Sigma$, they also commute with $\tau_{\mathscr{D}}$.
\end{rem}

We will now switch our focus to the case in which $\mathscr{D}$ denotes the bounded derived category of an acyclic quiver $Q$ of type $\tilde{\mathbb{A}}_{n-1}$. In our case, we have two non-homogeneous tubes, the left (right) tube denoted by $\mathscr{L}$ ($\mathscr{R}$).

\begin{lem}\label{lem: twists are twists}
Let $E_{*}$ be an exceptional $k$-cycle in $\Bbbk Q$. 

\begin{enumerate}
\item If $E_*$ consists of the quasi-simple modules in $\mathscr{L}$, then $T_{\mathscr{L}} := T_{E_*}$ acts on the objects of $\mathscr{D}$ in the same way as a counterclockwise elementary twist of the outer boundary component of $A_Q$.
\item If $E_*$ consists of the quasi-simple modules in $\mathscr{R}$, then $T_{\mathscr{R}} := T_{E_*}$ acts on the objects of $\mathscr{D}$ in the same way as a clockwise elementary twist of the inner boundary component of $A_Q$.
\end{enumerate}
\end{lem} 

\begin{proof}
We will prove 1 as the proof of 2 is analogous. Moreover, we will assume that the left tube has rank at least 2, as again, the proof in the case the rank is 1 is analogous. Whence, let $P_i$ be the projective module at vertex $i$. Then there are precisely two left regular modules with which $P_i$ has non-trivial Hom/Ext. Call these two regulars $E_i$ and $E_{i-1}$ and assume Hom$(P_i,E_i)\cong \Bbbk$ and Ext$(E_{i-1},P_i)\cong \Bbbk$. In particular, this means that as strings, $P_i = \alpha_i\alpha_{i+1}\cdots\alpha_{i+k}$ and $E_{i-1} = \beta_j\beta_{j+1}\cdots\alpha_{i-1}$. We have 

\begin{align*}
F_{\mathscr{L}}(P_i) &:= \text{Hom}^{\bullet}(E_*,P_i)\otimes E_* \\
&= \text{Hom}^{\bullet}(E_{i-1},P_i) \otimes E_{i-1} \\
&= \Sigma^{-1}\text{Hom}(E_{i-1},\Sigma P_i) \otimes E_{i-1} \\
&\cong \Sigma^{-1}E_{i-1}
\end{align*}

To compute the twist, we analyze the following triangle

$$\Sigma^{-1} E_{i-1} \rightarrow P_i \rightarrow T_{\mathscr{L}}(P_i) \rightarrow E_{i-1}$$

The unique completion of this triangle is the string $\beta_j\beta_{j+1}\cdots\alpha_{i-1}\alpha_i\cdots\alpha_{i+k}$. Since $E_{i-1}$ is a quasi-simple, this is equivalent to adding a hook at the start of $P_i$, which is equivalent to moving $P_i$ one position up a ray. Since twist functors commute with $\Sigma$ and $\tau$ by Remark \ref{rem: twist functor commutes with}, we have $T_{\mathscr{L}}$ moves each non-regular object one space {\color{black} up a ray}, and $T^{-1}_{\mathscr{L}}$ moves each non-regular object one space {\color{black} down a ray}; which is equivalent to a counterclockwise elementary twist of the outer boundary component by the dual statement of Lemma 4.7 in \cite{maresca2022combinatorics}.

Next, we see that $T_{\mathscr{L}}$ is the identity on all other tubes, since the modules in $\mathscr{L}$ and the modules in all other tubes are Hom-Ext orthogonal. This is equivalent to the action of a (counter)clockwise twist of the outer boundary component. 

Finally, let $X_*= (X, \tau X, \tau^2 X, \dots, \tau^k X)$ be a $\tau$-cycle of modules in $\mathscr{L}$. Then since $T_{\mathscr{L}}$ is an autoequivalence, we have that $T_{\mathscr{L}}(X_*)$ is a cyclic permutation of $X_*$ up to shift. Thus to understand what $T_{\mathscr{L}}$ does to $\mathscr{L}$, it suffices to analyze it on $E_*$:

\begin{align*}
F_{\mathscr{L}}(E_i) &:= \text{Hom}^{\bullet}(E_*,E_i)\otimes E_* \\
&= \text{Hom}^{\bullet}(E_{i-1},E_i) \otimes E_{i-1} \oplus \text{Hom}^{\bullet}(E_{i},E_i) \otimes E_{i}\\
&= \Sigma^{-1}\text{Hom}(E_{i-1},\Sigma E_i) \otimes E_{i-1} \oplus \text{Hom}(E_{i},E_i) \otimes E_{i} \\
&\cong \Sigma^{-1}E_{i-1} \oplus E_i\\ &\cong \Sigma^{-1}\tau^{-1} E_i \oplus E_i
\end{align*}

Thus there is a triangle 

$$\Sigma^{-1}\tau^{-1} E_i \oplus E_i \rightarrow E_i \rightarrow T_{\mathscr{L}}(E_i) \rightarrow \tau^{-1} E_i \oplus \Sigma E_i,$$

whose unique completion is $\tau^{-1} E_i$, which is the image of $E_{i}$ after performing an elementary  counterclockwise twist of the outer boundary component by Lemma 4.4 2. (b) in \cite{maresca2022combinatorics}.
\end{proof}

\begin{cor} \label{cor: twist functors act transitively}
The twist functors $T_{\mathcal{R}}$ and $T_{\mathcal{L}}$, together with the shift functor $\Sigma$, act transitively on all objects of $\mathscr{D}$ that are not shifts of regular modules.
\end{cor}

\begin{proof}
This follows immediately from Remark \ref{rem: twist functor commutes with} and Lemma \ref{lem: twists are twists}.
\end{proof}

We will now introduce two more types of autoequivalences that will be important in what is to come. These were defined in \cite{opper2019auto}, and we will follow those notations. First, recall that an automorphism of $\Bbbk Q$ induces an automorphism of $\mathscr{D}$. The first important autoequivalence is the following. Let $f: \Bbbk Q \rightarrow \Bbbk Q$ be an algebra automorphism that fixes every vertex of $Q$ and {\color{black} for each arrow $\alpha$,} $f(\alpha) = \lambda_{\alpha} \cdot \alpha$ for some $\lambda_{\alpha} \in \Bbbk^{\times}$, the multiplicative group of $\Bbbk$. The derived equivalences associated to these automorphisms are called \textbf{rescaling equivalences}, and they form a subgroup of Aut$(\mathscr{D})$.

The next type is called a coordinate transformation and is split into two parts. Let $Q$ denote the Kronecker quiver. Then there is a natural embedding $\text{PGL}_2(\Bbbk) \hookrightarrow \text{Aut}(\mathscr{D})$, and we call the automorphisms in the image \textbf{projective coordinate transformations}. 

On the other hand, for a quiver of type $\tilde{\mathbb{A}}_{n-1}$ where $n >2$, we will define an `affine' coordinate transformation for each subquiver $Q'$ of $Q$ of the form

\[ \begin{xymatrix}{i \ar[rr]^{\alpha} \ar[dr]_{\beta_1} & & j \\ & \cdots \ar[ur]_{\beta_m} & }\end{xymatrix}\]

For any such $Q'$ and $(\lambda_{\alpha}, \lambda_{\beta_1}, \dots, \lambda_{\beta_m})\in (\Bbbk^{\times})^{m+1}$, we define {\color{black} $\beta = \beta_1\cdots\beta_m$,} and an algebra automorphism $\sigma: \Bbbk Q \rightarrow \Bbbk Q$ by

\[ \sigma(\gamma) = \begin{cases} 
      \lambda_{\alpha}\alpha + \lambda_{\beta}\beta & \gamma = \alpha \\
      \lambda_i\beta_i & \gamma = \beta_i \\
   \end{cases}
\]

We call the associated derived equivalence an \textbf{affine coordinate transformation}. We will denote by $\mathscr{K}$ the subgroup of Aut$(\mathscr{D})$ that is generated by rescaling equivalences and coordinate transformations. We note that by \cite[Lemma 5.2]{opper2019auto}, the shift functor $\Sigma$ is an element of $\mathscr{K}$. 

In the next result, we provide an explicit generating set for Aut$(\mathscr{D})$. This has been done more generally in \cite{opper2019auto}; however, here we will provide explicit twist functors and a different proof inspired by the proof of Theorem 5.7 in \cite{broomhead2017discrete}. We require this result to be stated in this form to prove our main result of the section.

\begin{thm}\label{thm: gen set for derived}
The set $\mathscr{K} \cup \{T_{\mathscr{L}}, T_{\mathscr{R}}\}$ is a generating set for Aut$(\mathscr{D})$.
\end{thm}

\begin{proof}
{\color{black}
By Theorem 5.1 in \cite{opper2019auto}, we have that $\mathscr{K}$ is the set of automorphisms of $\mathscr{D}$ generated by rescaling equivalences and coordinate transformations. Moreover, these autoequivalence don't arise from orientation-preserving homeomorphisms of $A_Q$. }

Let $\psi \in \text{Aut}(\mathscr{D})$ and choose an indecomposable object $X \in \mathscr{D}$ such that $X$ is not the shift of a regular module. Then by Corollary \ref{cor: twist functors act transitively} and Remark \ref{rem: twist functor commutes with}, there exist $a,b,c \in \mathbb{Z}$ such that $\Sigma^a T_{\mathscr{L}}^b T_{\mathscr{R}}^c (X) = \psi(X)$. Then the automorphism $\varphi = \Sigma^a T_{\mathscr{L}}^b T_{\mathscr{R}}^c \psi^{-1}$ fixes $X$. Since all autoequivalences commute with $\tau_{\mathscr{D}}$ (Remark \ref{rem: twist functor commutes with}), we have that $\varphi$ fixes the $\tau_{\mathscr{D}}$-orbit of $X$. Since we have a quiver of type $\tilde{\mathbb{A}}$, this implies that $\varphi$ fixes all objects that are not shifts of regular modules {\color{black} or reflects these objects over the $\tau_{\mathscr{D}}$-orbit of $X$. In the latter case, $\varphi$ corresponds to an orientation reversing homeomorphism of $A_Q$ and $\varphi \in \mathscr{K}$}.

{\color{black}
If $\varphi$ fixes all shifts of projectives, it must either fix all shifts of quasi-simple regular modules, and hence all modules in the left and right tubes, or swap the two exceptional tubes in the case they have the same rank, which corresponds again to an orientation reversing homeomorphism of $A_Q$. Thus $\varphi \in \mathscr{K}$ and $\psi \in \langle \mathscr{K}, T_{\mathscr{L}},T_{\mathscr{R}}\rangle$.}
\end{proof}

\begin{cor} \label{cor: iso HE quivers iff derived equiv}
Two exceptional collections $\chi_1$ and $\chi_2$ in rep$Q$ have {\color{black} isomorphic} Hom-Ext quivers if and only if there exists an embedding $\psi: \text{rep}(Q) \hookrightarrow \mathscr{D}$ such that there is an autoequivalence {\color{black} $\varphi\in \langle T_{\mathscr{L}},T_{\mathscr{R}}\rangle$} with $\varphi(\psi(\chi_1)) = \psi(\chi_2). \hfill \qed$ 
\end{cor}

\begin{exmp}{\color{black}
    The four Hom-Ext quivers in Example \ref{exmp: classification of H-E quivers} describe all the exceptional collections of representations of \xymatrix{
1 \ar[r]_{\alpha_1} \ar@/^1pc/[rr]^{\alpha_3} & 2 \ar[r]_{\alpha_2} & 3 } up to autoequivalence of $\mathscr{D}^b(\text{rep}Q)$}. 
\end{exmp}

\begin{question}
The previous corollary shows that studying exceptional collections of modules in type $\tilde{\mathbb{A}}$ up to {\color{black} isomorphic} Hom-Ext quivers is {\color{black} equivalent} to studying exceptional collections up to derived auto-equivalence. We wonder, is this the case in general, or at least in the hereditary case? If so, the Hom-Ext quiver would be a useful tool in studying exceptional collections of modules over representation infinite algebras.
\end{question}

\section{Superquivers} \label{sec: superquivers}
{Hom-Ext quivers are examples of what we call ``superquivers''. An exceptional collection will be a representation of this superquiver. This interpretation leads to a number of questions and directions of future research which we hope readers will find interesting and motivating. One question we are interested in is: Is every superquiver a Hom-Ext quiver for a suitable choice of algebra?

We assemble basic properties of Hom-Ext quivers into this concept.

\begin{defn}\label{defn: superquiver} By a \textbf{superquiver} we mean a quiver $Q$ with relations with two kinds of arrows which we distinguish by \textbf{degrees} which can be 0 or 1. {\color{black}One of the relations is that any path of degree $\ge2$ is zero} where the degree of a path is defined to be the sum of the degrees of its arrows. Thus, the path algebra of $Q$ is $\mathbb Z/2$-graded and the product of any two paths of degree 1 is defined to be 0.

A \textbf{representation} of a superquiver in a triangulated category $\mathcal C$ consists of the following.
\begin{enumerate}
\item For each $v\in Q_0$ we have one object $M_v$ of $\mathcal C$.
\item For every degree 0 arrow $\alpha:v\to w$ we have a morphism $M_\alpha: M_v\to M_w$.
\item For every degree 1 arrow $\beta:v\to w$ we have a morphism $M_\beta:M_v\to \Sigma M_w$ {\color{black}which ``classifies'' an extension $M_w\to E(M_\beta)\to M_v$.}
\end{enumerate}
Given arrows $\alpha:u\to v$ of degree 0 and $\beta:v\to w$ of degree 1, {\color{black}we define $M_{\alpha\beta}:M_u\to \Sigma M_w$ to be the morphism which classifies the pull-back $(M_\alpha)^\ast(E(M_\beta))$ of $E(M(\beta))$ along $M_\alpha:M_u\to M_v$. If $\gamma:w\to t$ is an arrow of degree 0, we define $M_{\gamma\beta}:M_v\to \Sigma M_t$ to be the composition $\Sigma M_\gamma\circ M_\beta:M_v\to \Sigma M_t$ which classifies the push-forward of $E(M_{\beta})$ along $M_{\gamma}$}. We require any path of degree $\ge2$ in $Q$ to go to the zero morphism. This condition will automatically be satisfied if $\mathcal C$ is the bounded derived category of {\color{black} mod$\Lambda$} for $\Lambda$ a hereditary algebra.
\end{defn}
}

{
\begin{defn} \label{defn: superquiver twist}
    Given a superquiver $Q$, we allow selected arrows to be \textbf{frozen} and we define a \textbf{twist} of $Q$ to be another superquiver $\widetilde Q$ having the property that $Q$ {\color{black} and} $\widetilde Q$
 are isomorphic as quivers with relations with the restriction that this isomorphism sends frozen arrows to frozen arrows and the degrees of corresponding frozen arrows are the same. One example is the \textbf{trivial twist} which is defined to be the underlying quiver with relations and frozen arrows having all unfrozen arrows of degree 0 (and frozen arrows keeping the degrees that they had).
 \end{defn}
}

{
Given an exceptional collection, its Hom-Ext quiver is a superquiver with frozen arrows and the exceptional collection is a representation of this superquiver. But not all representations of superquivers are Hom-Ext quivers. We would like to characterize those representations which might be Hom-Ext quivers. One necessary property is that the representation should be ``irreducible'' which will be defined in terms of its ``inner endomorphism algebra.''

\begin{defn}
    Given a representation $M$ of a superquiver $Q$ in a hereditary category $\mathcal C$ we define the \textbf{endomorphism ring} of $M$ to be the $\mathbb Z/2$ graded algebra End$_\ast(M)=\text{End}_0(M)\oplus \text{End}_1(M)$ where $\text{End}_0(M)$ is the endomorphism ring of $M=\bigoplus M_v$ and $\text{End}_1(M)=\text{Hom}(M,\Sigma M)$. Since $\mathcal C$ is hereditary, the composition of any two morphisms in $\text{End}_1(M)$ will be zero. We say that the representation $M$ is \textbf{irreducible} if the morphism $M_\alpha$ for $\alpha\in Q_1$ do not lie in $r^2\text{End}_\ast(M)$. 
   \end{defn}

\textbf{Question}: Is it true that a superquiver is a Hom-Ext quiver if and only if it admits at least one irreducible representation?

\textbf{Conjecture}: The irreducible $\mathcal C$ representations of any superquiver are unique up to autoequivalences of $\mathcal C$.

For an exceptional collection of type $\tilde{\mathbb{A}}_n$ the frozen arrows of its Hom-Ext quiver are defined to be those whose endpoints are both regular modules. The statement is:

\begin{thm} \label{thm: superquiver twists}
If two exceptional collections for derived equivalent algebras of type $\tilde{\mathbb{A}}_n$ are equivalent by {\color{black} a sequence of twist functors associated to exceptional cycles}, their Hom-Ext quivers are twist equivalent as superquivers with frozen arrows.
\end{thm}

This is a restatement of Corollary \ref{cor: iso HE quivers iff derived equiv} above. We believe the converse of this theorem also holds.

\begin{exmp}
Let $\Lambda$ be the path algebra of the $\tilde{\mathbb{A}}_3$ quiver:
\[
\xymatrix{
1 \ar[r]\ar@/^1pc/[rrr] & 
	2 \ar[r]&
	3 \ar[r] & 4 
	} 
\]

The exceptional collection {\color{black} of simple modules} $\{S_1,S_2,S_3,S_4\}$ has Hom-Ext quiver

\begin{center}
\begin{tabular}{c}

\tikzset{every picture/.style={line width=0.75pt}} 

\begin{tikzpicture}[x=0.75pt,y=0.75pt,yscale=-1,xscale=1]

\draw   (451.91,80.75) .. controls (451.91,68.99) and (463.13,59.45) .. (476.98,59.45) .. controls (490.83,59.45) and (502.05,68.99) .. (502.05,80.75) .. controls (502.05,92.51) and (490.83,102.05) .. (476.98,102.05) .. controls (463.13,102.05) and (451.91,92.51) .. (451.91,80.75)(419.96,80.75) .. controls (419.96,51.34) and (445.49,27.5) .. (476.98,27.5) .. controls (508.47,27.5) and (534,51.34) .. (534,80.75) .. controls (534,110.16) and (508.47,134) .. (476.98,134) .. controls (445.49,134) and (419.96,110.16) .. (419.96,80.75) ;

\draw    (476.98,59.45) .. controls (504.86,42.71) and (516.86,55.53) .. (534,80.75) ;
\draw [color={rgb, 255:red, 245; green, 166; blue, 35 }  ,draw opacity=1 ]   (419.96,80.75) .. controls (429.44,62.02) and (449.47,38.17) .. (476.98,59.45) ;
\draw [color={rgb, 255:red, 208; green, 2; blue, 27 }  ,draw opacity=1 ]   (419.96,80.75) .. controls (431.79,107.22) and (447.12,123.54) .. (476.98,134) ;
\draw [color={rgb, 255:red, 22; green, 48; blue, 226 }  ,draw opacity=1 ]   (476.98,134) .. controls (510.76,116.01) and (520.19,110.98) .. (534,80.75) ;
\draw    (64,70) .. controls (107.56,44.26) and (159.94,45.96) .. (200.77,70.26) ;
\draw [shift={(202,71)}, rotate = 211.37] [color={rgb, 255:red, 0; green, 0; blue, 0 }  ][line width=0.75]    (10.93,-3.29) .. controls (6.95,-1.4) and (3.31,-0.3) .. (0,0) .. controls (3.31,0.3) and (6.95,1.4) .. (10.93,3.29)   ;

\draw (51,70.4) node [anchor=north west][inner sep=0.75pt]    {$S_{1}$};
\draw (102,70.4) node [anchor=north west][inner sep=0.75pt]  [color={rgb, 255:red, 22; green, 48; blue, 226 }  ,opacity=1 ]  {$S_{2}$};
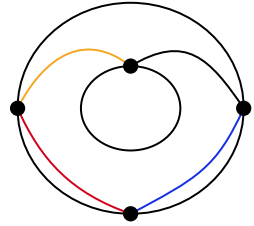
\draw (152,70.4) node [anchor=north west][inner sep=0.75pt]  [color={rgb, 255:red, 208; green, 2; blue, 27 }  ,opacity=1 ]  {$S_{3}$};
\draw (202,70.4) node [anchor=north west][inner sep=0.75pt]  [color={rgb, 255:red, 245; green, 166; blue, 35 }  ,opacity=1 ]  {$S_{4}$};

\draw    (71,80) -- (99,80) ;
\draw [shift={(101,80)}, rotate = 180] [color={rgb, 255:red, 0; green, 0; blue, 0 }  ][line width=0.75]    (10.93,-3.29) .. controls (6.95,-1.4) and (3.31,-0.3) .. (0,0) .. controls (3.31,0.3) and (6.95,1.4) .. (10.93,3.29)   ;
\draw    (120,80) -- (148,80) ;
\draw [shift={(150,80)}, rotate = 180] [color={rgb, 255:red, 0; green, 0; blue, 0 }  ][line width=0.75]    (10.93,-3.29) .. controls (6.95,-1.4) and (3.31,-0.3) .. (0,0) .. controls (3.31,0.3) and (6.95,1.4) .. (10.93,3.29)   ;
\draw    (170,80) -- (198,80) ;
\draw [shift={(200,80)}, rotate = 180] [color={rgb, 255:red, 0; green, 0; blue, 0 }  ][line width=0.75]    (10.93,-3.29) .. controls (6.95,-1.4) and (3.31,-0.3) .. (0,0) .. controls (3.31,0.3) and (6.95,1.4) .. (10.93,3.29)   ;

\draw    (534,80.75) ;
\draw [shift={(534,80.75)}, rotate = 0] [color={rgb, 255:red, 0; green, 0; blue, 0 }  ][fill={rgb, 255:red, 0; green, 0; blue, 0 }  ][line width=0.75]      (0, 0) circle [x radius= 3.35, y radius= 3.35]   ;
\draw    (476.98,134) ;
\draw [shift={(476.98,134)}, rotate = 0] [color={rgb, 255:red, 0; green, 0; blue, 0 }  ][fill={rgb, 255:red, 0; green, 0; blue, 0 }  ][line width=0.75]      (0, 0) circle [x radius= 3.35, y radius= 3.35]   ;
\draw    (476.98,59.45) ;
\draw [shift={(476.98,59.45)}, rotate = 0] [color={rgb, 255:red, 0; green, 0; blue, 0 }  ][fill={rgb, 255:red, 0; green, 0; blue, 0 }  ][line width=0.75]      (0, 0) circle [x radius= 3.35, y radius= 3.35]   ;
\draw    (419.96,80.75) ;
\draw [shift={(419.96,80.75)}, rotate = 0] [color={rgb, 255:red, 0; green, 0; blue, 0 }  ][fill={rgb, 255:red, 0; green, 0; blue, 0 }  ][line width=0.75]      (0, 0) circle [x radius= 3.35, y radius= 3.35]   ;
\end{tikzpicture}

\end{tabular}
\end{center}

\noindent
where all the arrows are extensions (with degree 1). So, the composition of any two arrows is zero. The arrow $S_2\to S_3$ is frozen since $S_2,S_3$ are regular modules. This exceptional collection can be twisted to give:

\begin{center}
\begin{tabular}{c}

\tikzset{every picture/.style={line width=0.75pt}} 

\begin{tikzpicture}[x=0.75pt,y=0.75pt,yscale=-1,xscale=1]

\draw   (451.91,80.75) .. controls (451.91,68.99) and (463.13,59.45) .. (476.98,59.45) .. controls (490.83,59.45) and (502.05,68.99) .. (502.05,80.75) .. controls (502.05,92.51) and (490.83,102.05) .. (476.98,102.05) .. controls (463.13,102.05) and (451.91,92.51) .. (451.91,80.75)(419.96,80.75) .. controls (419.96,51.34) and (445.49,27.5) .. (476.98,27.5) .. controls (508.47,27.5) and (534,51.34) .. (534,80.75) .. controls (534,110.16) and (508.47,134) .. (476.98,134) .. controls (445.49,134) and (419.96,110.16) .. (419.96,80.75) ;
\draw    (476.98,59.45) .. controls (399,40) and (455,185) .. (534,80.75) ;
\draw [color={rgb, 255:red, 245; green, 166; blue, 35 }  ,draw opacity=1 ]   (419.96,80.75) .. controls (439.12,50.87) and (460.91,40.31) .. (486,41) .. controls (511.09,41.69) and (535.18,80.33) .. (496.09,106.17) .. controls (457,132) and (415,61) .. (476.98,59.45) ;
\draw [color={rgb, 255:red, 208; green, 2; blue, 27 }  ,draw opacity=1 ]   (419.96,80.75) .. controls (431.79,107.22) and (447.12,123.54) .. (476.98,134) ;
\draw [color={rgb, 255:red, 22; green, 48; blue, 226 }  ,draw opacity=1 ]   (476.98,134) .. controls (510.76,116.01) and (520.19,110.98) .. (534,80.75) ;
\draw    (71,80) -- (99,80) ;
\draw [shift={(101,80)}, rotate = 180] [color={rgb, 255:red, 0; green, 0; blue, 0 }  ][line width=0.75]    (10.93,-3.29) .. controls (6.95,-1.4) and (3.31,-0.3) .. (0,0) .. controls (3.31,0.3) and (6.95,1.4) .. (10.93,3.29)   ;
\draw    (120,80) -- (148,80) ;
\draw [shift={(150,80)}, rotate = 180] [color={rgb, 255:red, 0; green, 0; blue, 0 }  ][line width=0.75]    (10.93,-3.29) .. controls (6.95,-1.4) and (3.31,-0.3) .. (0,0) .. controls (3.31,0.3) and (6.95,1.4) .. (10.93,3.29)   ;
\draw    (170,80) -- (198,80) ;
\draw [shift={(200,80)}, rotate = 180] [color={rgb, 255:red, 0; green, 0; blue, 0 }  ][line width=0.75]    (10.93,-3.29) .. controls (6.95,-1.4) and (3.31,-0.3) .. (0,0) .. controls (3.31,0.3) and (6.95,1.4) .. (10.93,3.29)   ;
\draw    (64,70) .. controls (107.56,44.26) and (159.94,45.96) .. (200.77,70.26) ;
\draw [shift={(202,71)}, rotate = 211.37] [color={rgb, 255:red, 0; green, 0; blue, 0 }  ][line width=0.75]    (10.93,-3.29) .. controls (6.95,-1.4) and (3.31,-0.3) .. (0,0) .. controls (3.31,0.3) and (6.95,1.4) .. (10.93,3.29)   ;

\draw (51,70.4) node [anchor=north west][inner sep=0.75pt]    {$P_{2}$};
\draw (102,70.4) node [anchor=north west][inner sep=0.75pt]  [color={rgb, 255:red, 22; green, 48; blue, 226 }  ,opacity=1 ]  {$S_{2}$};
\draw (152,70.4) node [anchor=north west][inner sep=0.75pt]  [color={rgb, 255:red, 208; green, 2; blue, 27 }  ,opacity=1 ]  {$S_{3}$};
\draw (202,70.4) node [anchor=north west][inner sep=0.75pt]  [color={rgb, 255:red, 245; green, 166; blue, 35 }  ,opacity=1 ]  {$P_{1}$};

\draw    (534,80.75) ;
\draw [shift={(534,80.75)}, rotate = 0] [color={rgb, 255:red, 0; green, 0; blue, 0 }  ][fill={rgb, 255:red, 0; green, 0; blue, 0 }  ][line width=0.75]      (0, 0) circle [x radius= 3.35, y radius= 3.35]   ;
\draw    (476.98,134) ;
\draw [shift={(476.98,134)}, rotate = 0] [color={rgb, 255:red, 0; green, 0; blue, 0 }  ][fill={rgb, 255:red, 0; green, 0; blue, 0 }  ][line width=0.75]      (0, 0) circle [x radius= 3.35, y radius= 3.35]   ;
\draw    (476.98,59.45) ;
\draw [shift={(476.98,59.45)}, rotate = 0] [color={rgb, 255:red, 0; green, 0; blue, 0 }  ][fill={rgb, 255:red, 0; green, 0; blue, 0 }  ][line width=0.75]      (0, 0) circle [x radius= 3.35, y radius= 3.35]   ;
\draw    (419.96,80.75) ;
\draw [shift={(419.96,80.75)}, rotate = 0] [color={rgb, 255:red, 0; green, 0; blue, 0 }  ][fill={rgb, 255:red, 0; green, 0; blue, 0 }  ][line width=0.75]      (0, 0) circle [x radius= 3.35, y radius= 3.35]   ;

\end{tikzpicture}

\end{tabular}
\end{center}
where $P_2\to S_2$ and $P_2\to P_1$ are {\color{black} morphisms}, $S_3\to P_1$ is an extension as is the arrow $S_2\to S_3$, {\color{black} which} is still frozen, and the composition $S_2\to S_3\to I_3$ is zero in Ext$(S_2,I_3)$ since $I_3$ is injective. Thus, these Hom-Ext quivers are twist equivalent as superquivers with frozen arrows.
\end{exmp}

}


\nocite{*}
\bibliographystyle{amsalpha}
\bibliography{bibliography}

\end{document}